\pdfoutput=1
\documentclass[12pt,reqno]{amsart}
\usepackage[utf8]{inputenc}
\usepackage[english]{babel}
\usepackage{amsmath,amssymb, mathtools, MnSymbol}
\usepackage{tikz}
\usetikzlibrary{decorations.pathreplacing, calc,intersections,fit, matrix, shapes.geometric, patterns,cd}
\usepackage{tkz-euclide}
\usetkzobj{all}
\usepackage{IEEEtrantools}
\usepackage{mleftright}
\usepackage{wrapfig}
\usepackage{csquotes}
\usepackage{multicol}
\usepackage{cutwin}
\usepackage{xfrac}
\usepackage[all]{xy}
\usepackage{mathrsfs}
\usepackage{chngcntr}
\usepackage{tabularx}
\usepackage{longtable}

\usepackage[backend=biber,style=alphabetic, maxnames=5]{biblatex}
\renewbibmacro{in:}{\ifentrytype{article}{}{\printtext{\bibstring{in}\intitlepunct}}}
\usepackage{enumitem} 
\setlist[enumerate]{nosep, label=(\arabic*)}

\usepackage{hyperref}
\hypersetup{
    pdftitle={Non-Hyperoctahedral Categories of Two-Colored Partitions},
    pdfauthor={Alexander Mang, Moritz Weber},
    pdfsubject={Non-Hyperoctahedral Categories of Two-Colored Partitions},
    pdfkeywords={quantum groups, unitary easy quantum groups, unitary group,  tensor category, two-colored partitions, partitions of sets, categories of partitions}
}
\usepackage{xpatch}
\xpatchcmd{\proof}{\itshape}{\normalfont\proofnamefont}{}{}
\xpatchcmd{\paragraph}{\normalfont}{{\normalfont\itshape}}{}{}

\makeatletter
\@namedef{r@tocindent4}{0pt}
\@namedef{r@tocindent5}{0pt} 
\makeatother
\newcommand\Tstrut{\rule{0pt}{2.6ex}}

\makeatletter
\let\save@mathaccent\mathaccent
\newcommand*\if@single[3]{%
  \setbox0\hbox{${\mathaccent"0362{#1}}^H$}%
  \setbox2\hbox{${\mathaccent"0362{\kern0pt#1}}^H$}%
  \ifdim\ht0=\ht2 #3\else #2\fi
  }
\newcommand*\rel@kern[1]{\kern#1\dimexpr\macc@kerna}
\newcommand*\widebar[1]{\@ifnextchar^{{\wide@bar{#1}{0}}}{\wide@bar{#1}{1}}}
\newcommand*\wide@bar[2]{\if@single{#1}{\wide@bar@{#1}{#2}{1}}{\wide@bar@{#1}{#2}{2}}}
\newcommand*\wide@bar@[3]{%
  \begingroup
  \def\mathaccent##1##2{%
    \let\mathaccent\save@mathaccent
    \if#32 \let\macc@nucleus\first@char \fi
    \setbox\z@\hbox{$\macc@style{\macc@nucleus}_{}$}%
    \setbox\tw@\hbox{$\macc@style{\macc@nucleus}{}_{}$}%
    \dimen@\wd\tw@
    \advance\dimen@-\wd\z@
    \divide\dimen@ 3
    \@tempdima\wd\tw@
    \advance\@tempdima-\scriptspace
    \divide\@tempdima 10
    \advance\dimen@-\@tempdima
    \ifdim\dimen@>\z@ \dimen@0pt\fi
    \rel@kern{0.6}\kern-\dimen@
    \if#31
      \overline{\rel@kern{-0.6}\kern\dimen@\macc@nucleus\rel@kern{0.4}\kern\dimen@}%
      \advance\dimen@0.4\dimexpr\macc@kerna
      \let\final@kern#2%
      \ifdim\dimen@<\z@ \let\final@kern1\fi
      \if\final@kern1 \kern-\dimen@\fi
    \else
      \overline{\rel@kern{-0.6}\kern\dimen@#1}%
    \fi
  }%
  \macc@depth\@ne
  \let\math@bgroup\@empty \let\math@egroup\macc@set@skewchar
  \mathsurround\z@ \frozen@everymath{\mathgroup\macc@group\relax}%
  \macc@set@skewchar\relax
  \let\mathaccentV\macc@nested@a
  \if#31
    \macc@nested@a\relax111{#1}%
  \else
    \def\gobble@till@marker##1\endmarker{}%
    \futurelet\first@char\gobble@till@marker#1\endmarker
    \ifcat\noexpand\first@char A\else
      \def\first@char{}%
    \fi
    \macc@nested@a\relax111{\first@char}%
  \fi
  \endgroup
}
\makeatother

\newenvironment{mycenter}[1][\topsep]
  {\setlength{\topsep}{#1}\par\kern\topsep\centering}
  {\par\kern\topsep}

\newcommand{\proofnamefont}{\scshape}

\newtheoremstyle{break}
  {}
  {}
  {\itshape}
  {\parindent}
  {\scshape}
  {.}
  {\newline}
  {}
\theoremstyle{break}

\theoremstyle{plain}

\newtheorem{theorem}{Theorem}[section]

\newtheorem{lemma}[theorem]{Lemma}

\newtheorem*{umain}{Main Theorem}

\theoremstyle{definition}
\newtheorem{definition}[theorem]{Definition}
\newtheorem{remark}[theorem]{Remark}
\newtheorem{notation}[theorem]{Notation}

\newcommand{\mr}{\mathrm}
\newcommand{\mc}{\mathcal}

\newcommand{\eqpd}{\coloneqq}
\newcommand{\upp}[1]{\prescript{\filledsquare}{}{#1}}
\newcommand{\lop}[1]{\prescript{}{\filledsquare}{#1}}

\newcommand{\pint}{\mathbb{N}}
\newcommand{\integers}{\mathbb{Z}}
\newcommand{\Cp}{\mc{P}^{\circ\bullet}}
\newcommand{\Cpp}{\mc{P}^{\circ\bullet}_{2}}
\newcommand{\Cppnb}{\mc{P}^{\circ\bullet}_{2,\mathrm{nb}}}

\newcommand{\colors}{\{\circ,\bullet\}}
\newcommand{\toco}{\Sigma}

\newcommand{\nhoc}{\mathsf{PCat}^{\circ\bullet}_{\mathrm{NHO}}}
\newcommand{\cotcp}{\mathsf{PCat}^{\circ\bullet}}
\newcommand{\dwi}[1]{\lsem{#1}\rsem}

\newcommand{\tightoverset}[2]{%
  \mathop{#2}\limits^{\vbox to -.5ex{\kern-0.75ex\hbox{$#1$}\vss}}}
\newcommand{\notastightoverset}[2]{%
  \mathop{#2}\limits^{\vbox to -.0ex{\kern-0.75ex\hbox{$#1$}\vss}}}

\newcommand{\PartCrossWW}{%
  \begin{tikzpicture}[scale=0.25,baseline=0.015cm]
    \def\xdist{1}
    \def\ydist{1}
    \node [scale=0.33, circle, draw=black, fill=white] (a1) at ({0*\xdist},{0*\ydist}) {};
    \node [scale=0.33, circle, draw=black, fill=white] (a2) at ({1*\xdist},{0*\ydist}) {};
    \node [scale=0.33, circle, draw=black, fill=white] (b1) at ({0*\xdist},{1*\ydist}) {};
    \node [scale=0.33, circle, draw=black, fill=white] (b2) at ({1*\xdist},{1*\ydist}) {};
    \draw (a1) to (b2);    
    \draw (a2) to (b1);
  \end{tikzpicture}
  }

  \newcommand{\PartIdenB}{%
  \begin{tikzpicture}[scale=0.25,baseline=0.015cm]
    \def\xdist{1}
    \def\ydist{1}
    \node [scale=0.33, circle, draw=black, fill=black] (a1) at ({0*\xdist},{0*\ydist}) {};
    \node [scale=0.33, circle, draw=black, fill=black] (b1) at ({0*\xdist},{1*\ydist}) {};
    \draw (a1) to (b1);    
  \end{tikzpicture}
  }

  \newcommand{\PartIdenW}{%
  \begin{tikzpicture}[scale=0.25,baseline=0.015cm]
    \def\xdist{1}
    \def\ydist{1}
    \node [scale=0.33, circle, draw=black, fill=white] (a1) at ({0*\xdist},{0*\ydist}) {};
    \node [scale=0.33, circle, draw=black, fill=white] (b1) at ({0*\xdist},{1*\ydist}) {};
    \draw (a1) to (b1);    
  \end{tikzpicture}
  }

\newcommand{\PartIdenLoBW}{%
  \begin{tikzpicture}[scale=0.25,baseline=-0.015cm]
    \def\xdist{1}
    \def\ydist{1}
    \def\ypar{1}
    \node [scale=0.33, circle, draw=black, fill=black] (a1) at ({0*\xdist},{0*\ydist}) {};
    \node [scale=0.33, circle, draw=black, fill=white] (a2) at ({1*\xdist},{0*\ydist}) {};
    \draw (a1) -- ++ (0,{\ypar}) -| (a2);    
  \end{tikzpicture}
  }
\newcommand{\PartIdenLoWB}{%
  \begin{tikzpicture}[scale=0.25,baseline=-0.015cm]
    \def\xdist{1}
    \def\ydist{1}
    \def\ypar{1}
    \node [scale=0.33, circle, draw=black, fill=white] (a1) at ({0*\xdist},{0*\ydist}) {};
    \node [scale=0.33, circle, draw=black, fill=black] (a2) at ({1*\xdist},{0*\ydist}) {};
    \draw (a1) -- ++ (0,{\ypar}) -| (a2);    
  \end{tikzpicture}
  }

  \newcommand{\PartGlobColWWBBTensor}{%
  \begin{tikzpicture}[scale=0.25,baseline=-0.015cm]
    \def\xdist{1}
    \def\ydist{1}
    \def\ypar{1}
    \node [scale=0.33, circle, draw=black, fill=white] (a1) at ({0*\xdist},{0*\ydist}) {};
    \node [scale=0.33, circle, draw=black, fill=white] (a2) at ({1*\xdist},{0*\ydist}) {};
    \node [scale=0.33, circle, draw=black, fill=black] (a3) at ({2.25*\xdist},{0*\ydist}) {};
    \node [scale=0.33, circle, draw=black, fill=black] (a4) at ({3.25*\xdist},{0*\ydist}) {};
    \draw (a1) -- ++ (0,{\ypar}) -| (a2);
    \draw (a3) -- ++ (0,{\ypar}) -| (a4);
    \node [ scale=0.666] at ($(a2)+({0.625*\xdist},{0.325*\ydist})$) {$\otimes$};
  \end{tikzpicture}
}

\newcommand{\PartFourWBWB}{%
  \begin{tikzpicture}[scale=0.25,baseline=-0.015cm]
    \def\xdist{0.666}
    \def\ydist{1}
    \def\ypar{1}
    \node [scale=0.33, circle, draw=black, fill=white] (a1) at ({0*\xdist},{0*\ydist}) {};
    \node [scale=0.33, circle, draw=black, fill=black] (a2) at ({1*\xdist},{0*\ydist}) {};
    \node [scale=0.33, circle, draw=black, fill=white] (a3) at ({2*\xdist},{0*\ydist}) {};
    \node [scale=0.33, circle, draw=black, fill=black] (a4) at ({3*\xdist},{0*\ydist}) {};
    \draw (a1) -- ++ (0,{\ypar}) -| (a4);
    \draw (a2) -- ++ (0,{\ypar});
    \draw (a3) -- ++ (0,{\ypar});    
  \end{tikzpicture}
}

\newcommand{\PartSinglesWB}{%
  \begin{tikzpicture}[scale=0.25,baseline=-0.015cm]
    \def\xdist{1}
    \def\ydist{1}
    \def\ypar{1}
    \node [scale=0.33, circle, draw=black, fill=white] (a1) at ({0*\xdist},{0*\ydist}) {};
    \node [scale=0.33, circle, draw=black, fill=black] (a2) at ({1*\xdist},{0*\ydist}) {};
    \draw[->] (a1) -- ++ (0,{\ypar});
    \draw[->] (a2) -- ++ (0,{\ypar});
  \end{tikzpicture}
}
\newcommand{\PartSinglesWBTensor}{%
  \begin{tikzpicture}[scale=0.25,baseline=-0.015cm]
    \def\xdist{1}
    \def\ydist{1}
    \def\ypar{1}
    \node [scale=0.33, circle, draw=black, fill=white] (a1) at ({0*\xdist},{0*\ydist}) {};
    \node [scale=0.33, circle, draw=black, fill=black] (a2) at ({1.5*\xdist},{0*\ydist}) {};
    \draw[->] (a1) -- ++ (0,{\ypar});
    \draw[->] (a2) -- ++ (0,{\ypar});
    \node [scale=0.66] at ($(a1)+({0.75*\xdist},{0.325*\ydist})$) {$\otimes$};
  \end{tikzpicture}
}

\newcommand{\UCPartFour}{%
  \begin{tikzpicture}[baseline=0.0cm]
    \def\xdist{0.2cm}
    \def\ydist{0.32cm}
    \def\ypar{0.2cm}
    \coordinate (a1) at ({0*\xdist},{0*\ydist});
    \coordinate (a2) at ({1*\xdist},{0*\ydist});
    \coordinate (a3) at ({2*\xdist},{0*\ydist});
    \coordinate (a4) at ({3*\xdist},{0*\ydist});
    \draw (a1) -- ++ (0,{\ypar}) -| (a4);
    \draw (a2) -- ++ (0,{\ypar});
    \draw (a3) -- ++ (0,{\ypar});    
  \end{tikzpicture}
}

\newcommand{\UCPartSinglesTensor}{%
  \begin{tikzpicture}[baseline=0.0cm]
    \def\xdist{0.2cm}
    \def\ydist{0.32cm}
    \def\ypar{0.25cm}
    \coordinate (a1) at ({0*\xdist},{0*\ydist});
    \coordinate (a2) at ({1.75*\xdist},{0*\ydist});    
    \draw[->] (a1) -- ++ (0,{\ypar});
    \draw[->] (a2) -- ++ (0,{\ypar});
    \node [scale=0.66] at ($(a1)+({0.5*1.75*\xdist},{0.325*\ydist})$) {$\otimes$};
  \end{tikzpicture}
}

\begin{document}
\title[Non-Hyperoctahedral Categories, Part~I]{Non-Hyperoctahedral Categories \\of Two-Colored Partitions\\ Part~I: New Categories} 
\author{Alexander Mang}
\author{Moritz Weber}
\address{Saarland University, Fachbereich Mathematik, 
	66041 Saarbrücken, Germany}
\email{s9almang@stud.uni-saarland.de, weber@math.uni-sb.de}
\thanks{The second author was supported by the ERC Advanced Grant NCDFP, held
by Roland Spei\-cher, by the SFB-TRR 195, and by the DFG project \emph{Quantenautomorphismen von Graphen}. This work was part of the first author's Master's thesis.}

\date{\today}

\begin{abstract}
Compact quantum groups can be studied by investigating their co-re\-pre\-sen\-ta\-tion ca\-te\-go\-ries in analogy to the Schur-Weyl/Tannaka-Krein  approach. For the special class of (unitary) \enquote{easy} quantum groups these categories arise from a combinatorial structure: Rows of two-colored points form the objects, partitions of two such rows the morphisms; vertical/horizontal concatenation and reflection give composition, monoidal product and involution. Of the four possible classes $\mathcal O$, $\mathcal B$, $\mathcal S$ and $\mathcal H$ of such categories (inspired respectively by the classical orthogonal, bistochastic, symmetric and hyperoctahedral groups) we treat the first three -- the non-hyperoctahedral ones. We introduce many new examples of such categories. They are defined in terms of subtle combinations of block size, coloring and non-crossing conditions.  This article is part of an effort to classify all non-hyperoctahedral categories of two-colored partitions. The article is purely combinatorial in nature; The quantum group aspects are left out.
\end{abstract}
\maketitle
\section{Introduction} 
\label{section:introduction}

In Woronowicz's approach (\cite{Wo87}, \cite{Wo88}, \cite{Wo98}), (compact) \enquote{quantum groups} are understood as certain non-commutative spaces, the formal duals of $C^*$-algebras, carrying a special Hopf algebra structure, for which a non-commutative version of Pontryagin duality can be proven. Reminiscent of the theorems of Tannaka-Krein and Schur-Weyl, a duality exists between the class of compact quantum groups and a particular class of involutive monoidal linear categories. The finite-dimensional unitary co-representations of a given compact quantum group form such a category, and, conversely, a unique maximal compact quantum group can be reconstructed from any such tensor category (\cite{Wo88}). 
\par
Banica and Speicher (\cite{BaSp09}) showed that sets of points as objects and partitions of finite sets as morphisms with vertical concatenation as composition, horizontal concatenation as monoidal product and reflection as involution provide concrete, combinatorial initial data for such representation categories. Their construction yields compact quantum subgroups of the free orthogonal quantum group $O_n^+$ introduced by Wang (\cite{Wa95}) as a non-commutative counterpart of the classical group  $O_n$ of orthogonal real-valued matrices. All examples of compact quantum groups arising in this fashion, the so-called \enquote{easy} quantum groups, have since been classified (\cite{BaSp09}, \cite{BaCuSp09a}, \cite{We12},\cite{RaWe15a}, and \cite{RaWe13}).
\par
As Freslon, Tarrago and the second author demonstrated (\cite{FrWe14}, \cite{TaWe15a}, \cite{TaWe15b}), Banica and Speicher's approach can be generalized to categories of partitions of sets of \emph{two-colored} points.  In contrast to the uncolored case, here, vertical concatenation of partitions, i.e.\ the composition of morphisms, is restricted to such partitions with matching colorings of their points. This construction yields combinatorial compact quantum subgroups of the free unitary quantum group $U_n^+$, a quantum analogue of the classical unitary group $U_n$ also introduced by Wang (\cite{Wa95}). A collective endeavour to find all such \enquote{unitary easy} quantum groups was initiated by Tarrago and the second author in \cite{TaWe15a} and has since been advanced by Gromada in \cite{Gr18} as well as by the authors in \cite{MaWe18a} and \cite{MaWe18b}.
\par
The classification of all unitary easy quantum groups has been approached from several different angles of attack.
Tarrago and the second author classified in \cite{TaWe15a} all \emph{non-crossing} categories $ \mc C\subseteq \Cp$ of two-colored partitions, i.e., $\mc C\subseteq \mc{NC}^{\circ\bullet}$, and all categories $\mc C$ of two-colored partitions with $\PartCrossWW\in \mc C$, the so-called \emph{group case}. In contrast, in \cite{Gr18} Gromada determined all categories $\mc C$ with the property of being \emph{globally colorized}, meaning $\PartGlobColWWBBTensor\in \mc C$. Lastly, the authors of the present article found all categories $\mc C$ with $\langle \emptyset\rangle\subseteq\mc C \subseteq \langle \PartCrossWW\rangle$, i.e.\ categories of \emph{neutral pair partitions}, corresponding to easy compact quantum groups $G$ with $U_n^+\supseteq G\supseteq U_n$, the \enquote{unitary half-liberations}, in \cite{MaWe18a} and \cite{MaWe18b}.  
\par
The present article is concerned with \emph{non-hy\-per\-oc\-ta\-hed\-ral} categories of two-colored partitions, i.e.\ categories $\mc C\subseteq \Cp$ with $\PartSinglesWB\in \mc C$ or $\PartFourWBWB\notin \mc C$. We define explicitly certain sets of partitions and show that each of them constitutes a non-hyperoctahedral category. See the next section for an overview.
\par
This article is part of an effort to classify all non-hyperoctahedral categories of two-colored partitions. In subsequent articles it will be shown that the categories found in the present article are  pairwise distinct and actually constitute \emph{all} possible non-hyperoctahedral categories. Furthermore, a set of generating partitions for each non-hyperoctahedral category will be determined.
\par
About \emph{hyperoctahedral} categories of two-colored partitions very little is known for the moment. Note that in the uncolored case categories $\mc C\subseteq \mc{P}$ with $\UCPartSinglesTensor\notin \mc C$ and $\UCPartFour\in \mc C$ give rise to quantum subgroups of the free hyperoctahedral quantum group~$H^+_n$ (\cite{Bi01}), hence the name. Amongst others, Laura Maaßen is currently doing research on the hyperoctahedral case.
\pagebreak
\section{Main Result}

Many new examples of non-hyperoctahedral categories of two-colored partitions are provided. Roughly, they are determined by combinations of constraints on the
\begin{enumerate}[label=(\roman*)]
\item sizes of blocks,
\item coloring of the points,
\item allowed crossings between blocks
\end{enumerate}
of their partitions.
\par
More precisely: The coloring of any two-colored partition $p\in \Cp$ induces on the set of points a measure-like structure, the \emph{color sum}~$\sigma_p$, and a metric-like one, the \emph{color distance}~$\delta_p$. Measuring the set of all points yields the \emph{total color sum}~$\toco(p)$.
\par
Let now $\mathcal S\subseteq \Cp$ be an arbitrary set of two-colored partitions and consider the following data:
\begin{enumerate}[label=(\arabic*)]
\item The set of block sizes: \[F(\mc S)\eqpd \{\, |B| \mid p\in \mc S,\, B\text{ block of } p\}.\]
\item The set of block color sums: \[V(\mc S)\eqpd  \{\,\sigma_p(B)\mid p\in \mc S,\, B\text{ block of }p\}.\]
\item The set of total color sums: \[\toco(\mc S)\eqpd \{\,\toco(p)\mid p\in \mc S\}.\]
\item The set of color distances between any two subsequent legs of the same block having the \emph{same} normalized color:
  \begin{IEEEeqnarray*}{rCl}
              L(\mc S)\eqpd \{\,\delta_p(\alpha_1,\alpha_2)&\mid& p\in \mc S, \, B\text{ block of }p,\, \alpha_1,\alpha_2\in B,\, \alpha_1\neq \alpha_2,\\
 & & ]\alpha_1,\alpha_2[_p\cap B=\emptyset,\, \sigma_p(\{\alpha_1,\alpha_2\})\neq 0\}.
\end{IEEEeqnarray*}
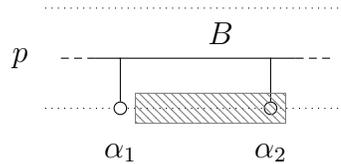
\begin{mycenter}
       \begin{tikzpicture}
    \def\scp{0.666}
    \def\linksize{\scp*0.075cm}
    \def\pointsize{\scp*0.25cm}
    \def\dd{\scp*0.5cm}
    \def\dx{\scp*1cm}
    \def\cx{\scp*0.3cm}
    \def\txu{5*\dx}    
    \def\txl{5*\dx}
    \def\dy{\scp*1cm}
    \def\cy{\scp*0.3cm}
    \def\ty{2*\dy}
    \tikzset{whp/.style={circle, inner sep=0pt, text width={\pointsize}, draw=black, fill=white}}
    \tikzset{blp/.style={circle, inner sep=0pt, text width={\pointsize}, draw=black, fill=black}}
    \tikzset{lk/.style={regular polygon, regular polygon sides=4, inner sep=0pt, text width={\linksize}, draw=black, fill=black}}
    \draw[dotted] ({0-\dd},{0}) -- ({\txl+\dd},{0});
    \draw[dotted] ({0-\dd},{\ty}) -- ({\txl+\dd},{\ty});
    \node [whp] (l1) at ({0+1*\dx},{0+0*\ty}) {};
    \node [whp] (l2) at ({0+4*\dx},{0+0*\ty}) {};    
    \draw (l1) --++(0,{\dy}) -| (l2) ($(l1)+(0,{\dy})$) --++({-0.5*\dx},0) ($(l2)+(0,{\dy})$) --++({0.5*\dx},0);
    \draw[densely dashed] ($(l1)+({-0.5*\dx},{\dy})$) -- ++ ({-0.75*\dx},0) ($(l2)+({0.5*\dx},{\dy})$) -- ++ ({0.75*\dx},0);
    \draw [draw=gray, pattern= north west lines, pattern color = gray] ($(l1)+({\cx},{\cy})$) -- ($(l1)+({\cx},{-\cy})$) --($(l2)+({\cx},{-\cy})$) |- cycle;
    \node at ({-\dx},{0.5*\ty}) {$p$};
    \node at ({3*\dx},{0.75*\ty}) {$B$};    
    \node [below = {0.35*\dy} of l1] {$\alpha_1$};
    \node [below = {0.35*\dy} of l2] {$\alpha_2$};    
  \end{tikzpicture}
\end{mycenter}
\item  The set of color distances between any two subsequent legs of the same block having \emph{different} normalized colors:
   \begin{IEEEeqnarray*}{rCl}
              K(\mc S)\eqpd \{\,\delta_p(\alpha_1,\alpha_2)&\mid& p\in \mc S, \, B\text{ block of }p,\, \alpha_1,\alpha_2\in B,\, \alpha_1\neq \alpha_2,\\
 & & ]\alpha_1,\alpha_2[_p\cap B=\emptyset,\, \sigma_p(\{\alpha_1,\alpha_2\})= 0\}.
\end{IEEEeqnarray*}
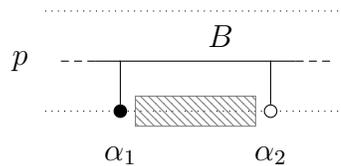
\begin{mycenter}
       \begin{tikzpicture}
    \def\scp{0.666}
    \def\linksize{\scp*0.075cm}
    \def\pointsize{\scp*0.25cm}
    \def\dd{\scp*0.5cm}
    \def\dx{\scp*1cm}
    \def\cx{\scp*0.3cm}
    \def\txu{5*\dx}    
    \def\txl{5*\dx}
    \def\dy{\scp*1cm}
    \def\cy{\scp*0.3cm}
    \def\ty{2*\dy}
    \tikzset{whp/.style={circle, inner sep=0pt, text width={\pointsize}, draw=black, fill=white}}
    \tikzset{blp/.style={circle, inner sep=0pt, text width={\pointsize}, draw=black, fill=black}}
    \tikzset{lk/.style={regular polygon, regular polygon sides=4, inner sep=0pt, text width={\linksize}, draw=black, fill=black}}
    \draw[dotted] ({0-\dd},{0}) -- ({\txl+\dd},{0});
    \draw[dotted] ({0-\dd},{\ty}) -- ({\txl+\dd},{\ty});
    \node [blp] (l1) at ({0+1*\dx},{0+0*\ty}) {};
    \node [whp] (l2) at ({0+4*\dx},{0+0*\ty}) {};    
    \draw (l1) --++(0,{\dy}) -| (l2) ($(l1)+(0,{\dy})$) --++({-0.5*\dx},0) ($(l2)+(0,{\dy})$) --++({0.5*\dx},0);
    \draw[densely dashed] ($(l1)+({-0.5*\dx},{\dy})$) -- ++ ({-0.75*\dx},0) ($(l2)+({0.5*\dx},{\dy})$) -- ++ ({0.75*\dx},0);
    \draw [draw=gray, pattern= north west lines, pattern color = gray] ($(l1)+({\cx},{\cy})$) -- ($(l1)+({\cx},{-\cy})$) --($(l2)+({-\cx},{-\cy})$) |- cycle;
    \node at ({-\dx},{0.5*\ty}) {$p$};
    \node at ({3*\dx},{0.75*\ty}) {$B$};    
    \node [below = {0.35*\dy} of l1] {$\alpha_1$};
    \node [below = {0.35*\dy} of l2] {$\alpha_2$};    
  \end{tikzpicture}
\end{mycenter}
\item The set of color distances between any two legs belonging to two crossing blocks
   \begin{IEEEeqnarray*}{rCl}
                          X(\mc S)\eqpd \{\,\delta_p(\alpha_1,\alpha_2)& \mid& p\in \mc S,\, B_1,B_2\text{ blocks of }p, \, B_1\text{ crosses } B_2,\\
& &  \alpha_1\in B_1,\,\alpha_2\in B_2\}.
\end{IEEEeqnarray*}
\begin{mycenter}
       \begin{tikzpicture}
    \def\scp{0.666}
    \def\linksize{\scp*0.075cm}
    \def\pointsize{\scp*0.25cm}
    \def\dd{\scp*0.5cm}
    \def\dx{\scp*1cm}
    \def\cx{\scp*0.3cm}
    \def\txu{8*\dx}    
    \def\txl{8*\dx}
    \def\dy{\scp*1cm}
    \def\cy{\scp*0.3cm}
    \def\ty{3*\dy}
    \tikzset{whp/.style={circle, inner sep=0pt, text width={\pointsize}, draw=black, fill=white}}
    \tikzset{blp/.style={circle, inner sep=0pt, text width={\pointsize}, draw=black, fill=black}}
    \tikzset{lk/.style={regular polygon, regular polygon sides=4, inner sep=0pt, text width={\linksize}, draw=black, fill=black}}
    \draw[dotted] ({0-\dd},{0}) -- ({\txl+\dd},{0});
    \draw[dotted] ({0-\dd},{\ty}) -- ({\txl+\dd},{\ty});
    \node [whp] (l1) at ({0+1*\dx},{0+0*\ty}) {};
    \node [whp] (l2) at ({0+3*\dx},{0+0*\ty}) {};
    \node [blp] (l3) at ({0+6*\dx},{0+0*\ty}) {};
    \node [whp] (l4) at ({0+8*\dx},{0+0*\ty}) {};        
    \draw (l1) --++(0,{\dy}) -| (l3) ($(l1)+(0,{\dy})$) --++({-0.5*\dx},0) ($(l3)+(0,{\dy})$) --++({0.5*\dx},0);
    \draw[densely dashed] ($(l1)+({-0.5*\dx},{\dy})$) -- ++ ({-0.75*\dx},0) ($(l3)+({0.5*\dx},{\dy})$) -- ++ ({0.75*\dx},0);
    \draw (l2) --++(0,{2*\dy}) -| (l4) ($(l2)+(0,{2*\dy})$) --++({-0.5*\dx},0) ($(l4)+(0,{2*\dy})$) --++({0.5*\dx},0);
    \draw[densely dashed] ($(l2)+({-0.5*\dx},{2*\dy})$) -- ++ ({-0.75*\dx},0) ($(l4)+({0.5*\dx},{2*\dy})$) -- ++ ({0.75*\dx},0);
    \draw [draw=gray, pattern= north west lines, pattern color = gray] ($(l2)+({\cx},{\cy})$) -- ($(l2)+({\cx},{-\cy})$) --($(l3)+({-\cx},{-\cy})$) |- cycle;
    \node at ({-\dx},{0.5*\ty}) {$p$};
    \node [above = {0.85*\dy} of l1] {$B_1$};
    \node [above = {1.85*\dy} of l4] {$B_2$};        
    \node [below = {0.35*\dy} of l2] {$\alpha_1$};
    \node [below = {0.35*\dy} of l3] {$\alpha_2$};    
  \end{tikzpicture}
\end{mycenter}
\end{enumerate}
It is a subtle question which combinations of conditions on these six quantities define categories of partitions. We clarify it and we obtain a huge variety of new categories of partitions.
\vspace{0.5em}
\begin{umain}[{Theorem~\ref{theorem:main}}] For every $u\in \{0\}\cup \pint$, $m\in \pint$, $D\subseteq \{0,\ldots,\lfloor\frac{m}{2}\rfloor\}$, $E\subseteq \{0\}\cup\pint$ and every subsemigroup $N$ of $(\pint,+)$ each parameter tuple $(f,v,s,l,k,x)$ listed below defines a non-hyperoctahedral category of two-colored partitions
  \begin{IEEEeqnarray*}{rCCCl}
    \mc R_{(f,v,s,l,k,x)}\eqpd\{p\in \Cp&\mid &F(\{p\})\subseteq f,& V(\{p\})\subseteq v,& \toco(\{p\})\subseteq s,\\ && L(\{p\})\subseteq l,& K(\{p\})\subseteq k,& X(\{p\})\subseteq x \}.
  \end{IEEEeqnarray*}\vspace{-1.5em}
                    \begin{align*}
                      \begin{matrix}
                        f&v&s& l& k& x  \\ \hline \\[-0.85em]
                        \{2\} & \pm\{0, 2\} & 2um\integers & m\integers & m\integers & \integers\\
                      \{2\} & \pm\{0, 2\} & 2um\integers & m\!+\!2m\integers & 2m\integers & \integers \\
                        \{2\} & \pm \{0, 2\} & 2um\integers & m\!+\!2m\integers & 2m\integers & \integers\backslash m\integers \\
                      \{2\} & \{0\} & \{0\} & \emptyset & m\integers & \integers\\
                      \{2\} & \pm\{0, 2\} & \{0\} & \{0\} & \{0\} &  \integers\backslash {\pm N} \\
                      \{2\} & \{0\} & \{0\} & \emptyset & \{0\} & \integers\backslash {\pm N} \\
                      \{2\} & \{0\} & \{0\} & \emptyset & \{0\} & \integers\backslash \{0\}\backslash {\pm N} \\
                         \{1,2\}&\pm\{0, 1, 2\} & um\integers & m\integers & m\integers & \integers\backslash (\pm D\!+\!m\integers)\\
                                                 \{1,2\}&\pm\{0, 1, 2\} & 2um\integers & m\!+\!2m\integers & 2m\integers & \integers\backslash (\pm D\!+\!m\integers)\\
                         \{1,2\}&\pm \{0, 1\} & um\integers & \emptyset & m\integers & \integers\backslash (\pm D\!+\!m\integers) \\
                        \{1,2\}&\pm\{0, 1, 2\} & \{0\} & \{0\} & \{0\} & \integers\backslash {\pm E} \\
                      \{1,2\}&\pm\{0,  1\} & \{0\} & \emptyset & \{0\} & \integers\backslash {\pm E}\\
                        \pint & \integers & um\integers & m\integers & m\integers & \integers\backslash (\pm D\!+\!m\integers)\\
                        \pint & \integers & \{0\} & \{0\} & \{0\} & \integers\backslash {\pm E}\\
                    \end{matrix}
                    \end{align*}
                    Here,    $\pm S\eqpd S\cup(-S)$ for any set $S\subseteq \integers$.
                  \end{umain}
                  \pagebreak
\par

\section{Basic Definitions: Partitions}
\label{section:basic-definitions}
Details on, examples and illustrations of two-colored partitions and their categories can be found in \cite{TaWe15a}. However, we quickly recall the basics. In addition, certain definitions from \cite{MaWe18a} and \cite{MaWe18b} are given here in greater generality.
\par
After revisiting the fundamental definition of two-colored partitions, specialized language is introduced to describe them, in particular the concepts of \emph{orientation},   \emph{normalized color}, \emph{color sum} and \emph{color distance}.
\subsection{Two-Colored Partitions}
  A \emph{(two-colored) partition} is specified by the following three data: 1) two disjoint (possibly empty) finite totally ordered sets $R_L$, the \emph{lower row} comprising the \emph{lower points}, and $R_U$, the \emph{upper row} comprising the \emph{upper points}, 2) an exhaustive decomposition of the union $R_L\cup R_U$, the set of \emph{points}, into disjoint subsets, the \emph{blocks}, and 3) a two-valued mapping on $R_L\cup R_U$, the \emph{coloration}, assigning to every point its \emph{color}, either $\bullet$ (\emph{black}) or $\circ$ (\emph{white}). We represent partitions pictorially as follows:
  \begin{mycenter}[1em]
     \begin{tikzpicture}
    \def\scp{0.666}
    \def\linksize{\scp*0.075cm}
    \def\pointsize{\scp*0.25cm}
    \def\dd{\scp*0.5cm}
    \def\dx{\scp*1cm}
    \def\cx{\scp*0.3cm}
    \def\txu{0*\dx}    
    \def\txl{0*\dx}
    \def\dy{\scp*1cm}
    \def\cy{\scp*0.3cm}
    \def\ty{2*\dy}
    \tikzset{whp/.style={circle, inner sep=0pt, text width={\pointsize}, draw=black, fill=white}}
    \tikzset{blp/.style={circle, inner sep=0pt, text width={\pointsize}, draw=black, fill=black}}
    \tikzset{lk/.style={regular polygon, regular polygon sides=4, inner sep=0pt, text width={\linksize}, draw=black, fill=black}}
    \draw[dotted] ({0-\dd},{0}) -- ({\txl+\dd},{0});
    \draw[dotted] ({0-\dd},{\ty}) -- ({\txl+\dd},{\ty});
    \node [whp] (l1) at ({0+0*\dx},{0+0*\ty}) {};
    \draw[->] (l1) --++(0,{\dy});
  \end{tikzpicture},\quad
    \begin{tikzpicture}
    \def\scp{0.666}
    \def\linksize{\scp*0.075cm}
    \def\pointsize{\scp*0.25cm}
    \def\dd{\scp*0.5cm}
    \def\dx{\scp*1cm}
    \def\cx{\scp*0.3cm}
    \def\txu{0*\dx}    
    \def\txl{0*\dx}
    \def\dy{\scp*1cm}
    \def\cy{\scp*0.3cm}
    \def\ty{2*\dy}
    \tikzset{whp/.style={circle, inner sep=0pt, text width={\pointsize}, draw=black, fill=white}}
    \tikzset{blp/.style={circle, inner sep=0pt, text width={\pointsize}, draw=black, fill=black}}
    \tikzset{lk/.style={regular polygon, regular polygon sides=4, inner sep=0pt, text width={\linksize}, draw=black, fill=black}}
    \draw[dotted] ({0-\dd},{0}) -- ({\txl+\dd},{0});
    \draw[dotted] ({0-\dd},{\ty}) -- ({\txl+\dd},{\ty});
    \node [whp] (l1) at ({0+0*\dx},{0+0*\ty}) {};
    \node [whp] (u1) at ({0+0*\dx},{0+1*\ty}) {};
    \draw (l1) -- (u1);
  \end{tikzpicture},\quad
    \begin{tikzpicture}
    \def\scp{0.666}
    \def\linksize{\scp*0.075cm}
    \def\pointsize{\scp*0.25cm}
    \def\dd{\scp*0.5cm}
    \def\dx{\scp*1cm}
    \def\cx{\scp*0.3cm}
    \def\txu{0*\dx}    
    \def\txl{0*\dx}
    \def\dy{\scp*1cm}
    \def\cy{\scp*0.3cm}
    \def\ty{2*\dy}
    \tikzset{whp/.style={circle, inner sep=0pt, text width={\pointsize}, draw=black, fill=white}}
    \tikzset{blp/.style={circle, inner sep=0pt, text width={\pointsize}, draw=black, fill=black}}
    \tikzset{lk/.style={regular polygon, regular polygon sides=4, inner sep=0pt, text width={\linksize}, draw=black, fill=black}}
    \draw[dotted] ({0-\dd},{0}) -- ({\txl+\dd},{0});
    \draw[dotted] ({0-\dd},{\ty}) -- ({\txl+\dd},{\ty});
    \node [blp] (l1) at ({0+0*\dx},{0+0*\ty}) {};
    \node [blp] (u1) at ({0+0*\dx},{0+1*\ty}) {};
    \draw (l1) -- (u1);
  \end{tikzpicture},\quad
      \begin{tikzpicture}
    \def\scp{0.666}
    \def\linksize{\scp*0.075cm}
    \def\pointsize{\scp*0.25cm}
    \def\dd{\scp*0.5cm}
    \def\dx{\scp*1cm}
    \def\cx{\scp*0.3cm}
    \def\txu{1*\dx}    
    \def\txl{1*\dx}
    \def\dy{\scp*1cm}
    \def\cy{\scp*0.3cm}
    \def\ty{2*\dy}
    \tikzset{whp/.style={circle, inner sep=0pt, text width={\pointsize}, draw=black, fill=white}}
    \tikzset{blp/.style={circle, inner sep=0pt, text width={\pointsize}, draw=black, fill=black}}
    \tikzset{lk/.style={regular polygon, regular polygon sides=4, inner sep=0pt, text width={\linksize}, draw=black, fill=black}}
    \draw[dotted] ({0-\dd},{0}) -- ({\txl+\dd},{0});
    \draw[dotted] ({0-\dd},{\ty}) -- ({\txl+\dd},{\ty});
    \node [whp] (l1) at ({0+0*\dx},{0+0*\ty}) {};
    \node [whp] (l2) at ({0+1*\dx},{0+0*\ty}) {};
    \draw (l1) --++ (0,{\dy}) -| (l2);
  \end{tikzpicture},\quad
      \begin{tikzpicture}
    \def\scp{0.666}
    \def\linksize{\scp*0.075cm}
    \def\pointsize{\scp*0.25cm}
    \def\dd{\scp*0.5cm}
    \def\dx{\scp*1cm}
    \def\cx{\scp*0.3cm}
    \def\txu{1*\dx}    
    \def\txl{1*\dx}
    \def\dy{\scp*1cm}
    \def\cy{\scp*0.3cm}
    \def\ty{2*\dy}
    \tikzset{whp/.style={circle, inner sep=0pt, text width={\pointsize}, draw=black, fill=white}}
    \tikzset{blp/.style={circle, inner sep=0pt, text width={\pointsize}, draw=black, fill=black}}
    \tikzset{lk/.style={regular polygon, regular polygon sides=4, inner sep=0pt, text width={\linksize}, draw=black, fill=black}}
    \draw[dotted] ({0-\dd},{0}) -- ({\txl+\dd},{0});
    \draw[dotted] ({0-\dd},{\ty}) -- ({\txl+\dd},{\ty});
    \node [whp] (l1) at ({0+0*\dx},{0+0*\ty}) {};
    \node [blp] (l2) at ({0+1*\dx},{0+0*\ty}) {};
    \draw (l1) --++ (0,{\dy}) -| (l2);
  \end{tikzpicture},\quad
    \begin{tikzpicture}
    \def\scp{0.666}
    \def\linksize{\scp*0.075cm}
    \def\pointsize{\scp*0.25cm}
    \def\dd{\scp*0.5cm}
    \def\dx{\scp*1cm}
    \def\cx{\scp*0.3cm}
    \def\txu{2*\dx}    
    \def\txl{2*\dx}
    \def\dy{\scp*1cm}
    \def\cy{\scp*0.3cm}
    \def\ty{2*\dy}
    \tikzset{whp/.style={circle, inner sep=0pt, text width={\pointsize}, draw=black, fill=white}}
    \tikzset{blp/.style={circle, inner sep=0pt, text width={\pointsize}, draw=black, fill=black}}
    \tikzset{lk/.style={regular polygon, regular polygon sides=4, inner sep=0pt, text width={\linksize}, draw=black, fill=black}}
    \draw[dotted] ({0-\dd},{0}) -- ({\txl+\dd},{0});
    \draw[dotted] ({0-\dd},{\ty}) -- ({\txl+\dd},{\ty});
    \node [blp] (u1) at ({0+0*\dx},{0+1*\ty}) {};
    \node [blp] (u2) at ({0+1*\dx},{0+1*\ty}) {};
    \node [blp] (u3) at ({0+2*\dx},{0+1*\ty}) {};
    \node [lk, yshift={-\dy}] at (u1) {};
    \node [lk, yshift={-\dy}] at (u2) {};
    \node [lk, yshift={-\dy}] at (u3) {};    
    \draw (u1) --++(0,{-\dy}) -|(u3);
    \draw (u2) --++(0,{-\dy}) -|(u3);
  \end{tikzpicture},\quad
      \begin{tikzpicture}
    \def\scp{0.666}
    \def\linksize{\scp*0.075cm}
    \def\pointsize{\scp*0.25cm}
    \def\dd{\scp*0.5cm}
    \def\dx{\scp*1cm}
    \def\cx{\scp*0.3cm}
    \def\txu{1*\dx}    
    \def\txl{1*\dx}
    \def\dy{\scp*1cm}
    \def\cy{\scp*0.3cm}
    \def\ty{2*\dy}
    \tikzset{whp/.style={circle, inner sep=0pt, text width={\pointsize}, draw=black, fill=white}}
    \tikzset{blp/.style={circle, inner sep=0pt, text width={\pointsize}, draw=black, fill=black}}
    \tikzset{lk/.style={regular polygon, regular polygon sides=4, inner sep=0pt, text width={\linksize}, draw=black, fill=black}}
    \draw[dotted] ({0-\dd},{0}) -- ({\txl+\dd},{0});
    \draw[dotted] ({0-\dd},{\ty}) -- ({\txl+\dd},{\ty});
    \node [blp] (l1) at ({0+0*\dx},{0+0*\ty}) {};
    \node [whp] (l2) at ({0+1*\dx},{0+0*\ty}) {};
    \node [whp] (u1) at ({0+0*\dx},{0+1*\ty}) {};
    \node [blp] (u2) at ({0+1*\dx},{0+1*\ty}) {};
    \draw (l2) -- (u1);    
    \draw (l1) -- (u2);
  \end{tikzpicture},
  \end{mycenter}
  \begin{mycenter}[1em]
    \begin{tikzpicture}
    \def\scp{0.666}
    \def\linksize{\scp*0.075cm}
    \def\pointsize{\scp*0.25cm}
    \def\dd{\scp*0.5cm}
    \def\dx{\scp*1cm}
    \def\cx{\scp*0.3cm}
    \def\txu{2*\dx}    
    \def\txl{3*\dx}
    \def\dy{\scp*1cm}
    \def\cy{\scp*0.3cm}
    \def\ty{3*\dy}
    \tikzset{whp/.style={circle, inner sep=0pt, text width={\pointsize}, draw=black, fill=white}}
    \tikzset{blp/.style={circle, inner sep=0pt, text width={\pointsize}, draw=black, fill=black}}
    \tikzset{lk/.style={regular polygon, regular polygon sides=4, inner sep=0pt, text width={\linksize}, draw=black, fill=black}}
    \draw[dotted] ({0-\dd},{0}) -- ({\txl+\dd},{0});
    \draw[dotted] ({0-\dd},{\ty}) -- ({\txl+\dd},{\ty});
    \node [blp] (l1) at ({0+0*\dx},{0+0*\ty}) {};
    \node [whp] (l2) at ({0+1*\dx},{0+0*\ty}) {};
    \node [whp] (l3) at ({0+2*\dx},{0+0*\ty}) {};
    \node [blp] (l4) at ({0+3*\dx},{0+0*\ty}) {};
    \node [whp] (u1) at ({0+0*\dx},{0+1*\ty}) {};
    \node [blp] (u2) at ({0+1*\dx},{0+1*\ty}) {};
    \node [whp] (u3) at ({0+2*\dx},{0+1*\ty}) {};
    \node [lk, yshift={-\dy}] at (u1) {};
    \node [lk, yshift={-\dy}] at (u2) {};
    \node [lk, yshift={-\dy}] at (u3) {};    
    \draw [->] (l1) --++ (0,{\dy});
    \draw (l2) -- (u2);
    \draw (u1) --++(0,{-\dy}) -|(u3);
    \draw (l3) --++ (0,{\dy}) -| (l4);
  \end{tikzpicture},\quad
    \begin{tikzpicture}
    \def\scp{0.666}
    \def\linksize{\scp*0.075cm}
    \def\pointsize{\scp*0.25cm}
    \def\dd{\scp*0.5cm}
    \def\dx{\scp*1cm}
    \def\cx{\scp*0.3cm}
    \def\txu{2*\dx}    
    \def\txl{3*\dx}
    \def\dy{\scp*1cm}
    \def\cy{\scp*0.3cm}
    \def\ty{3*\dy}
    \tikzset{whp/.style={circle, inner sep=0pt, text width={\pointsize}, draw=black, fill=white}}
    \tikzset{blp/.style={circle, inner sep=0pt, text width={\pointsize}, draw=black, fill=black}}
    \tikzset{lk/.style={regular polygon, regular polygon sides=4, inner sep=0pt, text width={\linksize}, draw=black, fill=black}}
    \draw[dotted] ({0-\dd},{0}) -- ({\txl+\dd},{0});
    \draw[dotted] ({0-\dd},{\ty}) -- ({\txl+\dd},{\ty});
    \node [whp] (u1) at ({0+0*\dx},{0+1*\ty}) {};
    \node [blp] (u2) at ({0+1*\dx},{0+1*\ty}) {};
    \node [whp] (u3) at ({0+2*\dx},{0+1*\ty}) {};
    \node [whp] (u4) at ({0+3*\dx},{0+1*\ty}) {};    
    \draw (u1) --++(0,{-\dy}) -| (u3);
    \draw (u2) --++(0,{-2*\dy}) -| (u4);    
  \end{tikzpicture},\quad  
   \begin{tikzpicture}
    \def\scp{0.666}
    \def\linksize{\scp*0.075cm}
    \def\pointsize{\scp*0.25cm}
    \def\dd{\scp*0.5cm}
    \def\dx{\scp*1cm}
    \def\cx{\scp*0.3cm}
    \def\txu{2*\dx}    
    \def\txl{3*\dx}
    \def\dy{\scp*1cm}
    \def\cy{\scp*0.3cm}
    \def\ty{3*\dy}
    \tikzset{whp/.style={circle, inner sep=0pt, text width={\pointsize}, draw=black, fill=white}}
    \tikzset{blp/.style={circle, inner sep=0pt, text width={\pointsize}, draw=black, fill=black}}
    \tikzset{lk/.style={regular polygon, regular polygon sides=4, inner sep=0pt, text width={\linksize}, draw=black, fill=black}}
    \draw[dotted] ({0-\dd},{0}) -- ({\txl+\dd},{0});
    \draw[dotted] ({0-\dd},{\ty}) -- ({\txl+\dd},{\ty});
    \node [blp] (l1) at ({0+0*\dx},{0+0*\ty}) {};
    \node [whp] (l2) at ({0+1*\dx},{0+0*\ty}) {};
    \node [blp] (l3) at ({0+2*\dx},{0+0*\ty}) {};
    \node [whp] (l4) at ({0+3*\dx},{0+0*\ty}) {};
    \node [lk, yshift={2*\dy}] at (l1) {};
    \node [lk, yshift={2*\dy}] at (l2) {};
    \node [lk, yshift={2*\dy}] at (l4) {};    
    \draw [->] (l3) --++ (0,{\dy});
    \draw (l1) --++ (0,{2*\dy}) -| (l2);    
    \draw (l2) --++ (0,{2*\dy}) -| (l4);
    \end{tikzpicture}
  \end{mycenter}
  \par
  Moreover, we say that the two colors $\bullet$ and $\circ$ are \emph{inverse} to each other. A point $\alpha$ which is less than another point $\beta$ with respect to the total order of that row is said to lie \emph{left} of $\beta$. And, if so, $\beta$ lies \emph{right} of $\alpha$.
  \par
  The elements of a block are called its \emph{legs}. If a block contains points from both $R_L$ and $R_U$, we speak of a \emph{through block}. If on the other hand a block is contained in one row, either $R_L$ or $R_U$, it is said to be \emph{non-through}, speaking of \emph{lower} and \emph{upper non-through blocks} respectively. 
\par
We call a block with just one leg a \emph{singleton}. Blocks with exactly two legs are \emph{pairs}.  The number of legs of a block is its \emph{size}. Likewise, the total number of points of $R_L\cup R_U$ is the \emph{size} of the partition. The number of lower points is the \emph{length} of the lower row and we adopt the same terminology  for the upper row. A partition with only pairs is called a \emph{pair partition}. 
\par
We say that the upper and the lower rows are \emph{opposite} each other. To address points in a specified row we sometimes use the rank of the point in the total order of that row, which we count up from the least element, which ranks $1$. To refer to the lower point of rank $i$, we write $\lop{i}$, and, similarly, $\upp{i}$ for the upper point of rank $i$. The points $\lop{i}$ and $\upp{i}$ are said to be \emph{counterparts} of each other. Likewise, a set of points of one given row is said to be the \emph{counterpart} of the set comprising exactly the points of the same ranks of the opposite row. A through block $B$ is called \emph{straight} if $L\cap B$ and $L\cap U$ are counterparts of each other. 
\par
The set of all two-colored partitions is denoted $\Cp$.  The \emph{empty partition} $\emptyset$ is the unique element with empty rows. We denote by $\Cpp$ the set of all pair partitions and by $\Cp_{\leq 2}$ the set of all partitions all of whose blocks have sizes one or two.
\par
Whenever we deal with multiple partitions at a time, we identify all their points via the total orderings to the extent that this is possible. For example, given $p,p'\in \Cp$ and a point $\lop{i}$ in $p$, we do not distinguish between the point $\lop{i}$ in $p$ and the point $\lop{i}$ in $p'$ (if $p'$ has at least $i$ lower points). And the same goes for subsets of points. Here is an overview of the graphical representation of partitions:
    
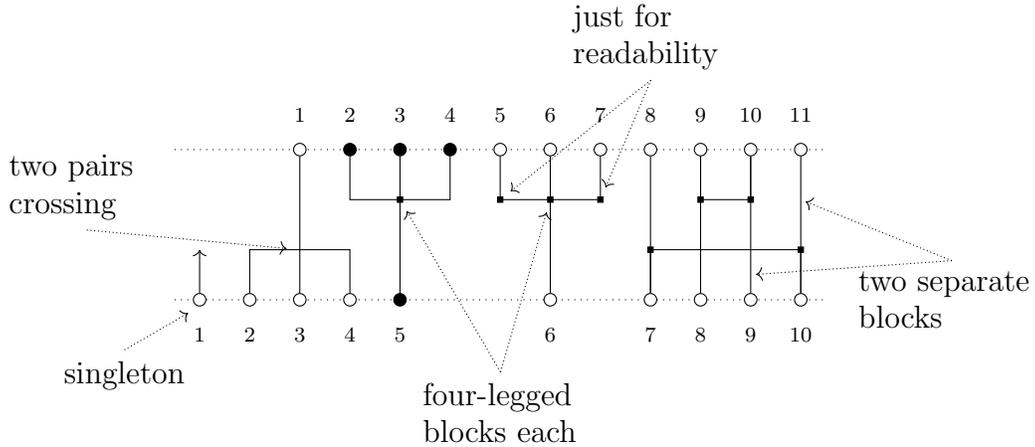
\begin{mycenter}[1em]
    \begin{tikzpicture}
    \def\scp{0.666}
    \def\linksize{\scp*0.075cm}
    \def\pointsize{\scp*0.25cm}
    \def\dd{\scp*0.5cm}
    \def\dx{\scp*1cm}
    \def\cx{\scp*0.3cm}
    \def\txu{12*\dx}    
    \def\txl{12*\dx}
    \def\dy{\scp*1cm}
    \def\cy{\scp*0.3cm}
    \def\ty{3*\dy}
    \tikzset{whp/.style={circle, inner sep=0pt, text width={\pointsize}, draw=black, fill=white}}
    \tikzset{blp/.style={circle, inner sep=0pt, text width={\pointsize}, draw=black, fill=black}}
    \tikzset{lk/.style={regular polygon, regular polygon sides=4, inner sep=0pt, text width={\linksize}, draw=black, fill=black}}
    \draw[dotted] ({0-\dd},{0}) -- ({\txl+\dd},{0});
    \draw[dotted] ({0-\dd},{\ty}) -- ({\txl+\dd},{\ty});
    \node [whp] (l1) at ({0+0*\dx},{0+0*\ty}) {};    
    \node [whp] (l2) at ({0+1*\dx},{0+0*\ty}) {};    
    \node [whp] (l3) at ({0+2*\dx},{0+0*\ty}) {};
    \node [whp] (l4) at ({0+3*\dx},{0+0*\ty}) {};
    \node [blp] (l5) at ({0+4*\dx},{0+0*\ty}) {};
    \node [whp] (l8) at ({0+7*\dx},{0+0*\ty}) {};
    \node [whp] (l10) at ({0+9*\dx},{0+0*\ty}) {};
    \node [whp] (l11) at ({0+10*\dx},{0+0*\ty}) {};
    \node [whp] (l12) at ({0+11*\dx},{0+0*\ty}) {};
    \node [whp] (l13) at ({0+12*\dx},{0+0*\ty}) {};        
    \node [whp] (u3) at ({0+2*\dx},{0+1*\ty}) {};
    \node [blp] (u4) at ({0+3*\dx},{0+1*\ty}) {};
    \node [blp] (u5) at ({0+4*\dx},{0+1*\ty}) {};
    \node [blp] (u6) at ({0+5*\dx},{0+1*\ty}) {};
    \node [whp] (u7) at ({0+6*\dx},{0+1*\ty}) {};
    \node [whp] (u8) at ({0+7*\dx},{0+1*\ty}) {};
    \node [whp] (u9) at ({0+8*\dx},{0+1*\ty}) {};
    \node [whp] (u10) at ({0+9*\dx},{0+1*\ty}) {};
    \node [whp] (u11) at ({0+10*\dx},{0+1*\ty}) {};
    \node [whp] (u12) at ({0+11*\dx},{0+1*\ty}) {};
    \node [whp] (u13) at ({0+12*\dx},{0+1*\ty}) {};        
    \node [lk, yshift={-\dy}] at (u5) {};
    \node [lk, yshift={-\dy}] at (u7) {};
    \node [lk, yshift={-\dy}] at (u8) {};    
    \node [lk, yshift={-\dy}] at (u9) {};
    \node [lk, yshift={-\dy}] at (u11) {};        
    \node [lk, yshift={-\dy}] at (u12) {};
    \node [lk, yshift={\dy}] at (l10) {};        
    \node [lk, yshift={\dy}] at (l13) {};            
    \draw [->] (l1) --++ (0,{\dy});
    \draw (l3) -- (u3);
    \draw (l5) -- (u5);
    \draw (l8) -- (u8);
    \draw (l10) -- (u10);
    \draw (l11) -- (u11);
    \draw (l12) -- (u12);
    \draw (l13) -- (u13);    
    \draw (u4) --++(0,{-\dy}) -|(u6);
    \draw (u7) --++(0,{-\dy}) -|(u9);    
    \draw (l2) --++ (0,{\dy}) -| (l4);
    \draw (u11) --++(0,{-\dy}) -|(u12);
    \draw (l10) --++ (0,{\dy}) -| (l13);
    \node (lab1) at (-1,-1) {singleton};
    \draw[densely dotted,->, shorten >=5pt] (lab1.north) -- (l1.south);
    \node[text width=5em] (lab2) at (-1.5,1.5) {two pairs\\crossing}; 
    \draw[densely dotted,->, shorten >=5pt] (lab2.south) -- ($(l3)+(0,{\dy})$);
    \node [below = 4pt of l1] {$\scriptstyle 1$};
    \node [below = 4pt of l2] {$\scriptstyle 2$};
    \node [below = 4pt of l3] {$\scriptstyle 3$};
    \node [below = 4pt of l4] {$\scriptstyle 4$};
    \node [below = 4pt of l5] {$\scriptstyle 5$};
    \node [below = 4pt of l8] {$\scriptstyle 6$};
    \node [below = 4pt of l10] {$\scriptstyle 7$};
    \node [below = 4pt of l11] {$\scriptstyle 8$};
    \node [below = 4pt of l12] {$\scriptstyle 9$};
    \node [below = 4pt of l13] {$\scriptstyle 10$};
    \node [above = 4pt of u3] {$\scriptstyle 1$};
    \node [above = 4pt of u4] {$\scriptstyle 2$};
    \node [above = 4pt of u5] {$\scriptstyle 3$};
    \node [above = 4pt of u6] {$\scriptstyle 4$};
    \node [above = 4pt of u7] {$\scriptstyle 5$};
    \node [above = 4pt of u8] {$\scriptstyle 6$};
    \node [above = 4pt of u9] {$\scriptstyle 7$};
    \node [above = 4pt of u10] {$\scriptstyle 8$};
    \node [above = 4pt of u11] {$\scriptstyle 9$};
    \node [above = 4pt of u12] {$\scriptstyle 10$};
    \node [above = 4pt of u13] {$\scriptstyle 11$};
    \node[text width=5em] (lab3) at (4,-1.5) {four-legged\\blocks each};
    \draw[densely dotted,->, shorten >=5pt] (lab3.north) -- ($(u5)+(0,{-\dy})$);
    \draw[densely dotted,->, shorten >=5pt] (lab3.north) -- ($(u8)+(0,{-\dy})$);
    \node[text width = 5em] (lab4) at (6,3.5) {just for\\readability}; 
    \draw[densely dotted,->, shorten >=5pt] (lab4.south) -- ($(u7)+(0,{-\dy})$);
    \draw[densely dotted,->, shorten >=5pt] (lab4.south) -- ($(u9)+(0,{-\dy})$);
    \node[text width = 6em] (lab5) at (10,0) {two separate\\blocks}; 
    \draw[densely dotted,->, shorten >=2pt] (lab5.north) -- ($(l12)+(0,{0.5*\dy})$);
    \draw[densely dotted,->, shorten >=2pt] (lab5.north) -- ($(u13)+(0,{-\dy})$);        
  \end{tikzpicture}
\end{mycenter}
\vspace{-2em}
    \subsection{Orientation and Intervals}
\label{section:orientation}
While any $p\in \Cp$ already comes with total orders $\leq_L$ on its lower row $R_L$ and $\leq _U$ on its upper row $R_U$, in many situations it is more advantageous to consider a cyclic order on the entire set $R_L\cup R_U$ of its points. This \emph{orientation} is uniquely determined by the following four conditions: It concurs with $\leq_L$ on $R_L$, but with the inverse of $\leq_U$ on $R_U$, and the maximum of $\leq_U$ succeeds the maximum of $\leq_L$ and likewise the minimum of $\leq_L$ succeeds the minimum of $\leq_U$. 

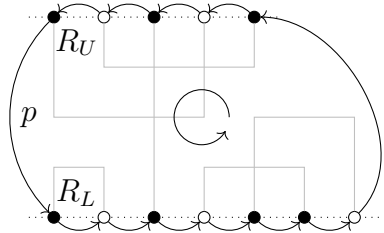
  \begin{mycenter}[0.5em]
 \begin{tikzpicture}[scale=0.666]
    \def\dx{1}
    \def\ty{4}
    \def\ox{0}
    \def\oy{0}
    \def\dy{\ty / 4}
    \def\tx{7*\dx}
    \def\od{0.5}
    \draw [dotted] ({\ox-\od},{\oy+\ty}) --  ({\ox+4*\dx+\od},{\oy+\ty});
    \draw [dotted] ({\ox-\od},{\oy}) --  ({\ox+6*\dx+\od},{\oy});
    \node [scale=0.4, circle, draw=black, fill=black] (a0) at ({\ox+0*\dx},{\oy}) {};
    \node [scale=0.4, circle, draw=black, fill=white] (a1) at ({\ox+1*\dx},{\oy}) {};
    \node [scale=0.4, circle, draw=black, fill=black] (a2) at ({\ox+2*\dx},{\oy}) {};
    \node [scale=0.4, circle, draw=black, fill=white] (a3) at ({\ox+3*\dx},{\oy}) {};
    \node [scale=0.4, circle, draw=black, fill=black] (a4) at ({\ox+4*\dx},{\oy}) {};
    \node [scale=0.4, circle, draw=black, fill=black] (a5) at ({\ox+5*\dx},{\oy}) {};
    \node [scale=0.4, circle, draw=black, fill=white] (a6) at ({\ox+6*\dx},{\oy}) {};
    \node [scale=0.4, circle, draw=black, fill=black] (b0) at ({\ox+0*\dx},{\oy+\ty}) {};
    \node [scale=0.4, circle, draw=black, fill=white] (b1) at ({\ox+1*\dx},{\oy+\ty}) {};
    \node [scale=0.4, circle, draw=black, fill=black] (b2) at ({\ox+2*\dx},{\oy+\ty}) {};
    \node [scale=0.4, circle, draw=black, fill=white] (b3) at ({\ox+3*\dx},{\oy+\ty}) {};
    \node [scale=0.4, circle, draw=black, fill=black] (b4) at ({\ox+4*\dx},{\oy+\ty}) {};
    \draw [lightgray] (a0) -- ++(0,{\dy}) -| (a1);
    \draw [lightgray] (a3) -- ++(0,{\dy}) -| (a5);
    \draw [lightgray] (a4) -- ++(0,{2*\dy}) -| (a6);
    \draw [lightgray] (a2) -- (b2);
    \draw [lightgray] (b0) -- ++(0,{-2*\dy}) -| (b3);
    \draw [lightgray] (b1) -- ++(0,{-1*\dy}) -| (b4);
    	\draw[->] (b4) to [out=135, in=45] (b3);
	\draw[->] (b3) to [out=135, in=45] (b2);
	\draw[->] (b2) to [out=135, in=45] (b1);
	\draw[->] (b1) to [out=135, in=45] (b0);
	\draw[->] (b0) to [out=225, in=135] (a0);
	\draw[->] (a0) to [out=-45, in=225] (a1);
	\draw[->] (a1) to [out=-45, in=225] (a2);
	\draw[->] (a2) to [out=-45, in=225] (a3);
	\draw[->] (a3) to [out=-45, in=225] (a4);
	\draw[->] (a4) to [out=-45, in=225] (a5);
	\draw[->] (a5) to [out=-45, in=225] (a6);
	\draw[->] (a6) to [out=45, in=0] (b4);
        \node at (0.45,3.5) {$R_U$};
        \node at (0.45,0.5) {$R_L$};
        \node at (-0.5,2) {$p$};
        \draw[->] (3.5,2) arc (0:330:0.55);        
      \end{tikzpicture}
\end{mycenter}
\noindent
That means $p$ carries the counter-clockwise orientation.\pagebreak\par
    The terms \emph{successor}, \emph{predecessor} and \emph{neighbor} always refer to the cyclic order.   If necessary to avoid
con\-fu\-sion between this cyclic order and the total orders $\leq_L$ and $\leq_U$, the latter two will be referred to as the \emph{native} orderings.
\par
  With respect to the cyclic order it makes sense to speak of \emph{intervals} $]\alpha,\beta[_p$, $]\alpha,\beta]_p$, $[\alpha,\beta[_p$ and $[\alpha,\beta]_p$ for any two points $\alpha,\beta\in R_L\cup R_U$ of $p$ with $\alpha\neq \beta$. Note that intervals of any kind are only defined for distinct limits.
  
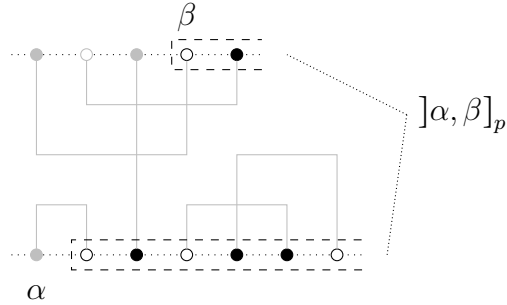
\begin{mycenter}[0.25em]
    \begin{tikzpicture}[scale=0.666]
    \def\dx{1}
    \def\ty{4}
    \def\ox{0}
    \def\oy{0}
    \def\dy{\ty / 4}
    \def\tx{7*\dx}
    \def\od{0.5}
    \def\cx{0.3}
    \def\cy{0.3}    
    \draw [dotted] ({\ox-\od},{\oy+\ty}) --  ({\ox+4*\dx+\od},{\oy+\ty});
    \draw [dotted] ({\ox-\od},{\oy}) --  ({\ox+6*\dx+\od},{\oy});
    \node [scale=0.4, circle, draw=lightgray, fill=lightgray] (a0) at ({\ox+0*\dx},{\oy}) {};
    \node [scale=0.4, circle, draw=black, fill=white] (a1) at ({\ox+1*\dx},{\oy}) {};
    \node [scale=0.4, circle, draw=black, fill=black] (a2) at ({\ox+2*\dx},{\oy}) {};
    \node [scale=0.4, circle, draw=black, fill=white] (a3) at ({\ox+3*\dx},{\oy}) {};
    \node [scale=0.4, circle, draw=black, fill=black] (a4) at ({\ox+4*\dx},{\oy}) {};
    \node [scale=0.4, circle, draw=black, fill=black] (a5) at ({\ox+5*\dx},{\oy}) {};
    \node [scale=0.4, circle, draw=black, fill=white] (a6) at ({\ox+6*\dx},{\oy}) {};
    \node [scale=0.4, circle, draw=lightgray, fill=lightgray] (b0) at ({\ox+0*\dx},{\oy+\ty}) {};
    \node [scale=0.4, circle, draw=lightgray, fill=white] (b1) at ({\ox+1*\dx},{\oy+\ty}) {};
    \node [scale=0.4, circle, draw=lightgray, fill=lightgray] (b2) at ({\ox+2*\dx},{\oy+\ty}) {};
    \node [scale=0.4, circle, draw=black, fill=white] (b3) at ({\ox+3*\dx},{\oy+\ty}) {};
    \node [scale=0.4, circle, draw=black, fill=black] (b4) at ({\ox+4*\dx},{\oy+\ty}) {};
    \draw[lightgray] (a0) -- ++(0,{\dy}) -| (a1);
    \draw[lightgray] (a3) -- ++(0,{\dy}) -| (a5);
    \draw[lightgray] (a4) -- ++(0,{2*\dy}) -| (a6);
    \draw[lightgray] (a2) -- (b2);
    \draw[lightgray] (b0) -- ++(0,{-2*\dy}) -| (b3);
    \draw[lightgray] (b1) -- ++(0,{-1*\dy}) -| (b4);
    \node[yshift=-0.5cm] at (a0) {$\alpha$};
    \node[yshift=0.5cm] at (b3) {$\beta$};
    \draw[dashed] ({\ox+6*\dx+\od},{\oy-\cy})  -- ({\ox+\dx-\cx},{\oy-\cy}) -- ({\ox+\dx-\cx},{\oy+\cy}) -- ({\ox+6*\dx+\od},{\oy+\cy});
    \draw[dashed] ({\ox+4*\dx+\od},{\oy+\ty-\cy})  -- ({\ox+3*\dx-\cx},{\oy+\ty-\cy}) -- ({\ox+3*\dx-\cx},{\oy+\ty+\cy}) -- ({\ox+4*\dx+\od},{\oy+\ty+\cy});
    \node (lab) at ({\ox+\tx+1.5*\dx},{\oy+0.7*\ty}) {$\left]\alpha,\beta\right]_p$};
      \draw[densely dotted] (lab.west) to ({\ox+5*\dx},{\oy+\ty});
      \draw[densely dotted] (lab.west) to ({\ox+7*\dx},{\oy});      
    \end{tikzpicture}
    \end{mycenter}
    We call a set $S$ of points \emph{consecutive} if $S$ is empty, an interval or of the forms $\{\alpha\}$ or $(R_L\cup R_U)\backslash \{\alpha\}$ for some point $\alpha$.
    \vspace{-0.5em}
 \subsection{Ordered Tuples and Crossings}
 We can extend the notion of intervals to tuples of more than two points:
 For $n\geq 3$ pairwise distinct points $\alpha_1,\ldots,\alpha_n$ in $p\in \Cp$ we say that the tuple $(\alpha_1,\ldots,\alpha_n)$ is \emph{ordered} in $p$ if for all $i,j,k\in \pint$ with $i,j,k\leq n$ and $i<j$ the set $]\alpha_i,\alpha_j[_p$ contains $\alpha_k$ if and only if $i<k<j$. In fact, we can even talk about tuples of pairwise disjoint consecutive sets being ordered.
 \par
 We say that two distinct blocks $B$ and $B'$  of $p$ \emph{cross} each other if there exist points $\alpha,\beta,\alpha',\beta'$ in $p$ such that the tuple $(\alpha,\alpha',\beta,\beta')$ is ordered and such that $\alpha,\beta\in B$ and $\alpha',\beta'\in B'$. 
 
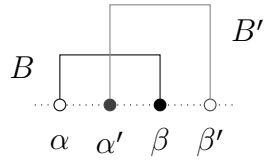
\par
     \begin{mycenter}[0.25em]
\begin{tikzpicture}[scale=0.666]
    \def\dx{1}
    \def\ty{4}
    \def\ox{0}
    \def\oy{0}
    \def\dy{\ty / 4}
    \def\tx{7*\dx}
    \def\od{0.5}
    \def\cx{0.3}
    \def\cy{0.3}    
    \draw [dotted] ({\ox-\od},{\oy}) --  ({\ox+3*\dx+\od},{\oy});
    \node [scale=0.4, circle, draw=black, fill=white] (a0) at ({\ox+0*\dx},{\oy}) {};
    \node [scale=0.4, circle, draw=darkgray, fill=darkgray] (a1) at ({\ox+1*\dx},{\oy}) {};
    \node [scale=0.4, circle, draw=black, fill=black] (a2) at ({\ox+2*\dx},{\oy}) {};
    \node [scale=0.4, circle, draw=darkgray, fill=white] (a3) at ({\ox+3*\dx},{\oy}) {};
    \draw (a0) -- ++(0,{\dy}) -| (a2);
    \draw[gray] (a1) -- ++(0,{2*\dy}) -| (a3);
    \node [yshift=-0.5cm] at (a0) {$\alpha$};
    \node [yshift=-0.5cm] at (a1) {$\alpha'$};
    \node [yshift=-0.5cm] at (a2) {$\beta$};
    \node [yshift=-0.5cm] at (a3) {$\beta'$};
    \node [yshift=0.5cm,xshift=-0.5cm] at (a0) {$B$};
    \node [yshift=1cm,xshift=0.5cm] at (a3) {$B'$};            
\end{tikzpicture}
\end{mycenter}
If $p$ has no crossing blocks, we call $p$ a \emph{non-crossing} partition. The set of all non-crossing partitions is denoted by $\mc{NC}^{\circ\bullet}$.
    \vspace{-0.5em}
\subsection{Normalized Color and Color Sum}
Just like the cyclic order is often more convenient than the total orderings of the rows, it is useful to consider besides the original, \emph{native} coloring of the points a second one:
By the \emph{normalized color} of a point $\alpha$ in $p\in \Cp$ we mean simply its native color in case $\alpha$ is a lower point, but the inverse of its native color if $\alpha$ is an upper point.\nopagebreak
\par\nopagebreak
  We call the signed measure $\sigma_p$ on the set $P_p$ of all points of $p$ which assigns $1$ to normalized $\circ$ and $-1$ to normalized $\bullet$ the \emph{color sum} of $p$.  
  Null sets of $\sigma_p$ are also referred to as \emph{neutral}.\nopagebreak
  \par\nopagebreak
 The color sum $\toco(p)\eqpd \sigma_p(P_p)$ of the set $P_p$ of all points of $p$ is called the \emph{total color sum} of $p$.\pagebreak

\subsection{Color Distance}
Besides the color sum, a measure-like structure on the points of a partition, the coloring of the partition also induces a metric-like one: Given two points $\alpha$ and $\beta$  in $p\in \Cp$ we call
   \begin{align*}
     \delta_p(\alpha,\beta)\eqpd
     \begin{cases}
       \toco(p)&\text{if }\alpha= \beta,\\
       \sigma_p(]\alpha,\beta[_p)&\text{if }\alpha\neq \beta\text{ and } \alpha, \beta \text{ have different normalized colors},\\
       \sigma_p(]\alpha,\beta]_p)&\text{if }\alpha\neq \beta\text{ and } \alpha, \beta \text{ have the same normalized color},
     \end{cases}  
   \end{align*}
   the \emph{color distance} from $\alpha$ to $\beta$ in $p$. The map $\delta_p$ indeed has properties of a \enquote{distance}.
\begin{lemma}
  \label{lemma:color-distance}
   Let  $\alpha$, $\beta$ and $\gamma$ be points in $p\in \Cp$.
   \begin{enumerate}[wide,label=(\alph*)]
   \item\label{lemma:color-distance-1}  $\delta_p(\alpha,\alpha)\equiv 0 \mod \toco(p)$.
   \item\label{lemma:color-distance-2}  $\delta_p(\alpha,\beta)\equiv -\delta_p(\beta,\alpha)\mod \toco(p)$.
   \item\label{lemma:color-distance-3}  $\delta_p(\alpha,\gamma)\equiv \delta_p(\alpha,\beta)+\delta_p(\beta,\gamma)\mod \toco(p)$.
   \end{enumerate}
 \end{lemma}
 \begin{proof}
   \begin{enumerate}[wide,label=(\alph*)]
   \item The first claim is part of the definition of $\delta_p$.
   \item We can assume $\alpha\neq \beta$. Rewrite the definition of $\delta_p$ as
\[\delta_p(\alpha,\beta)=\sigma_p(]\alpha,\beta]_p)+{\textstyle\frac{1}{2}}\left(\sigma_p(\{\alpha\})-\sigma_p(\{\beta\})\right).\]
Using $\sigma_p(]\alpha,\beta]_p)\equiv -\sigma_p(]\beta,\alpha]_p)\mod \toco(p)$ now proves the claim.
\item Again, we can suppose $\alpha,\beta$ and $\gamma$ are pairwise different. Otherwise Parts~\ref{lemma:color-distance-1} and~\ref{lemma:color-distance-2} already prove the claim. Now compute, employing the formula for $\delta_p$ from the proof of Claim~\ref{lemma:color-distance-2},
  \begin{align*} \delta_p(\alpha,\beta)+\delta_p(\beta,\gamma)&=\sigma_p(]\alpha,\beta]_p)+\sigma_p(]\beta,\gamma]_p)\\ & \phantom{{}={}}+{\textstyle\frac{1}{2}}\left(\sigma_p(\{\alpha\})-\sigma_p(\{\beta\})\right)+{\textstyle\frac{1}{2}}\left(\sigma_p(\{\beta\})-\sigma_p(\{\gamma\})\right).
  \end{align*}
  From $\sigma_p(]\alpha,\beta]_p)+\sigma_p(]\beta,\gamma]_p)\equiv \sigma_p(]\alpha,\gamma]_p)\mod \toco(p)$ now follows the claim.\qedhere
   \end{enumerate}
 \end{proof}
 \section{Basic Definitions: Categories of Partitions}
 The definition of two-colored partitions recalled, we recapitulate the definitions of \emph{operations} for partitions and of \emph{categories}. Again, see \cite{TaWe15a} for more.
\subsection{Fundamental Operations on Partitions} 
  For all $p,p'\in \Cp$ we call the partition which is obtained by appending the rows of $p'$ (left to right) to the right of the respective rows of $p$ the \emph{tensor product} $p\otimes p'$ of $(p,p')$. Especially, we can write tensor powers like $p^{\otimes n}$ given by $p\otimes \ldots\otimes p$ with $n$ factors. And we define $p^{\otimes 0}\eqpd\emptyset$.
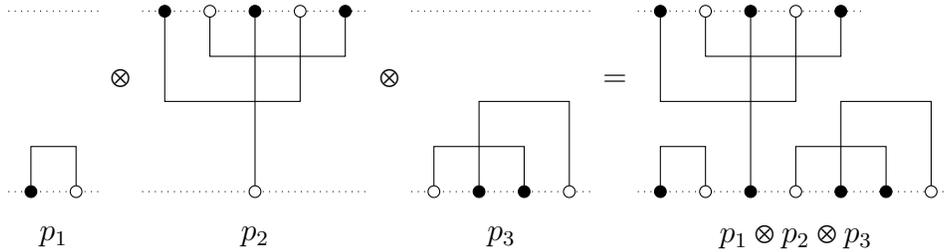
    \begin{mycenter}[1em]
  \begin{tikzpicture}[scale=0.6,baseline=1.4cm]
    \def\dx{1}
    \def\ty{4}
    \def\ox{0}
    \def\oy{0}
    \def\dy{\ty / 4}
    \def\tx{1*\dx}
    \def\od{0.5}
    \def\cx{0.3}
    \def\cy{0.3}    
    \draw [dotted] ({\ox-\od},{\oy}) --  ({\ox+1*\dx+\od},{\oy});
    \draw [dotted] ({\ox-\od},{\oy+\ty}) --  ({\ox+1*\dx+\od},{\oy+\ty});    
    \node [scale=0.4, circle, draw=black, fill=black] (a0) at ({\ox+0*\dx},{\oy}) {};
    \node [scale=0.4, circle, draw=black, fill=white] (a1) at ({\ox+1*\dx},{\oy}) {};
    \draw[black] (a0) -- ++(0,{\dy}) -| (a1);
    \node at ({0.5*\tx},{-\dy}) {$p_1$};
  \end{tikzpicture}
  $\otimes$
        \begin{tikzpicture}[scale=0.6,baseline=1.4cm]
    \def\dx{1}
    \def\ty{4}
    \def\ox{0}
    \def\oy{0}
    \def\dy{\ty / 4}
    \def\tx{4*\dx}
    \def\od{0.5}
    \def\cx{0.3}
    \def\cy{0.3}    
    \draw [dotted] ({\ox-\od},{\oy+\ty}) --  ({\ox+4*\dx+\od},{\oy+\ty});
    \draw [dotted] ({\ox-\od},{\oy}) --  ({\ox+4*\dx+\od},{\oy});
    \node [scale=0.4, circle, draw=black, fill=white] (a2) at ({\ox+2*\dx},{\oy}) {};
    \node [scale=0.4, circle, draw=black, fill=black] (b0) at ({\ox+0*\dx},{\oy+\ty}) {};
    \node [scale=0.4, circle, draw=black, fill=white] (b1) at ({\ox+1*\dx},{\oy+\ty}) {};
    \node [scale=0.4, circle, draw=black, fill=black] (b2) at ({\ox+2*\dx},{\oy+\ty}) {};
    \node [scale=0.4, circle, draw=black, fill=white] (b3) at ({\ox+3*\dx},{\oy+\ty}) {};
    \node [scale=0.4, circle, draw=black, fill=black] (b4) at ({\ox+4*\dx},{\oy+\ty}) {};
    \draw[black] (a2) -- (b2);
    \draw[black] (b0) -- ++(0,{-2*\dy}) -| (b3);
    \draw[black] (b1) -- ++(0,{-1*\dy}) -| (b4);
    \node at ({0.5*\tx},{-\dy}) {$p_2$};    
  \end{tikzpicture}
  $\otimes$
        \begin{tikzpicture}[scale=0.6,baseline=1.4cm]
    \def\dx{1}
    \def\ty{4}
    \def\ox{0}
    \def\oy{0}
    \def\dy{\ty / 4}
    \def\tx{3*\dx}
    \def\od{0.5}
    \def\cx{0.3}
    \def\cy{0.3}    
    \draw [dotted] ({\ox-\od},{\oy+\ty}) --  ({\ox+3*\dx+\od},{\oy+\ty});
    \draw [dotted] ({\ox-\od},{\oy}) --  ({\ox+3*\dx+\od},{\oy});
    \node [scale=0.4, circle, draw=black, fill=white] (a3) at ({\ox+0*\dx},{\oy}) {};
    \node [scale=0.4, circle, draw=black, fill=black] (a4) at ({\ox+1*\dx},{\oy}) {};
    \node [scale=0.4, circle, draw=black, fill=black] (a5) at ({\ox+2*\dx},{\oy}) {};
    \node [scale=0.4, circle, draw=black, fill=white] (a6) at ({\ox+3*\dx},{\oy}) {};
    \draw[black] (a3) -- ++(0,{\dy}) -| (a5);
    \draw[black] (a4) -- ++(0,{2*\dy}) -| (a6);
        \node at ({0.5*\tx},{-\dy}) {$p_3$};
  \end{tikzpicture}
  =
    \begin{tikzpicture}[scale=0.6,baseline=1.4cm]
    \def\dx{1}
    \def\ty{4}
    \def\ox{0}
    \def\oy{0}
    \def\dy{\ty / 4}
    \def\tx{6*\dx}
    \def\od{0.5}
    \def\cx{0.3}
    \def\cy{0.3}    
    \draw [dotted] ({\ox-\od},{\oy+\ty}) --  ({\ox+4*\dx+\od},{\oy+\ty});
    \draw [dotted] ({\ox-\od},{\oy}) --  ({\ox+6*\dx+\od},{\oy});
    \node [scale=0.4, circle, draw=black, fill=black] (a0) at ({\ox+0*\dx},{\oy}) {};
    \node [scale=0.4, circle, draw=black, fill=white] (a1) at ({\ox+1*\dx},{\oy}) {};
    \node [scale=0.4, circle, draw=black, fill=black] (a2) at ({\ox+2*\dx},{\oy}) {};
    \node [scale=0.4, circle, draw=black, fill=white] (a3) at ({\ox+3*\dx},{\oy}) {};
    \node [scale=0.4, circle, draw=black, fill=black] (a4) at ({\ox+4*\dx},{\oy}) {};
    \node [scale=0.4, circle, draw=black, fill=black] (a5) at ({\ox+5*\dx},{\oy}) {};
    \node [scale=0.4, circle, draw=black, fill=white] (a6) at ({\ox+6*\dx},{\oy}) {};
    \node [scale=0.4, circle, draw=black, fill=black] (b0) at ({\ox+0*\dx},{\oy+\ty}) {};
    \node [scale=0.4, circle, draw=black, fill=white] (b1) at ({\ox+1*\dx},{\oy+\ty}) {};
    \node [scale=0.4, circle, draw=black, fill=black] (b2) at ({\ox+2*\dx},{\oy+\ty}) {};
    \node [scale=0.4, circle, draw=black, fill=white] (b3) at ({\ox+3*\dx},{\oy+\ty}) {};
    \node [scale=0.4, circle, draw=black, fill=black] (b4) at ({\ox+4*\dx},{\oy+\ty}) {};
    \draw[black] (a0) -- ++(0,{\dy}) -| (a1);
    \draw[black] (a3) -- ++(0,{\dy}) -| (a5);
    \draw[black] (a4) -- ++(0,{2*\dy}) -| (a6);
    \draw[black] (a2) -- (b2);
    \draw[black] (b0) -- ++(0,{-2*\dy}) -| (b3);
    \draw[black] (b1) -- ++(0,{-1*\dy}) -| (b4);
        \node at ({0.5*\tx},{-\dy}) {$p_1\otimes p_2\otimes p_3$};
  \end{tikzpicture}
\end{mycenter}
\par
The partition which is obtained from $p$ by switching the roles of upper and lower row is called the \emph{involution} $p^*$ of $p$.
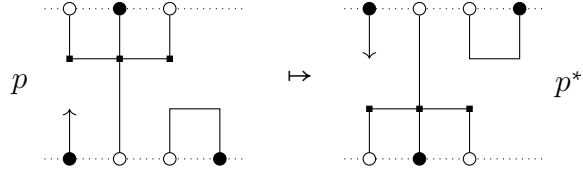
  \begin{mycenter}[1em]
        \begin{tikzpicture}[baseline=0.666*1.5cm]
    \def\scp{0.666}
    \def\linksize{\scp*0.075cm}
    \def\pointsize{\scp*0.25cm}
    \def\dd{\scp*0.5cm}
    \def\dx{\scp*1cm}
    \def\cx{\scp*0.3cm}
    \def\txu{2*\dx}    
    \def\txl{3*\dx}
    \def\dy{\scp*1cm}
    \def\cy{\scp*0.3cm}
    \def\ty{3*\dy}
    \tikzset{whp/.style={circle, inner sep=0pt, text width={\pointsize}, draw=black, fill=white}}
    \tikzset{blp/.style={circle, inner sep=0pt, text width={\pointsize}, draw=black, fill=black}}
    \tikzset{lk/.style={regular polygon, regular polygon sides=4, inner sep=0pt, text width={\linksize}, draw=black, fill=black}}
    \draw[dotted] ({0-\dd},{0}) -- ({\txl+\dd},{0});
    \draw[dotted] ({0-\dd},{\ty}) -- ({\txl+\dd},{\ty});
    \node [blp] (l1) at ({0+0*\dx},{0+0*\ty}) {};
    \node [whp] (l2) at ({0+1*\dx},{0+0*\ty}) {};
    \node [whp] (l3) at ({0+2*\dx},{0+0*\ty}) {};
    \node [blp] (l4) at ({0+3*\dx},{0+0*\ty}) {};
    \node [whp] (u1) at ({0+0*\dx},{0+1*\ty}) {};
    \node [blp] (u2) at ({0+1*\dx},{0+1*\ty}) {};
    \node [whp] (u3) at ({0+2*\dx},{0+1*\ty}) {};
    \node [lk, yshift={-\dy}] at (u1) {};
    \node [lk, yshift={-\dy}] at (u2) {};
    \node [lk, yshift={-\dy}] at (u3) {};    
    \draw [->] (l1) --++ (0,{\dy});
    \draw (l2) -- (u2);
    \draw (u1) --++(0,{-\dy}) -|(u3);
    \draw (l3) --++ (0,{\dy}) -| (l4);
    \node at ({-\dx},{0.5*\ty}) {$p$};
  \end{tikzpicture}
  \quad$\mapsto$\quad
          \begin{tikzpicture}[scale=1,yscale=-1, baseline=-0.666*1.5cm]
    \def\scp{0.666}
    \def\linksize{\scp*0.075cm}
    \def\pointsize{\scp*0.25cm}
    \def\dd{\scp*0.5cm}
    \def\dx{\scp*1cm}
    \def\cx{\scp*0.3cm}
    \def\txu{2*\dx}    
    \def\txl{3*\dx}
    \def\dy{\scp*1cm}
    \def\cy{\scp*0.3cm}
    \def\ty{3*\dy}
    \tikzset{whp/.style={circle, inner sep=0pt, text width={\pointsize}, draw=black, fill=white}}
    \tikzset{blp/.style={circle, inner sep=0pt, text width={\pointsize}, draw=black, fill=black}}
    \tikzset{lk/.style={regular polygon, regular polygon sides=4, inner sep=0pt, text width={\linksize}, draw=black, fill=black}}
    \draw[dotted] ({0-\dd},{0}) -- ({\txl+\dd},{0});
    \draw[dotted] ({0-\dd},{\ty}) -- ({\txl+\dd},{\ty});
    \node [blp] (l1) at ({0+0*\dx},{0+0*\ty}) {};
    \node [whp] (l2) at ({0+1*\dx},{0+0*\ty}) {};
    \node [whp] (l3) at ({0+2*\dx},{0+0*\ty}) {};
    \node [blp] (l4) at ({0+3*\dx},{0+0*\ty}) {};
    \node [whp] (u1) at ({0+0*\dx},{0+1*\ty}) {};
    \node [blp] (u2) at ({0+1*\dx},{0+1*\ty}) {};
    \node [whp] (u3) at ({0+2*\dx},{0+1*\ty}) {};
    \node [lk, yshift={\dy}] at (u1) {};
    \node [lk, yshift={\dy}] at (u2) {};
    \node [lk, yshift={\dy}] at (u3) {};    
    \draw [->] (l1) --++ (0,{\dy});
    \draw (l2) -- (u2);
    \draw (u1) --++(0,{-\dy}) -|(u3);
    \draw (l3) --++ (0,{\dy}) -| (l4);
    \node at ({\txl+\dx},{0.5*\ty}) {$p^\ast$};    
  \end{tikzpicture}
  \end{mycenter}
  \par
  We say that the pairing $(p,p')$ is \emph{composable} if the upper row of $p$ and the lower row of $p'$ agree in size and coloration if compared according to their total orders both left to right.
  \par
  For all $x\in \{p,p'\}$ let $R_{x,U}$ denote the upper and $R_{x,L}$ the lower row of $x$. If $(p,p')$ is composable, then the \emph{composition} $pp'$ of $(p,p')$ is defined as follows: Its lower row is given by $R_{p,L}$ and its upper row by $R_{p',U}$. The blocks of $p$ contained in $R_{p,L}$ and those of $p'$ contained in $R_{p',U}$ are also blocks of $pp'$.  Once we identify the points of $R_{p,U}$ and $R_{p',L}$ we can form the join $s$ of the partitions induced there by $p$ and $p'$. For every block $D$ of $s$ the set
    \begin{align*}
       \left(R_{p,L}\cap\bigcup_{\substack{B \text{ block of } p\\B\cap D\neq \emptyset }}B\right)\cup\left( R_{p',U}\cap\bigcup_{\substack{B' \text{ block of } p'\\B'\cap D\neq \emptyset }}B'\right)
    \end{align*}
    is a block of $pp'$ unless it is empty.
    
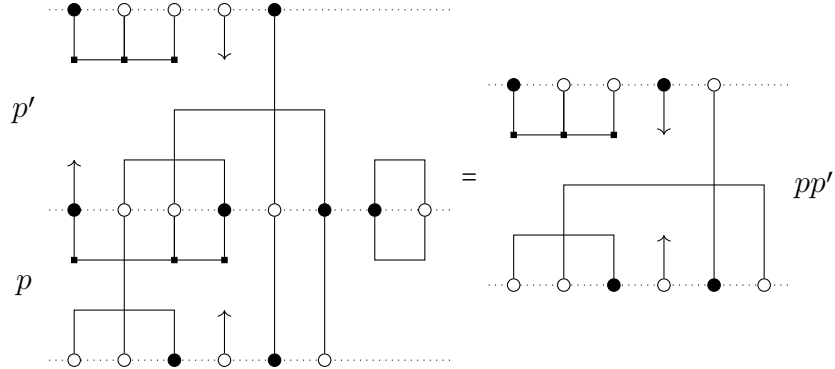
\begin{mycenter}[1em]
          \begin{tikzpicture}[baseline=3.5cm*0.666]
    \def\scp{0.666}
    \def\linksize{\scp*0.075cm}
    \def\pointsize{\scp*0.25cm}
    \def\dd{\scp*0.5cm}
    \def\dx{\scp*1cm}
    \def\cx{\scp*0.3cm}
    \def\txu{7*\dx}    
    \def\txl{7*\dx}
    \def\dy{\scp*1cm}
    \def\cy{\scp*0.3cm}
    \def\tyl{3*\dy}
    \def\tyu{4*\dy}    
    \tikzset{whp/.style={circle, inner sep=0pt, text width={\pointsize}, draw=black, fill=white}}
    \tikzset{blp/.style={circle, inner sep=0pt, text width={\pointsize}, draw=black, fill=black}}
    \tikzset{lk/.style={regular polygon, regular polygon sides=4, inner sep=0pt, text width={\linksize}, draw=black, fill=black}}
    \draw[dotted] ({0-\dd},{0}) -- ({\txl+\dd},{0});
    \draw[dotted] ({0-\dd},{\tyl}) -- ({\txl+\dd},{\tyl});
    \draw[dotted] ({0-\dd},{\tyl+\tyu}) -- ({\txl+\dd},{\tyl+\tyu});    
    \node [whp] (b1) at ({0+0*\dx},{0+0*\tyl}) {};
    \node [whp] (b2) at ({0+1*\dx},{0+0*\tyl}) {};
    \node [blp] (b3) at ({0+2*\dx},{0+0*\tyl}) {};
    \node [whp] (b4) at ({0+3*\dx},{0+0*\tyl}) {};
    \node [blp] (b5) at ({0+4*\dx},{0+0*\tyl}) {};
    \node [whp] (b6) at ({0+5*\dx},{0+0*\tyl}) {};    
    \node [blp] (m1) at ({0+0*\dx},{0+1*\tyl}) {};
    \node [whp] (m2) at ({0+1*\dx},{0+1*\tyl}) {};
    \node [whp] (m3) at ({0+2*\dx},{0+1*\tyl}) {};
    \node [blp] (m4) at ({0+3*\dx},{0+1*\tyl}) {};
    \node [whp] (m5) at ({0+4*\dx},{0+1*\tyl}) {};
    \node [blp] (m6) at ({0+5*\dx},{0+1*\tyl}) {};
    \node [blp] (m7) at ({0+6*\dx},{0+1*\tyl}) {};
    \node [whp] (m8) at ({0+7*\dx},{0+1*\tyl}) {};
    \node [blp] (t1) at ({0+0*\dx},{0+1*\tyl+1*\tyu}) {};
    \node [whp] (t2) at ({0+1*\dx},{0+1*\tyl+1*\tyu}) {};
    \node [whp] (t3) at ({0+2*\dx},{0+1*\tyl+1*\tyu}) {};
    \node [whp] (t4) at ({0+3*\dx},{0+1*\tyl+1*\tyu}) {};
    \node [blp] (t5) at ({0+4*\dx},{0+1*\tyl+1*\tyu}) {};
    \node [lk, yshift={-\dy}] at (m1) {};
    \node [lk, yshift={-\dy}] at (m3) {};
    \node [lk, yshift={-\dy}] at (m4) {};
    \node [lk, yshift={-\dy}] at (t1) {};
    \node [lk, yshift={-\dy}] at (t2) {};
    \node [lk, yshift={-\dy}] at (t3) {};        
    \draw [->] (b4) --++ (0,{\dy});
    \draw [->] (m1) --++ (0,{\dy});
    \draw [->] (t4) --++ (0,{-\dy});        
    \draw (b2) -- (m2);
    \draw (b5) -- (m5);
    \draw (b6) -- (m6);
    \draw (m5) -- (t5);    
    \draw (b1) --++(0,{\dy}) -| (b3);
    \draw (m1) --++(0,{-\dy}) -| (m3);
    \draw (m3) --++(0,{-\dy}) -| (m4);
    \draw (m7) --++(0,{-\dy}) -| (m8);        
    \draw (m2) --++(0,{\dy}) -| (m4);
    \draw (m3) --++(0,{2*\dy}) -| (m6);
    \draw (m7) --++(0,{\dy}) -| (m8);
    \draw (t1) --++(0,{-\dy}) -| (t2);
    \draw (t2) --++(0,{-\dy}) -| (t3);
    \node at ({-\dx},{\tyl+0.5*\tyu}) {$p'$};
    \node at ({-\dx},{0.5*\tyl}) {$p$};    
  \end{tikzpicture}
  $=$
            \begin{tikzpicture}[baseline=2cm*0.666]
    \def\scp{0.666}
    \def\linksize{\scp*0.075cm}
    \def\pointsize{\scp*0.25cm}
    \def\dd{\scp*0.5cm}
    \def\dx{\scp*1cm}
    \def\cx{\scp*0.3cm}
    \def\txu{5*\dx}    
    \def\txl{5*\dx}
    \def\dy{\scp*1cm}
    \def\cy{\scp*0.3cm}
    \def\ty{4*\dy}
    \tikzset{whp/.style={circle, inner sep=0pt, text width={\pointsize}, draw=black, fill=white}}
    \tikzset{blp/.style={circle, inner sep=0pt, text width={\pointsize}, draw=black, fill=black}}
    \tikzset{lk/.style={regular polygon, regular polygon sides=4, inner sep=0pt, text width={\linksize}, draw=black, fill=black}}
    \draw[dotted] ({0-\dd},{0}) -- ({\txl+\dd},{0});
    \draw[dotted] ({0-\dd},{\ty}) -- ({\txl+\dd},{\ty});
    \node [whp] (b1) at ({0+0*\dx},{0+0*\ty}) {};
    \node [whp] (b2) at ({0+1*\dx},{0+0*\ty}) {};
    \node [blp] (b3) at ({0+2*\dx},{0+0*\ty}) {};
    \node [whp] (b4) at ({0+3*\dx},{0+0*\ty}) {};
    \node [blp] (b5) at ({0+4*\dx},{0+0*\ty}) {};
    \node [whp] (b6) at ({0+5*\dx},{0+0*\ty}) {};    
    \node [blp] (m1) at ({0+0*\dx},{0+1*\ty}) {};
    \node [whp] (m2) at ({0+1*\dx},{0+1*\ty}) {};
    \node [whp] (m3) at ({0+2*\dx},{0+1*\ty}) {};
    \node [blp] (m4) at ({0+3*\dx},{0+1*\ty}) {};
    \node [whp] (m5) at ({0+4*\dx},{0+1*\ty}) {};
    \node [lk, yshift={-\dy}] at (m1) {};
    \node [lk, yshift={-\dy}] at (m2) {};
    \node [lk, yshift={-\dy}] at (m3) {};
    \draw [->] (b4) --++ (0,{\dy});
    \draw [->] (m4) --++ (0,{-\dy});        
    \draw (b5) -- (m5);
    \draw (b1) --++(0,{\dy}) -| (b3);
    \draw (b2) --++(0,{2*\dy}) -| (b6);
    \draw (m1) --++(0,{-\dy}) -| (m2);
    \draw (m2) --++(0,{-\dy}) -| (m3);
    \node at ({\txl+\dx},{0.5*\ty}) {$pp'$};
  \end{tikzpicture}
  \end{mycenter}
\vspace{-0.5em}
\subsection{Categories of Two-Colored Partitions}
We call a set $\mc C\subseteq \Cp$ a \emph{category} if it contains the partitions $\emptyset$, $\PartIdenW$, $\PartIdenB$, $\PartIdenLoBW$ and $\PartIdenLoWB$ and is closed under tensor products, involution and composition of composable pairings (\cite[Section~1.3]{TaWe15a}, building on \cite[Definition~2.2]{BaSp09}). The set of all categories of two-colored partitions is denoted by $\cotcp$.
\par
For every set $\mc G\subseteq \Cp$ we  write $\langle \mc G\rangle$ for the smallest category (with respect to $\subseteq$) which contains $\mc G$. We say that $\mc G$ \emph{generates} $\langle \mc G\rangle$. If $\mc G=\{p\}$ for some $p\in \Cp$, we slightly abuse notation by writing $\langle p\rangle$ instead of $\langle \{p\}\rangle$ for $\langle \mc G\rangle$. Also, we mix the two, writing, $\langle \mc G,q\rangle$ for $q\in \Cp$ instead of $\langle \mc G\cup \{p\}\rangle$.\pagebreak

\begin{definition}{{\normalfont \cite[Definition~2.2]{TaWe15a}}}
  \label{definition:cases}
  Let $\mc C\subseteq \Cp$  be a category. We say  $\mc C$ is
  \begin{enumerate}[label=(\alph*)]
  \item \ldots \emph{case $\mc O$} if $\PartSinglesWBTensor\notin \mc C$ and $\PartFourWBWB\notin \mc C$,
  \item \ldots \emph{case $\mc B$} if $\PartSinglesWBTensor\in \mc C$ and $\PartFourWBWB\notin \mc C$,
      \item \ldots \emph{case $\mc S$} if $\PartSinglesWBTensor\in \mc C$ and $\PartFourWBWB\in \mc C$,
  \item \ldots \emph{case $\mc H$} if $\PartSinglesWBTensor\notin \mc C$ and $\PartFourWBWB\in \mc C$.
  \end{enumerate}
\end{definition}
In this series of articles, we are only interested in the first three cases; Compare with the classification in the uncolored case, where also the corresponding Case~$\mc H$, i.e.\ $\UCPartSinglesTensor\notin \mc C$ or $\UCPartFour\in\mc C$ for categories $\mc C\subseteq \mc{P}$, is  most complex (\cite{RaWe15a} and \cite{RaWe13}).
\begin{definition}
   Let $\mc C\subseteq \Cp$ be a category.
   \begin{enumerate}[label=(\alph*)]
     \label{definition:non-hyperoctahedral}
  \item\label{definition:non-hyperoctahedral-1}  Equivalently to saying $\mc C$ is case~$\mc H$, we also call $\mc C$ \emph{hy\-per\-oc\-ta\-he\-dral}.
  \item\label{definition:non-hyperoctahedral-2} Hence, if instead $\mc C$ is case $\mc O$, $\mc B$ or $\mc S$, we say that $\mc C$ is \emph{non-hy\-per\-oc\-ta\-he\-dral}.
    \item\label{definition:non-hyperoctahedral-3} The set of all non-hyperoctahedral categories is denoted by $\nhoc$.
  \end{enumerate}
\end{definition}

\subsection{Composite Category Operations}
The basic category operations can be utilized to construct further generic transformations.
\par
Categories are closed under four \emph{basic rotations}: We obtain the partition $p^\rcurvearrowdown$ from $p\in \Cp$ by removing the leftmost upper point $\alpha$ and adding a new point $\beta$ to the left of the leftmost lower point, assigning to $\beta$ the inverse color of $\alpha$ and also replacing $\alpha$ by $\beta$ as far as the blocks are concerned. We say that $\alpha$ is \emph{rotated down}. Similarly, we can rotate down the rightmost upper point by inverting its color and appending it to the lower row, producing $p^\lcurvearrowdown$. And the rotations $p^\lcurvearrowup$ and $p^\rcurvearrowup$ result form reversing these procedures and \emph{rotating up} the leftmost respectively rightmost point. We call $p^{\circlearrowright}\eqpd (p^\lcurvearrowup)^{\lcurvearrowdown}$  the \emph{clockwise} and  $p^{\circlearrowleft}\eqpd (p^\rcurvearrowdown)^{\rcurvearrowup}$ the \emph{counter-clockwise cyclic rotation} of $p$.
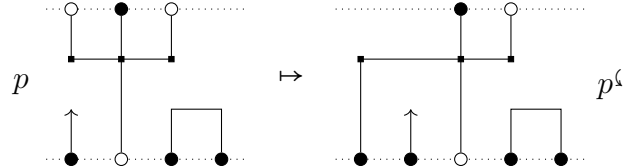
    \begin{mycenter}[0.5em]
                  \begin{tikzpicture}[baseline=0.666*1.5cm]
    \def\scp{0.666}
    \def\linksize{\scp*0.075cm}
    \def\pointsize{\scp*0.25cm}
    \def\dd{\scp*0.5cm}
    \def\dx{\scp*1cm}
    \def\cx{\scp*0.3cm}
    \def\txu{2*\dx}    
    \def\txl{3*\dx}
    \def\dy{\scp*1cm}
    \def\cy{\scp*0.3cm}
    \def\ty{3*\dy}
    \tikzset{whp/.style={circle, inner sep=0pt, text width={\pointsize}, draw=black, fill=white}}
    \tikzset{blp/.style={circle, inner sep=0pt, text width={\pointsize}, draw=black, fill=black}}
    \tikzset{lk/.style={regular polygon, regular polygon sides=4, inner sep=0pt, text width={\linksize}, draw=black, fill=black}}
    \draw[dotted] ({0-\dd},{0}) -- ({\txl+\dd},{0});
    \draw[dotted] ({0-\dd},{\ty}) -- ({\txl+\dd},{\ty});
    \node [blp] (l1) at ({0+0*\dx},{0+0*\ty}) {};
    \node [whp] (l2) at ({0+1*\dx},{0+0*\ty}) {};
    \node [blp] (l3) at ({0+2*\dx},{0+0*\ty}) {};
    \node [blp] (l4) at ({0+3*\dx},{0+0*\ty}) {};
    \node [whp] (u1) at ({0+0*\dx},{0+1*\ty}) {};
    \node [blp] (u2) at ({0+1*\dx},{0+1*\ty}) {};
    \node [whp] (u3) at ({0+2*\dx},{0+1*\ty}) {};
    \node [lk, yshift={-\dy}] at (u1) {};
    \node [lk, yshift={-\dy}] at (u2) {};
    \node [lk, yshift={-\dy}] at (u3) {};    
    \draw [->] (l1) --++ (0,{\dy});
    \draw (l2) -- (u2);
    \draw (u1) --++(0,{-\dy}) -|(u3);
    \draw (l3) --++ (0,{\dy}) -| (l4);
    \node at ({-\dx},{0.5*\ty}) {$p$};
  \end{tikzpicture}\quad$\mapsto$\quad
                  \begin{tikzpicture}[baseline=0.666*1.5cm]
    \def\scp{0.666}
    \def\linksize{\scp*0.075cm}
    \def\pointsize{\scp*0.25cm}
    \def\dd{\scp*0.5cm}
    \def\dx{\scp*1cm}
    \def\cx{\scp*0.3cm}
    \def\txu{2*\dx}    
    \def\txl{3*\dx}
    \def\dy{\scp*1cm}
    \def\cy{\scp*0.3cm}
    \def\ty{3*\dy}
    \tikzset{whp/.style={circle, inner sep=0pt, text width={\pointsize}, draw=black, fill=white}}
    \tikzset{blp/.style={circle, inner sep=0pt, text width={\pointsize}, draw=black, fill=black}}
    \tikzset{lk/.style={regular polygon, regular polygon sides=4, inner sep=0pt, text width={\linksize}, draw=black, fill=black}}
    \draw[dotted] ({-\dx-\dd},{0}) -- ({\txl+\dd},{0});
    \draw[dotted] ({-\dx-\dd},{\ty}) -- ({\txl+\dd},{\ty});
    \node [blp] (l0) at ({0-1*\dx},{0+0*\ty}) {};    
    \node [blp] (l1) at ({0+0*\dx},{0+0*\ty}) {};
    \node [whp] (l2) at ({0+1*\dx},{0+0*\ty}) {};
    \node [blp] (l3) at ({0+2*\dx},{0+0*\ty}) {};
    \node [blp] (l4) at ({0+3*\dx},{0+0*\ty}) {};
    \node [blp] (u2) at ({0+1*\dx},{0+1*\ty}) {};
    \node [whp] (u3) at ({0+2*\dx},{0+1*\ty}) {};
    \node [lk, yshift={2*\dy}] at (l0) {};
    \node [lk, yshift={-\dy}] at (u2) {};
    \node [lk, yshift={-\dy}] at (u3) {};    
    \draw [->] (l1) --++ (0,{\dy});
    \draw (l2) -- (u2);
    \draw (l0) --++(0,{2*\dy}) -|(u3);
    \draw (l3) --++ (0,{\dy}) -| (l4);
    \node at ({\txl+\dx},{0.5*\ty}) {$p^\rcurvearrowdown$};
  \end{tikzpicture}  
  \end{mycenter}
\par
From $p$ we obtain the \emph{reflection} $\hat p$ by reversing the native total orders on both rows. The \emph{color inversion} $\overline{p}$ of $p$ is constructed by inverting the native coloring.
And the \emph{verticolor reflection} $\tilde p$ is the color inversion of the reflection of $p$. Categories are closed under verticolor reflection but generally neither under reflection nor color inversion. 
  
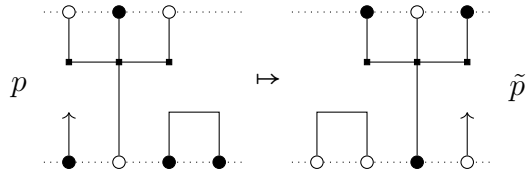
\begin{mycenter}[0.5em]
            \begin{tikzpicture}[baseline=0.666*1.5cm]
    \def\scp{0.666}
    \def\linksize{\scp*0.075cm}
    \def\pointsize{\scp*0.25cm}
    \def\dd{\scp*0.5cm}
    \def\dx{\scp*1cm}
    \def\cx{\scp*0.3cm}
    \def\txu{2*\dx}    
    \def\txl{3*\dx}
    \def\dy{\scp*1cm}
    \def\cy{\scp*0.3cm}
    \def\ty{3*\dy}
    \tikzset{whp/.style={circle, inner sep=0pt, text width={\pointsize}, draw=black, fill=white}}
    \tikzset{blp/.style={circle, inner sep=0pt, text width={\pointsize}, draw=black, fill=black}}
    \tikzset{lk/.style={regular polygon, regular polygon sides=4, inner sep=0pt, text width={\linksize}, draw=black, fill=black}}
    \draw[dotted] ({0-\dd},{0}) -- ({\txl+\dd},{0});
    \draw[dotted] ({0-\dd},{\ty}) -- ({\txl+\dd},{\ty});
    \node [blp] (l1) at ({0+0*\dx},{0+0*\ty}) {};
    \node [whp] (l2) at ({0+1*\dx},{0+0*\ty}) {};
    \node [blp] (l3) at ({0+2*\dx},{0+0*\ty}) {};
    \node [blp] (l4) at ({0+3*\dx},{0+0*\ty}) {};
    \node [whp] (u1) at ({0+0*\dx},{0+1*\ty}) {};
    \node [blp] (u2) at ({0+1*\dx},{0+1*\ty}) {};
    \node [whp] (u3) at ({0+2*\dx},{0+1*\ty}) {};
    \node [lk, yshift={-\dy}] at (u1) {};
    \node [lk, yshift={-\dy}] at (u2) {};
    \node [lk, yshift={-\dy}] at (u3) {};    
    \draw [->] (l1) --++ (0,{\dy});
    \draw (l2) -- (u2);
    \draw (u1) --++(0,{-\dy}) -|(u3);
    \draw (l3) --++ (0,{\dy}) -| (l4);
    \node at ({-\dx},{0.5*\ty}) {$p$};    
  \end{tikzpicture}
  $\mapsto$
  \begin{tikzpicture}[baseline=0.666*1.5cm]
    \def\scp{0.666}
    \def\linksize{\scp*0.075cm}
    \def\pointsize{\scp*0.25cm}
    \def\dd{\scp*0.5cm}
    \def\dx{\scp*1cm}
    \def\cx{\scp*0.3cm}
    \def\txu{2*\dx}    
    \def\txl{3*\dx}
    \def\dy{\scp*1cm}
    \def\cy{\scp*0.3cm}
    \def\ty{3*\dy}
    \tikzset{whp/.style={circle, inner sep=0pt, text width={\pointsize}, draw=black, fill=white}}
    \tikzset{blp/.style={circle, inner sep=0pt, text width={\pointsize}, draw=black, fill=black}}
    \tikzset{lk/.style={regular polygon, regular polygon sides=4, inner sep=0pt, text width={\linksize}, draw=black, fill=black}}
    \draw[dotted] ({0-\dd},{0}) -- ({\txl+\dd},{0});
    \draw[dotted] ({0-\dd},{\ty}) -- ({\txl+\dd},{\ty});
    \node [whp] (l1) at ({0+0*\dx},{0+0*\ty}) {};
    \node [whp] (l2) at ({0+1*\dx},{0+0*\ty}) {};
    \node [blp] (l3) at ({0+2*\dx},{0+0*\ty}) {};
    \node [whp] (l4) at ({0+3*\dx},{0+0*\ty}) {};
    \node [blp] (u1) at ({0+1*\dx},{0+1*\ty}) {};
    \node [whp] (u2) at ({0+2*\dx},{0+1*\ty}) {};
    \node [blp] (u3) at ({0+3*\dx},{0+1*\ty}) {};
    \node [lk, yshift={-\dy}] at (u1) {};
    \node [lk, yshift={-\dy}] at (u2) {};
    \node [lk, yshift={-\dy}] at (u3) {};    
    \draw [->] (l4) --++ (0,{\dy});
    \draw (l3) -- (u2);
    \draw (u1) --++(0,{-\dy}) -|(u3);
    \draw (l1) --++ (0,{\dy}) -| (l2);
    \node at ({\txl+\dx},{0.5*\ty}) {$\tilde p$};    
  \end{tikzpicture}
  \end{mycenter}

 \par
  Lastly, categories are closed under erasing turns: A \emph{turn} is a consecutive neutral set of size two. The \emph{erasing} of a set $S$ of points from $p$ is the partition $E(p,S)$ obtained by removing $S$ and combining all the points whose blocks had a non-empty intersection with $S$ into one block.\pagebreak
  
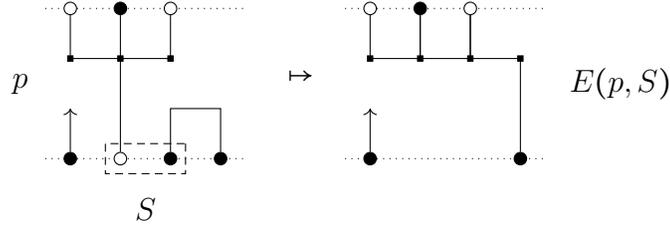
\begin{mycenter}[0.5em]
                \begin{tikzpicture}[baseline=0.666*1.5cm]
    \def\scp{0.666}
    \def\linksize{\scp*0.075cm}
    \def\pointsize{\scp*0.25cm}
    \def\dd{\scp*0.5cm}
    \def\dx{\scp*1cm}
    \def\cx{\scp*0.3cm}
    \def\txu{2*\dx}    
    \def\txl{3*\dx}
    \def\dy{\scp*1cm}
    \def\cy{\scp*0.3cm}
    \def\ty{3*\dy}
    \tikzset{whp/.style={circle, inner sep=0pt, text width={\pointsize}, draw=black, fill=white}}
    \tikzset{blp/.style={circle, inner sep=0pt, text width={\pointsize}, draw=black, fill=black}}
    \tikzset{lk/.style={regular polygon, regular polygon sides=4, inner sep=0pt, text width={\linksize}, draw=black, fill=black}}
    \draw[dotted] ({0-\dd},{0}) -- ({\txl+\dd},{0});
    \draw[dotted] ({0-\dd},{\ty}) -- ({\txl+\dd},{\ty});
    \node [blp] (l1) at ({0+0*\dx},{0+0*\ty}) {};
    \node [whp] (l2) at ({0+1*\dx},{0+0*\ty}) {};
    \node [blp] (l3) at ({0+2*\dx},{0+0*\ty}) {};
    \node [blp] (l4) at ({0+3*\dx},{0+0*\ty}) {};
    \node [whp] (u1) at ({0+0*\dx},{0+1*\ty}) {};
    \node [blp] (u2) at ({0+1*\dx},{0+1*\ty}) {};
    \node [whp] (u3) at ({0+2*\dx},{0+1*\ty}) {};
    \node [lk, yshift={-\dy}] at (u1) {};
    \node [lk, yshift={-\dy}] at (u2) {};
    \node [lk, yshift={-\dy}] at (u3) {};    
    \draw [->] (l1) --++ (0,{\dy});
    \draw (l2) -- (u2);
    \draw (u1) --++(0,{-\dy}) -|(u3);
    \draw (l3) --++ (0,{\dy}) -| (l4);
    \draw [densely dashed] ($(l3)+({\cx},{\cy})$) |- ($(l2)+({-\cx},{-\cy})$) |- cycle;
    \node at ({0.5*\txl},{-\dy}) {$S$};
    \node at ({-\dx},{0.5*\ty}) {$p$};    
  \end{tikzpicture}
  \quad$\mapsto$\quad
                \begin{tikzpicture}[baseline=0.666*1.5cm]
    \def\scp{0.666}
    \def\linksize{\scp*0.075cm}
    \def\pointsize{\scp*0.25cm}
    \def\dd{\scp*0.5cm}
    \def\dx{\scp*1cm}
    \def\cx{\scp*0.3cm}
    \def\txu{2*\dx}    
    \def\txl{3*\dx}
    \def\dy{\scp*1cm}
    \def\cy{\scp*0.3cm}
    \def\ty{3*\dy}
    \tikzset{whp/.style={circle, inner sep=0pt, text width={\pointsize}, draw=black, fill=white}}
    \tikzset{blp/.style={circle, inner sep=0pt, text width={\pointsize}, draw=black, fill=black}}
    \tikzset{lk/.style={regular polygon, regular polygon sides=4, inner sep=0pt, text width={\linksize}, draw=black, fill=black}}
    \draw[dotted] ({0-\dd},{0}) -- ({\txl+\dd},{0});
    \draw[dotted] ({0-\dd},{\ty}) -- ({\txl+\dd},{\ty});
    \node [blp] (l1) at ({0+0*\dx},{0+0*\ty}) {};
    \node [blp] (l4) at ({0+3*\dx},{0+0*\ty}) {};
    \node [whp] (u1) at ({0+0*\dx},{0+1*\ty}) {};
    \node [blp] (u2) at ({0+1*\dx},{0+1*\ty}) {};
    \node [whp] (u3) at ({0+2*\dx},{0+1*\ty}) {};
    \node [lk, yshift={-\dy}] at (u1) {};
    \node [lk, yshift={-\dy}] at (u2) {};
    \node [lk, yshift={-\dy}] at (u3) {};
    \node [lk, yshift={2*\dy}] at (l4) {};        
    \draw [->] (l1) --++ (0,{\dy});
    \draw (u1) --++(0,{-\dy}) -|(u2);
    \draw (u2) --++(0,{-\dy}) -|(u3);
    \draw (u3) --++(0,{-\dy}) -|(l4);    
    \node at ({\txl+2*\dx},{0.5*\ty}) {$E(p,S)$};    
  \end{tikzpicture}  
  \end{mycenter}

\vspace{-1em}
\subsection{Alternative Characterization of Categories}
Using the composite operations, we can give a helpful characterization of categories based on the idea in the proof of \cite[Lemma~3.6]{RaWe13b}.
 
\begin{lemma}[Alternative Characterization]
  \label{lemma:characterization-categories}
  A set $\mc C\subseteq \Cp$ is a category of two-colored partitions if and only if
  \begin{enumerate}[label=(\roman*)]
  \item\label{lemma:characterization-categories-1} $\PartIdenW\in \mc C$ and
  \item\label{lemma:characterization-categories-2} $\mc C$ is closed under tensor products, basic rotations, verticolor reflection and  erasing turns.
  \end{enumerate}
   
    \end{lemma}
    \begin{proof}
      Suppose $\mc C$ satisfies \ref{lemma:characterization-categories-1} and  \ref{lemma:characterization-categories-2}.
      Since $\PartIdenW^\lcurvearrowdown=\PartIdenLoWB$, since $\PartIdenLoWB^\lcurvearrowup=\PartIdenB$ and since $\PartIdenB^\lcurvearrowdown=\PartIdenLoBW$ and because $\mc C$ is closed under rotations, we find $\PartIdenB,\PartIdenLoWB,\PartIdenLoBW\in\mc C$. Erasing the only turn in $\PartIdenLoWB\in \mc C$, under which $\mc C$ is invariant, produces  $\emptyset\in \mc C$.
      \par
      For every partition $p\in \mc C$ with $k$ upper and $m$ lower points, the identity $p^\ast= ((\tilde p)^{\lcurvearrowup k})^{\lcurvearrowdown l}$ and the assumptions that $\mc C$ is stable under rotations and verticolor reflection proves that $p^\ast\in \mc C$. Hence, $\mc C$ is also involution-invariant.
      \par 
      Lastly, suppose that $(p,q)$ is a composable pairing from $\mc C$. We want to show  $r\eqpd pq\in \mc C$.  Let $q$ have $k$ upper and $l$ lower and let $p$ have $l$ upper and $m$ lower points. Since $\mc C$ is closed under rotations, it suffices to prove  $r^{\rcurvearrowup m}\in\mc C$. Let $(c_1,\ldots,c_l)$ for $c_1,\ldots,c_l\in \colors$ be the coloring of the lower row of $q$ left to right. Let $s$ be the tensor product of partitions from $\{\PartIdenW,\PartIdenB\}$ with lower row of coloring $(c_1,\ldots,c_l)$. Then, $(p,s)$ and $(s,q)$ are composable and $psq=pq=r$. The diagram below illustrates that the pairing $(s^{\rcurvearrowup l},q\otimes ((p^{\rcurvearrowdown l})^{\rcurvearrowup m}))$ is composable as well and that its composition yields the partition $r^{\rcurvearrowup m}$.
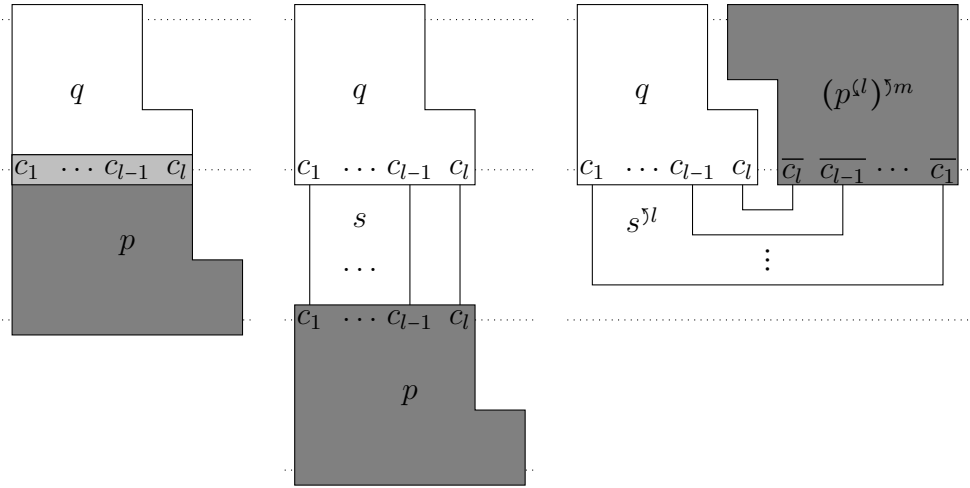
      \begin{mycenter}[1em]
    \begin{tikzpicture}[baseline=0.666*0cm]
    \def\scp{0.666}
    \def\linksize{\scp*0.075cm}
    \def\pointsize{\scp*0.25cm}
    \def\dd{\scp*0.5cm}
    \def\dx{\scp*1cm}
    \def\cx{\scp*0.3cm}
    \def\txu{4*\dx}    
    \def\txl{4*\dx}
    \def\dy{\scp*1cm}
    \def\cy{\scp*0.3cm}
    \def\ty{3*\dy}
    \tikzset{whp/.style={circle, inner sep=0pt, text width={\pointsize}, draw=black, fill=white}}
    \tikzset{blp/.style={circle, inner sep=0pt, text width={\pointsize}, draw=black, fill=black}}
    \tikzset{lk/.style={regular polygon, regular polygon sides=4, inner sep=0pt, text width={\linksize}, draw=black, fill=black}}
    \draw[dotted] ({0-\dd},{0}) -- ({\txl+\dd},{0});
    \draw[dotted] ({0-\dd},{\ty}) -- ({\txl+\dd},{\ty});
    \draw[dotted] ({0-\dd},{-\ty}) -- ({\txl+\dd},{-\ty});
    \coordinate (u1) at ({0*\dx},{1*\ty});
    \coordinate (u2) at ({2*\dx},{1*\ty});
    \coordinate (m1) at ({0*\dx},{0*\ty});
    \coordinate (m2) at ({2*\dx},{0*\ty});
    \coordinate (m3) at ({3*\dx},{0*\ty});
    \coordinate (l1) at ({0*\dx},{-1*\ty});
    \coordinate (l2) at ({4*\dx},{-1*\ty});    
    \draw [fill=white] ($(u1)+({-\cx},{\cy})$) -- ($(m1)+({-\cx},{-\cy})$) -- ($(m3)+({\cx},{-\cy})$) --++ (0,{0.5*\ty}) -| ($(u2)+({\cx},{\cy})$) -- cycle;
    \draw [fill=gray] ($(m1)+({-\cx},{\cy})$) -- ($(l1)+({-\cx},{-\cy})$) -- ($(l2)+({\cx},{-\cy})$) --++ (0,{0.5*\ty}) -| ($(m3)+({\cx},{\cy})$) -- cycle;
    \draw[fill=lightgray] ($(m1)+({-\cx},{-\cy})$) rectangle ($(m3)+({\cx},{\cy})$);
    \path (m1)-- node[pos=0.5] {$\ldots$}  (m2);
    \node at (m1) {$c_1$};
    \node at (m2) {$c_{l-1}$};
    \node at (m3) {$c_{l}$};
    \node at ({1*\dx},{0.5*\ty}) {$q$};
    \node at ({2*\dx},{-0.5*\ty}) {$p$};    
  \end{tikzpicture}\quad
    \begin{tikzpicture}[baseline=0.666*0cm]
    \def\scp{0.666}
    \def\linksize{\scp*0.075cm}
    \def\pointsize{\scp*0.25cm}
    \def\dd{\scp*0.5cm}
    \def\dx{\scp*1cm}
    \def\cx{\scp*0.3cm}
    \def\txu{4*\dx}    
    \def\txl{4*\dx}
    \def\dy{\scp*1cm}
    \def\cy{\scp*0.3cm}
    \def\ty{3*\dy}
    \tikzset{whp/.style={circle, inner sep=0pt, text width={\pointsize}, draw=black, fill=white}}
    \tikzset{blp/.style={circle, inner sep=0pt, text width={\pointsize}, draw=black, fill=black}}
    \tikzset{lk/.style={regular polygon, regular polygon sides=4, inner sep=0pt, text width={\linksize}, draw=black, fill=black}}
    \draw[dotted] ({0-\dd},{0}) -- ({\txl+\dd},{0});
    \draw[dotted] ({0-\dd},{\ty}) -- ({\txl+\dd},{\ty});
    \draw[dotted] ({0-\dd},{-\ty}) -- ({\txl+\dd},{-\ty});
    \draw[dotted] ({0-\dd},{-2*\ty}) -- ({\txl+\dd},{-2*\ty});    
    \coordinate (u1) at ({0*\dx},{1*\ty});
    \coordinate (u2) at ({2*\dx},{1*\ty});
    \coordinate (m1) at ({0*\dx},{0*\ty});
    \coordinate (m2) at ({2*\dx},{0*\ty});
    \coordinate (m3) at ({3*\dx},{0*\ty});
    \coordinate (l1) at ({0*\dx},{-1*\ty});
    \coordinate (l2) at ({2*\dx},{-1*\ty});
    \coordinate (l3) at ({3*\dx},{-1*\ty});
    \coordinate (b1) at ({0*\dx},{-2*\ty});
    \coordinate (b2) at ({4*\dx},{-2*\ty});    
    \draw (l1.north) -- (m1.south);
    \draw (l2.north) -- (m2.south);
    \draw (l3.north) -- (m3.south);
    \draw [fill=white] ($(u1)+({-\cx},{\cy})$) -- ($(m1)+({-\cx},{-\cy})$) -- ($(m3)+({\cx},{-\cy})$) --++ (0,{0.5*\ty}) -| ($(u2)+({\cx},{\cy})$) -- cycle;
    \draw [fill=gray] ($(l1)+({-\cx},{\cy})$) -- ($(b1)+({-\cx},{-\cy})$) -- ($(b2)+({\cx},{-\cy})$) --++ (0,{0.5*\ty}) -| ($(l3)+({\cx},{\cy})$) -- cycle;
    \path (m1)-- node[pos=0.5] {$\ldots$}  (m2);
    \path (l1)-- node[pos=0.5] {$\ldots$}  (l2);
    \path ($(l1)+(0,{0.333*\ty})$) -- node[pos=0.5] {$\ldots$}  ($(l2)+(0,{0.333*\ty})$);
    \node at (m1) {$c_1$};
    \node at (m2) {$c_{l-1}$};
    \node at (m3) {$c_{l}$};
    \node at (l1) {$c_1$};
    \node at (l2) {$c_{l-1}$};
    \node at (l3) {$c_{l}$};    
    \node at ({1*\dx},{0.5*\ty}) {$q$};
    \node at ({2*\dx},{-1.5*\ty}) {$p$};
    \node at ({1*\dx},{-0.333*\ty}) {$s$};
  \end{tikzpicture}\quad
      \begin{tikzpicture}[baseline=0.666*0cm]
    \def\scp{0.666}
    \def\linksize{\scp*0.075cm}
    \def\pointsize{\scp*0.25cm}
    \def\dd{\scp*0.5cm}
    \def\dx{\scp*1cm}
    \def\cx{\scp*0.3cm}
    \def\txu{7*\dx}    
    \def\txl{7*\dx}
    \def\dy{\scp*1cm}
    \def\cy{\scp*0.3cm}
    \def\ty{3*\dy}
    \tikzset{whp/.style={circle, inner sep=0pt, text width={\pointsize}, draw=black, fill=white}}
    \tikzset{blp/.style={circle, inner sep=0pt, text width={\pointsize}, draw=black, fill=black}}
    \tikzset{lk/.style={regular polygon, regular polygon sides=4, inner sep=0pt, text width={\linksize}, draw=black, fill=black}}
    \draw[dotted] ({0-\dd},{0}) -- ({\txl+\dd},{0});
    \draw[dotted] ({0-\dd},{\ty}) -- ({\txl+\dd},{\ty});
    \draw[dotted] ({0-\dd},{-\ty}) -- ({\txl+\dd},{-\ty});
    \coordinate (u1) at ({0*\dx},{1*\ty});
    \coordinate (u2) at ({2*\dx},{1*\ty});
    \coordinate (u3) at ({7*\dx},{1*\ty});
    \coordinate (m1) at ({0*\dx},{0*\ty});
    \coordinate (m2) at ({2*\dx},{0*\ty});
    \coordinate (m3) at ({3*\dx},{0*\ty});
    \coordinate (m4) at ({4*\dx},{0*\ty});    
    \coordinate (m5) at ({5*\dx},{0*\ty});
    \coordinate (m6) at ({7*\dx},{0*\ty});        
    \draw (m1.south) -- ++ (0,{-\cy-2*\dy}) -| (m6.south);
    \draw (m2.south) -- ++ (0,{-\cy-1*\dy}) -| (m5.south);
    \draw (m3.south) -- ++ (0,{-\cy-0.5*\dy}) -| (m4.south);
    \draw [fill=white] ($(u1)+({-\cx},{\cy})$) -- ($(m1)+({-\cx},{-\cy})$) -- ($(m3)+({\cx},{-\cy})$) -- ++ (0,{0.5*\ty}) -| ($(u2)+({\cx},{\cy})$) -- cycle;
    \draw [fill=gray] ($(m4)+({-\cx},{-\cy})$) -- ($(m6)+({\cx},{-\cy})$) -- ($(u3)+({\cx},{\cy})$) -- ($(u2)+({\dy-\cx},{\cy})$) -- ++ (0,{-0.5*\ty}) -| cycle;
    \path (m1)-- node[pos=0.5] {$\ldots$}  (m2);
    \path (m5)-- node[pos=0.5] {$\ldots$}  (m6);
    \node at (m1) {$c_1$};
    \node at (m2) {$c_{l-1}$};
    \node at (m3) {$c_{l}$};
    \node at (m4) {$\overline{c_l}$};
    \node at (m5) {$\overline{c_{l-1}}$};
    \node at (m6) {$\overline{c_{1}}$};    
    \node at ({1*\dx},{0.5*\ty}) {$q$};
    \node at ({5.5*\dx},{0.5*\ty}) {$(p^{\rcurvearrowdown l})^{\rcurvearrowup m}$};
    \node at ({1*\dx},{-1*\dy}) {$s^{\rcurvearrowup l}$};
    \path ({0.5*\txl},{-\cy-\dy})-- node[pos=0.5] {$\vdots$}  ({0.5*\txl},{-\cy-2*\dy});    
  \end{tikzpicture}  
\end{mycenter}

      \par
      Our assumptions guarantee $e_0\eqpd q\otimes ((p^{\rcurvearrowdown l})^{\rcurvearrowup m})\in\mc C$. Because $\mc C$ is assumed invariant under erasing  turns, if we define the turn $T_0\eqpd \{\lop{l},\lop{(l+1)}\}$ in $e_0\in \mc C$ and then for every $j=1,\ldots,l-1$ the turn $T_j\eqpd \{\lop{(l-j)},\lop{(l-j+1)}\}$ in $e_j\eqpd E(e_{j-1}, T_{j-1})\in \mc C$, then the partition $e_l\eqpd E(e_{l-1}, T_{l-1})\in \mc C$ is identical with $r^{\rcurvearrowup m}$ as the diagram below shows.
      
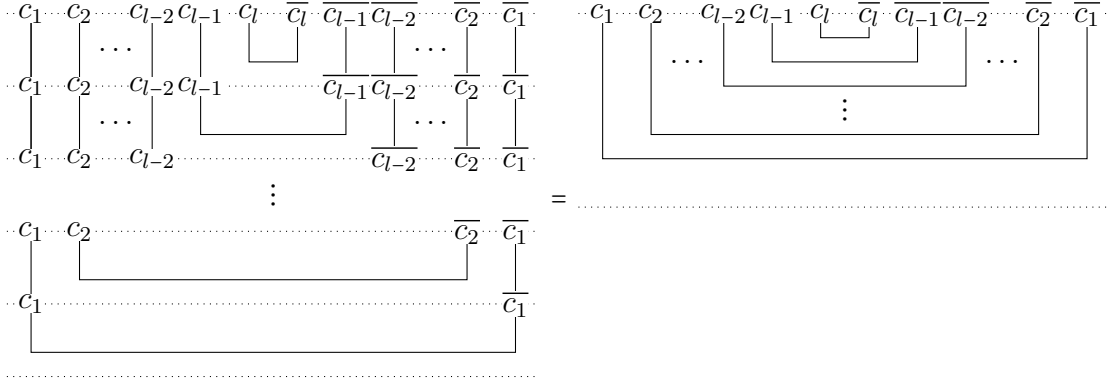
\begin{mycenter}[1em]
              \begin{tikzpicture}[baseline=0.644*-1cm]
    \def\scp{0.644}
    \def\linksize{\scp*0.075cm}
    \def\pointsize{\scp*0.25cm}
    \def\dd{\scp*0.5cm}
    \def\dx{\scp*1cm}
    \def\cx{\scp*0.3cm}
    \def\txu{10*\dx}    
    \def\txl{10*\dx}
    \def\dy{\scp*1cm}
    \def\cy{\scp*0.3cm}
    \def\ty{1.5*\dy}
    \tikzset{whp/.style={circle, inner sep=0pt, text width={\pointsize}, draw=black, fill=white}}
    \tikzset{blp/.style={circle, inner sep=0pt, text width={\pointsize}, draw=black, fill=black}}
    \tikzset{lk/.style={regular polygon, regular polygon sides=4, inner sep=0pt, text width={\linksize}, draw=black, fill=black}}
    \draw[dotted] ({\dx+0-\dd},{2*\ty}) -- ({\dx+\txl+\dd},{2*\ty});
    \draw[dotted] ({\dx+0-\dd},{1*\ty}) -- ({\dx+\txl+\dd},{1*\ty});    
    \draw[dotted] ({\dx+0-\dd},{0*\ty}) -- ({\dx+\txl+\dd},{0*\ty});
    \draw[dotted] ({\dx+0-\dd},{-1*\ty}) -- ({\dx+\txl+\dd},{-1*\ty});
    \draw[dotted] ({\dx+0-\dd},{-2*\ty}) -- ({\dx+\txl+\dd},{-2*\ty});
    \draw[dotted] ({\dx+0-\dd},{-3*\ty}) -- ({\dx+\txl+\dd},{-3*\ty});   
    \coordinate (t1) at ({1*\dx},{2*\ty});
    \coordinate (t2) at ({2*\dx},{2*\ty});
    \coordinate (t3) at ({3.5*\dx},{2*\ty});
    \coordinate (t4) at ({4.5*\dx},{2*\ty});
    \coordinate (t5) at ({5.5*\dx},{2*\ty});
    \coordinate (t6) at ({6.5*\dx},{2*\ty});
    \coordinate (t7) at ({7.5*\dx},{2*\ty});
    \coordinate (t8) at ({8.5*\dx},{2*\ty});
    \coordinate (t9) at ({10*\dx},{2*\ty});
    \coordinate (t10) at ({11*\dx},{2*\ty});
    \coordinate (u1) at ({1*\dx},{1*\ty});
    \coordinate (u2) at ({2*\dx},{1*\ty});
    \coordinate (u3) at ({3.5*\dx},{1*\ty});
    \coordinate (u4) at ({4.5*\dx},{1*\ty});
    \coordinate (u5) at ({5.5*\dx},{1*\ty});
    \coordinate (u6) at ({6.5*\dx},{1*\ty});
    \coordinate (u7) at ({7.5*\dx},{1*\ty});
    \coordinate (u8) at ({8.5*\dx},{1*\ty});
    \coordinate (u9) at ({10*\dx},{1*\ty});
    \coordinate (u10) at ({11*\dx},{1*\ty});
    \coordinate (m1) at ({1*\dx},{0*\ty});
    \coordinate (m2) at ({2*\dx},{0*\ty});
    \coordinate (m3) at ({3.5*\dx},{0*\ty});
    \coordinate (m4) at ({4.5*\dx},{0*\ty});
    \coordinate (m5) at ({5.5*\dx},{0*\ty});
    \coordinate (m6) at ({6.5*\dx},{0*\ty});
    \coordinate (m7) at ({7.5*\dx},{0*\ty});
    \coordinate (m8) at ({8.5*\dx},{0*\ty});
    \coordinate (m9) at ({10*\dx},{0*\ty});
    \coordinate (m10) at ({11*\dx},{0*\ty});
    \coordinate (l1) at ({1*\dx},{-1*\ty});
    \coordinate (l2) at ({2*\dx},{-1*\ty});
    \coordinate (l3) at ({3.5*\dx},{-1*\ty});
    \coordinate (l4) at ({4.5*\dx},{-1*\ty});
    \coordinate (l5) at ({5.5*\dx},{-1*\ty});
    \coordinate (l6) at ({6.5*\dx},{-1*\ty});
    \coordinate (l7) at ({7.5*\dx},{-1*\ty});
    \coordinate (l8) at ({8.5*\dx},{-1*\ty});
    \coordinate (l9) at ({10*\dx},{-1*\ty});
    \coordinate (l10) at ({11*\dx},{-1*\ty});
    \coordinate (b1) at ({1*\dx},{-2*\ty});
    \coordinate (b2) at ({2*\dx},{-2*\ty});
    \coordinate (b3) at ({3.5*\dx},{-2*\ty});
    \coordinate (b4) at ({4.5*\dx},{-2*\ty});
    \coordinate (b5) at ({5.5*\dx},{-2*\ty});
    \coordinate (b6) at ({6.5*\dx},{-2*\ty});
    \coordinate (b7) at ({7.5*\dx},{-2*\ty});
    \coordinate (b8) at ({8.5*\dx},{-2*\ty});
    \coordinate (b9) at ({10*\dx},{-2*\ty});
    \coordinate (b10) at ({11*\dx},{-2*\ty});
    \path ($(t2)+(0,{-0.5*\ty})$)-- node[pos=0.5] {$\ldots$} ($(t3)+(0,{-0.5*\ty})$);
    \path ($(u2)+(0,{-0.5*\ty})$)-- node[pos=0.5] {$\ldots$} ($(u3)+(0,{-0.5*\ty})$);
    \path ($(t8)+(0,{-0.5*\ty})$)-- node[pos=0.5] {$\ldots$} ($(t9)+(0,{-0.5*\ty})$);
    \path ($(u8)+(0,{-0.5*\ty})$)-- node[pos=0.5] {$\ldots$} ($(u9)+(0,{-0.5*\ty})$);
    \path ({\dx+0.5*\txl},{0*\ty}) -- node[pos=0.5] {$\vdots$} ({\dx+0.5*\txl},{-1*\ty});
    \draw (b1) -- (l1);
    \draw (b10) -- (l10);
    \draw (m1) -- (t1);
    \draw (m10) -- (t10);    
    \draw (m1) -- (t1);
    \draw (m2) -- (t2);
    \draw (m3) -- (t3);
    \draw (m8) -- (t8);    
    \draw (m9) -- (t9);    
    \draw (m10) -- (t10);
    \draw (u4) -- (t4);
    \draw (u7) -- (t7);
    \draw (t5) --++(0,{-\dy}) -| (t6);            
    \draw (u4) --++(0,{-\dy}) -| (u7);        
    \draw (l2) --++(0,{-\dy}) -| (l9);    
    \draw (b1) --++(0,{-\dy}) -| (b10);
    \node[fill=white, inner sep=0cm] at (t1) {$c_1$};
    \node[fill=white, inner sep=0cm] at (t2) {$c_2$};
    \node[fill=white, inner sep=0cm] at (t3) {$c_{l-2}$};
    \node[fill=white, inner sep=0cm] at (t4) {$c_{l-1}$};
    \node[fill=white, inner sep=0cm] at (t5) {$c_{l}$};
    \node[fill=white, inner sep=0cm] at (t10) {$\overline{c_1}$};
    \node[fill=white, inner sep=0cm] at (t9) {$\overline{c_2}$};
    \node[fill=white, inner sep=0cm] at (t8) {$\overline{c_{l-2}}$};
    \node[fill=white, inner sep=0cm] at (t7) {$\overline{c_{l-1}}$};
    \node[fill=white, inner sep=0cm] at (t6) {$\overline{c_{l}}$};
    \node[fill=white, inner sep=0cm] at (u1) {$c_1$};
    \node[fill=white, inner sep=0cm] at (u2) {$c_2$};
    \node[fill=white, inner sep=0cm] at (u3) {$c_{l-2}$};
    \node[fill=white, inner sep=0cm] at (u4) {$c_{l-1}$};
    \node[fill=white, inner sep=0cm] at (u10) {$\overline{c_1}$};
    \node[fill=white, inner sep=0cm] at (u9) {$\overline{c_2}$};
    \node[fill=white, inner sep=0cm] at (u8) {$\overline{c_{l-2}}$};
    \node[fill=white, inner sep=0cm] at (u7) {$\overline{c_{l-1}}$};
    \node[fill=white, inner sep=0cm] at (m1) {$c_1$};
    \node[fill=white, inner sep=0cm] at (m2) {$c_2$};
    \node[fill=white, inner sep=0cm] at (m3) {$c_{l-2}$};
    \node[fill=white, inner sep=0cm] at (m10) {$\overline{c_1}$};
    \node[fill=white, inner sep=0cm] at (m9) {$\overline{c_2}$};
    \node[fill=white, inner sep=0cm] at (m8) {$\overline{c_{l-2}}$};
    \node[fill=white, inner sep=0cm] at (l1) {$c_1$};
    \node[fill=white, inner sep=0cm] at (l2) {$c_2$};
    \node[fill=white, inner sep=0cm] at (l10) {$\overline{c_1}$};
    \node[fill=white, inner sep=0cm] at (l9) {$\overline{c_2}$};
    \node[fill=white, inner sep=0cm] at (b1) {$c_1$};
    \node[fill=white, inner sep=0cm] at (b10) {$\overline{c_1}$};
  \end{tikzpicture}
  $=$
              \begin{tikzpicture}[baseline=0.644*0cm]
    \def\scp{0.644}
    \def\linksize{\scp*0.075cm}
    \def\pointsize{\scp*0.25cm}
    \def\dd{\scp*0.5cm}
    \def\dx{\scp*1cm}
    \def\cx{\scp*0.3cm}
    \def\txu{10*\dx}    
    \def\txl{10*\dx}
    \def\dy{\scp*1cm}
    \def\cy{\scp*0.3cm}
    \def\ty{2*\dy}
    \tikzset{whp/.style={circle, inner sep=0pt, text width={\pointsize}, draw=black, fill=white}}
    \tikzset{blp/.style={circle, inner sep=0pt, text width={\pointsize}, draw=black, fill=black}}
    \tikzset{lk/.style={regular polygon, regular polygon sides=4, inner sep=0pt, text width={\linksize}, draw=black, fill=black}}
    \draw[dotted] ({\dx+0-\dd},{2*\ty}) -- ({\dx+\txl+\dd},{2*\ty});
    \draw[dotted] ({\dx+0-\dd},{0*\ty}) -- ({\dx+\txl+\dd},{0*\ty});
    \coordinate (t1) at ({1*\dx},{2*\ty});
    \coordinate (t2) at ({2*\dx},{2*\ty});
    \coordinate (t3) at ({3.5*\dx},{2*\ty});
    \coordinate (t4) at ({4.5*\dx},{2*\ty});
    \coordinate (t5) at ({5.5*\dx},{2*\ty});
    \coordinate (t6) at ({6.5*\dx},{2*\ty});
    \coordinate (t7) at ({7.5*\dx},{2*\ty});
    \coordinate (t8) at ({8.5*\dx},{2*\ty});
    \coordinate (t9) at ({10*\dx},{2*\ty});
    \coordinate (t10) at ({11*\dx},{2*\ty});            
    \path ($(t2)+(0,{-0.5*\ty})$)-- node[pos=0.5] {$\ldots$} ($(t3)+(0,{-0.5*\ty})$);
    \path ($(t8)+(0,{-0.5*\ty})$)-- node[pos=0.5] {$\ldots$} ($(t9)+(0,{-0.5*\ty})$);
    \path ({\dx+0.5*\txl},{2*\ty}) -- node[pos=0.5] {$\vdots$} ({\dx+0.5*\txl},{0*\ty});
    \draw (t5) --++(0,{-1*0.5*\dy}) -| (t6);
    \draw (t4) --++(0,{-2*0.5*\dy}) -| (t7);
    \draw (t3) --++(0,{-3*0.5*\dy}) -| (t8);
    \draw (t2) --++(0,{-5*0.5*\dy}) -| (t9);
    \draw (t1) --++(0,{-6*0.5*\dy}) -| (t10);    
    \node[fill=white, inner sep=0cm] at (t1) {$c_1$};
    \node[fill=white, inner sep=0cm] at (t2) {$c_2$};
    \node[fill=white, inner sep=0cm] at (t3) {$c_{l-2}$};
    \node[fill=white, inner sep=0cm] at (t4) {$c_{l-1}$};
    \node[fill=white, inner sep=0cm] at (t5) {$c_{l}$};
    \node[fill=white, inner sep=0cm] at (t10) {$\overline{c_1}$};
    \node[fill=white, inner sep=0cm] at (t9) {$\overline{c_2}$};
    \node[fill=white, inner sep=0cm] at (t8) {$\overline{c_{l-2}}$};
    \node[fill=white, inner sep=0cm] at (t7) {$\overline{c_{l-1}}$};
    \node[fill=white, inner sep=0cm] at (t6) {$\overline{c_{l}}$};
  \end{tikzpicture}
  \end{mycenter}
      \par
      Thus we have deduced $r^{\rcurvearrowup m}\in \mc C$ as claimed. This proves that $\mc C$ is closed under composition of composable pairs and thus a category. 
    \end{proof}

\section{\texorpdfstring{The Sets $\mc R_Q$}{The Sets R(Q)}}
\label{section:definition-of-the-categories}

In the following, we will define in several steps an index set $\mathsf Q$ and for each $Q\in \mathsf Q$  a set $\mc R_Q\subseteq \Cp$ of partitions. The aim will be to show that each of these constitutes a non-hyperoctahedral category (see Theorem~\ref{theorem:main}, the main result of this article). Auxiliary objects $\mathsf L$ and $Z$ aid in defining $\mathsf Q$ and $(\mc R_Q)_{Q\in \mathsf Q}$.
        \begin{notation}
          \label{notation:power-set}
For every set $S$ denote its power set by $\mathfrak P(S)$.
        \end{notation}
        \begin{definition}
          \label{definition:parameter-domain}
          We define the \emph{parameter domain} $\mathsf L$ as the six-fold Cartesian product of $\mathfrak P(\integers)$: \[\displaystyle\mathsf L\eqpd \mathfrak{P}(\integers)\times\mathfrak{P}(\integers)\times \mathfrak{P}(\integers)\times \mathfrak{P}(\integers)\times \mathfrak{P}(\integers)\times \mathfrak{P}(\integers).\]
        \end{definition}
        \begin{definition}
          \label{definition:Z}
  Define the \emph{analyzer}  $Z: \, \mathfrak{P}(\Cp)\to \mathsf L$ by
        \begin{align*}
                    Z\eqpd (\, F,\, V,\,  \toco,\,  L,\,  K,\,  X\,)
        \end{align*}
where, for all $\mc S\subseteq \Cp$,
\begin{enumerate}[label=(\alph*)]
\item  \(F(\mc S)\eqpd \{\, |B| \mid p\in \mc S,\, B\text{ block of } p\}\) is the set of block sizes,
\item  \(V(\mc S)\eqpd  \{\,\sigma_p(B)\mid p\in \mc S,\, B\text{ block of }p\}\) is the set of block color sums,
\item \(\toco(\mc S)\eqpd \{\,\toco(p)\mid p\in \mc S\}\) is the set of total color sums,
\item $\begin{aligned}[t]
                L(\mc S)\eqpd \{\,\delta_p(\alpha_1,\alpha_2)\mid&\, p\in \mc S, \, B\text{ block of }p,\, \alpha_1,\alpha_2\in B,\, \alpha_1\neq \alpha_2,\\
                &\, ]\alpha_1,\alpha_2[_p\cap B=\emptyset,\, \sigma_p(\{\alpha_1,\alpha_2\})\neq 0\}
                \end{aligned}$ \\
is the set of color distances between any two subsequent legs of the \emph{same} block having the \emph{same} normalized color,
\item $\begin{aligned}[t]
    K(\mc S)\eqpd \{\,\delta_p(\alpha_1,\alpha_2)\mid\,& p\in \mc S, \, B\text{ block of }p,\, \alpha_1,\alpha_2\in B,\, \alpha_1\neq \alpha_2,\\
    &\, ]\alpha_1,\alpha_2[_p\cap B=\emptyset,\, \sigma_p(\{\alpha_1,\alpha_2\})= 0\}
    \end{aligned}$\\
is the set of color distances between any two subsequent legs of the \emph{same} block having \emph{different} normalized colors and
\item $\begin{aligned}[t]
    X(\mc S)\eqpd \{\,\delta_p(\alpha_1,\alpha_2) \mid \, & p\in \mc S,\, B_1,B_2\text{ blocks of }p, \, B_1\text{ crosses } B_2,\\
 &\,  \alpha_1\in B_1,\,\alpha_2\in B_2\}
    \end{aligned}$ \\
is the set of color distances between any two legs belonging to two \emph{crossing} blocks.
\end{enumerate}
\end{definition}        
        The parameter domain $\mathsf L$ and the analyzer $Z$ allow us to now define the following map which will later induce the announced family $(\mc R_Q)_{Q\in \mathsf Q}$.

        \begin{notation}
Given a family $(S_i)_{i\in I}$ of sets, we write $\leq$ for the product order on the Cartesian product  $\bigtimes_{i\in I}\mathfrak{P}(S_i)$ induced by the partial orders $\subseteq$ on the factors.
        \end{notation}
\begin{definition}
  \label{definition:parametrization}
        Define the \emph{parametrization} as the mapping
        \begin{align*}
          \mc R:\; \mathsf L\to \mathfrak{P}(\Cp),\; L \mapsto \mc R_L\eqpd \{p\in \Cp \mid Z(\{p\})\leq L\}.
        \end{align*}
      \end{definition}
      With $\mc R$ we can single out sets of partitions by placing restrictions  on the six aggregated combinatorial features of partitions listed above.
      \par
   \begin{notation}
     \label{integer-notations}
     \begin{enumerate}[label=(\alph*)]
            \item\label{integer-notations-2} For all $x,y\in \integers$ and $A,B\subseteq \integers$ write
         \begin{align*}
           xA\!+\!yB\eqpd \{xa+yb\mid a\in A, \,b\in B\}.
         \end{align*}
         Moreover, put then $xA-yB\eqpd xA+(-y)B$. Per $A=\{1\}$ expressions like $x+yB$  are defined as well, and per $x=1$ so are such like $A+yB$. 
       \item\label{integer-notations-3} Let  $\pm S\eqpd S\cup(-S)$ for all sets $S\subseteq \integers$.
       \item\label{integer-notations-4} For all $m\in \integers$ and $D\subseteq \integers$ define
         \begin{align*}
           D_m\eqpd (D\cup(m\!-\! D))+m\integers \quad\text{and}\quad D_m'\eqpd (D\cup(m\!-\!D)\cup \{0\})+m\integers.
         \end{align*}
                \item\label{integer-notations-5}   Use the abbreviations $\dwi{0}\eqpd\emptyset$ and $\dwi{k}\eqpd \{1,\ldots,k\}$ for all $k\in \pint$.
     \end{enumerate}
   \end{notation}

   \begin{definition}
     \label{definition:Q}
    Define the \emph{parameter range} $\mathsf Q$ as the subset of $\mathsf L$ comprising exactly all tuples
    \begin{align*}
      (f,v,s,l,k,x)
    \end{align*}
    listed below, where  $u\in \{0\}\cup \pint$, $m\in \pint$, $D\subseteq \{0\}\cup\dwi{\lfloor\frac{m}{2}\rfloor}$, where $E\subseteq \{0\}\cup\pint$ and where $N$ is a subsemigroup of $(\pint,+)$:\pagebreak
                    \begin{align*}
                      \begin{matrix}
                        f&v&s&l&k& x  \\ \hline \\[-0.85em]
                        \{2\} & \pm\{0, 2\} & 2um\integers & m\integers & m\integers & \integers\\
                      \{2\} & \pm\{0, 2\} & 2um\integers & m\hspace{-2.5pt}+\hspace{-2.5pt}2m\integers & 2m\integers & \integers \\
                        \{2\} & \pm \{0, 2\} & 2um\integers & m\hspace{-2.5pt}+\hspace{-2.5pt}2m\integers & 2m\integers & \integers\backslash m\integers \\
                      \{2\} & \{0\} & \{0\} & \emptyset & m\integers & \integers\\
                      \{2\} & \pm\{0, 2\} & \{0\} & \{0\} & \{0\} &  \integers\backslash N_0 \\
                      \{2\} & \{0\} & \{0\} & \emptyset & \{0\} & \integers\backslash N_0 \\
                      \{2\} & \{0\} & \{0\} & \emptyset & \{0\} &\integers\backslash N_0' \\
                         \{1,2\}&\pm\{0, 1, 2\} & um\integers & m\integers & m\integers & \integers\backslash D_m\\
                                                 \{1,2\}&\pm\{0, 1, 2\} & 2um\integers & m\hspace{-2.5pt}+\hspace{-2.5pt}2m\integers & 2m\integers & \integers\backslash D_m\\
                         \{1,2\}&\pm \{0, 1\} & um\integers & \emptyset & m\integers & \integers\backslash D_m \\
                        \{1,2\}&\pm\{0, 1, 2\} & \{0\} & \{0\} & \{0\} & \integers\backslash E_0 \\
                      \{1,2\}&\pm\{0,  1\} & \{0\} & \emptyset & \{0\} & \integers\backslash E_0\\
                        \pint & \integers & um\integers & m\integers & m\integers & \integers\backslash D_m\\
                        \pint & \integers & \{0\} & \{0\} & \{0\} & \integers\backslash E_0\\
                    \end{matrix}
                    \end{align*}
                  \end{definition}
With $\mathsf Q$ and $\mc R$ defined, so has been the family $(\mc R_Q)_{Q\in \mathsf Q}$. The characterizing conditions $Z(\{\,\cdot\,\})\leq Q$ of the sets $\mc R_Q$, $Q\in \mathsf Q$, will be successively explained in Section~\ref{section:proof-main-theorem} in the process of proving their invariance under the category operations.

\section{\texorpdfstring{Invariance of $\mc R_Q$ under the Category Operations}{Invariance of R(Q) under Category Operations}}
\label{section:proof-main-theorem}
The strategy for proving that the sets defined in the preceding Section~\ref{section:definition-of-the-categories} are actually categories of partitions is the following: We choose the most convenient elements $Q$ of $\mathsf Q$, the ones for which it is easiest to prove that $\mc R_Q$ is a category. Once we have verified that these sets $\mc R_Q$ are categories, we show that every other $Q\in \mathsf{Q}$ can be written as a meet of a suitable family $\mathsf Q'\subseteq \mathsf Q$ of convenient ones  in the complete lattice $\mathsf L$. Then, the following Lemma~\ref{lemma:R-monotonic-meet-preserving} will allow us to conclude that $\mc R_Q$ is a category for every $Q\in \mathsf Q$.
\begin{notation}
 Given a family $(S_i)_{i\in I}$ of sets, we use the symbols $\bigcap_\times$ for the meet and $\bigcup_\times$ for the join operator of the product order $\leq$ with respect to $\subseteq$ on  $\bigtimes_{i\in I}\mathfrak{P}(S_i)$.
\end{notation}
      \begin{lemma}
        \label{lemma:R-monotonic-meet-preserving}
        The mapping $\mc R:\, \mathsf L\to \mathfrak{P}(\Cp)$ is monotonic and preserves meets.
      \end{lemma}

      \begin{proof}
 For $L,L'\in \mathsf L$ with $L\leq L'$ and $p\in \Cp$ the condition $p\in \mc R_L$, i.e.\ $Z(\{p\})\leq L$, implies $Z(\{p\})\leq L'$ and thus $p\in \mc R_{L'}$. Hence, $\mc R$ is monotonic. And for all subsets $\mathsf L'\subseteq \mathsf L$ holds
        \begin{align*}
          \bigcap \{\mc R_L\mid L\in \mathsf L'\}&=\bigcap \{\{p\in \Cp \mid Z(\{p\})\leq L\}\mid L\in \mathsf L'\}\\
                                                 &= \{p\in \Cp \mid \forall L\in \mathsf L': Z(\{p\})\leq L\}\\
          &= \{p\in \Cp \mid  Z(\{p\})\leq  {\bigcap}_\times \mathsf L'\}\\          
                                                   &=\mc R_{{\bigcap}_\times \mathsf L'},
        \end{align*}
        where we have only used the definition of $\mc R$. \qedhere
      \end{proof}

      \subsection{Behavior of $Z$ under Category Operations}
      In order to show that $\mc R_Q$ is a category for the convenient values $Q \in \mathsf Q$, we will use the Alternative Characterization (Lemma~\ref{lemma:characterization-categories}). It  pays to consider abstractly how $Z$ behaves under the (alternative) category operations beforehand.
      \subsubsection{Behavior of $Z$ under Rotation}
\begin{lemma}
      \label{lemma:Z-rotation}
  For all $r\in \{\lcurvearrowup,\lcurvearrowdown,\rcurvearrowdown,\rcurvearrowup\}$ and $p\in \mc C$ holds
  \begin{align*}
    Z(\{p^r\})=Z(\{p\}).
  \end{align*}
\end{lemma}
\begin{proof}
  Let $\rho$ be the map that rotates the points of $p$ to those of $p^r$.
  \par  
  \textbf{Step~1:}~\emph{Relating $\sigma_{p^r}$ to $\sigma_p$}.  Because the normalized color is defined specially to compensate for the color change involved in rotations,
  \begin{align}
    \label{eq:Z-rotation-1}
    \sigma_{p^r}(S)=\sigma_p(\rho^{-1}(S))
  \end{align}
  for all sets $S$ of points of $p^r$.
  \par
  \textbf{Step~2:}~\emph{Relating $\delta_{p^r}$ to $\delta_p$}. The cyclic order is also defined exactly to be respected by rotations:
  For all points $\alpha$ and $\beta$ of $p^r$, the identity  $]\alpha,\beta]_{p^r}=\rho(]\rho^{-1}(\alpha),\rho^{-1}(\beta)]_p)$, Equation~\eqref{eq:Z-rotation-1} and the definition of $\delta_p$ imply
  \begin{align}
    \label{eq:Z-rotation-2}    \delta_{p^r}(\alpha,\beta)&=\sigma_{p^r}(]\alpha,\beta]_{p^r})+{\textstyle\frac{1}{2}}(\sigma_{p^r}(\{\alpha\})-\sigma_{p^r}(\{\beta\}))\notag\\
&=\sigma_{p}(]\rho^{-1}(\alpha),\rho^{-1}(\beta)]_{p})+{\textstyle\frac{1}{2}}(\sigma_{p}(\rho^{-1}(\{\alpha\}))-\sigma_{p}(\rho^{-1}\{\beta\}))\notag\\
    &=\delta_p(\rho^{-1}(\alpha),\rho^{-1}(\beta)).
  \end{align}
  \par
  \textbf{Step~3:}~\emph{Implications for $Z$.} By definition, the blocks of $p^r$ are precisely the sets $\rho(B)$ for blocks $B$ of $p$. Hence, $F(\{p^r\})=F(\{p\})$ as $\rho$ is bijective. But also therefore, $V(\{p^r\})=V(\{p\})$ by Equation~\eqref{eq:Z-rotation-1}. The same equation also shows $\toco(p^r)=\toco(p)$ by choosing for $S$ the total set of points.
  \par
  Not only does $\rho$ map blocks to blocks but also subsequent legs to subsequent legs because rotation respects the cyclic order. And this is how all pairs of subsequent legs in $p^r$ arise. And Equation~\eqref{eq:Z-rotation-1} shows that two subsequent legs $\alpha$ and $\beta$ in $p^r$ satisfy $\sigma_{p^r}(\{\alpha,\beta\})=0$ if and only if their preimages, subsequent legs in $p$, satisfy $\sigma_p(\{\rho^{-1}(\alpha),\rho^{-1}(\beta)\})=0$. Hence, Equation~\eqref{eq:Z-rotation-2} proves  $(L,K)(\{p^r\})=(L,K)(\{p\})$.
  \par
  Since $\rho$ maps crossing blocks to crossing blocks and since all crossings of $p^r$ come from crossings in $p$, Equation~\eqref{eq:Z-rotation-2} also verifies $X(\{p^r\})=X(\{p\})$. Thus, $Z(\{p^r\})=Z(\{p\})$ as asserted.
\end{proof}

      \subsubsection{Behavior of $Z$ under Verticolor Reflection}
\begin{lemma}
  \label{lemma:Z-verticolor-reflection}
  For all $p\in \mc C$ holds
  \begin{align*}
    Z(\{\tilde p\})=(F,-V,-\toco,-L,-K,-X)(\{p\}).
  \end{align*}
\end{lemma}
\begin{proof}
  Same procedure as in Lemma~\ref{lemma:Z-rotation}. Let $\kappa$ be the map that reflects the points of $p$.
  \par  
  \textbf{Step~1:}~\emph{Relating $\sigma_{\tilde p}$ to $\sigma_p$}.  Reflection only moves points and does not change their native colors. As it sends upper points to upper points and lower points to lower points, the normalized colors are unaffected too. Verticolor reflection however involves not only reflection but also an inversion of (native and thus also normalized) color. We hence find for all sets of $S$ of points of $\tilde p$:
  \begin{align}
    \label{eq:Z-verticolor-reflection-1}
    \sigma_{p^r}(S)=-\sigma_{\hat p}(S)=-\sigma_p(\kappa^{-1}(S)).
  \end{align}
  \par
  \textbf{Step~2:}~\emph{Relating $\delta_{\tilde p}$ to $\delta_p$}.
The reflection operation reverses the cyclic order. Let $\alpha$ and $\beta$ be arbitrary points of $\tilde p$. Then, $]\alpha,\beta]_{\tilde p}=]\alpha,\beta]_{\hat p}=\kappa([\kappa^{-1}(\beta),\kappa^{-1}(\alpha)[_p)$. Note that the left-open interval has become a right-open one. Hence, by definition of $\delta_p$ and by Equation~\eqref{eq:Z-verticolor-reflection-1},
\begin{align}
\label{eq:Z-verticolor-reflection-2}
  \delta_{\tilde p}(\alpha,\beta)&=\sigma_{\tilde p}(]\alpha,\beta]_{\tilde p})+{\textstyle\frac{1}{2}}(\sigma_{\tilde p}(\{\alpha\})-\sigma_{\tilde p}(\{\beta\}))\notag\\
                                                &=-\sigma_{p}([\kappa^{-1}(\beta),\kappa^{-1}(\alpha)[_{p})-{\textstyle\frac{1}{2}}(\sigma_{p}(\{\kappa^{-1}(\alpha)\})-\sigma_{p}(\{\kappa^{-1}(\beta)\}))\notag\\
    &=-\sigma_p(]\kappa^{-1}(\beta),\kappa^{-1}(\alpha)]_p)-{\textstyle\frac{1}{2}}(\sigma_p(\{\kappa^{-1}(\beta)\})-\sigma_p(\{\kappa^{-1}(\alpha)\}))\notag\\
    &=-\delta_p(\kappa^{-1}(\beta),\kappa^{-1}(\alpha)).
  \end{align}
  In short, $\delta_{\tilde p}$ is is given by flipping the arguments in $\delta_p$ and inverting the sign.
  \par
  \textbf{Step~3:}~\emph{Implications for $Z$.} Still, the blocks of $\tilde p$ are exactly the sets $\kappa(B)$ for blocks of $B$ of $p$. Thus, $F(\{\tilde p\})=F(\{p\})$ is true. Equation~\eqref{eq:Z-verticolor-reflection-1} hence shows  $V(\{\tilde p\})=-V(\{p\})$. It also proves $\toco(\tilde p)=-\toco(p)$, again by choosing for $S$ the total set of points.
  \par
  Given points $\alpha$ and $\beta$ in $p'$, Equation~\eqref{eq:Z-verticolor-reflection-1} shows $\sigma_{\tilde p}(\{\alpha,\beta\})=0$ if and only if $\sigma_p(\{\kappa^{-1}(\alpha),\kappa^{-1}(\beta)\})=0$. Moreover, since $\kappa$ inverts the cyclic order, if $\alpha$ and $\beta$ belong to the same block $B$ of $\tilde p$, then  $]\alpha,\beta[_{\tilde p}\cap B=\kappa(]\kappa^{-1}(\beta),\kappa^{-1}(\alpha)[_p\cap \kappa^{-1}(B))$, where $\kappa^{-1}(\alpha)$ and $\kappa^{-1}(\beta)$ are now legs of the block $\kappa^{-1}(B)$ of $p$. That means, $\kappa$ maps subsequent legs to subsequent legs, it however reverses the order in which they are subsequent. And all pairs of subsequent legs arise in this way. Since the exchange of the order is compensated for by the flip of the arguments in Equation~\eqref{eq:Z-verticolor-reflection-2}, it follows $(L,K)(\{\tilde p\})=(-L,-K)(\{p\})$.
  \par
  And Equation~\eqref{eq:Z-rotation-2} also immediately proves $X(\{\tilde p\})=- X(\{p\})$. Here the flip is irrelevant since being crossing is a symmetric relation on the set of blocks. That is all we needed to see.
\end{proof}

      \subsubsection{Behavior of $Z$ under Tensor Products}
\begin{lemma}
  \label{lemma:Z-tensor-product}
  Let $p,p'\in \mc C$ be arbitrary.
  \begin{enumerate}[label=(\alph*)]
  \item  \label{lemma:Z-tensor-product-1} $F(\{p\otimes p'\})= F(\{p,p'\})$.
  \item  \label{lemma:Z-tensor-product-2} $V(\{p\otimes p'\})= V(\{p,p'\})$.
  \item  \label{lemma:Z-tensor-product-3}   $\toco(p\otimes p')=\toco(p)+\toco(p')$.
  \item  \label{lemma:Z-tensor-product-4}
    $Y(\{p\otimes p'\})\subseteq Y(\{p,p'\})+\gcd(\toco(p),\toco(p'))\integers$ for all $Y\in \{L,K,X\}$.
  \end{enumerate}
\end{lemma}
\begin{proof}
  One last time we proceed as in Lemmata~\ref{lemma:Z-rotation} and~\ref{lemma:Z-verticolor-reflection}.
  Let $S_p$ and $S_p'$ be the sets of points in $p\otimes p'$ stemming from $p$ and $p'$ respectively, and let $\tau$ denote the map translating the points of $p'$ to their locations in $S_{p'}$.
  \par
  \textbf{Step~1:}~\emph{Relating $\sigma_{p\otimes p'}$ to $\sigma_p$ and $\sigma_{p'}$.} The tensor product leaves $S_p$ untouched, it only involves a shift of $\tau^{-1}(S_{p'})$ to $S_{p'}$. The native colors are not affected. Since $\tau$ maps upper points to upper points and lower points to lower points, the normalized colors do not  change either. We will be able to confine our considerations to sets $S$ of points with $S\subseteq S_p$ or $S\subseteq S_{p'}$, for which then holds in conclusion:
  \begin{align}
    \label{eq:Z-tensor-product-1}
    \sigma_{p\otimes p'}(S)&=
                             \begin{cases}
                               \sigma_{p}(S)&\text{if } S\subseteq S_p,\\
\sigma_{p'}(\tau^{-1}(S))&\text{if }S\subseteq S_{p'}.
                             \end{cases}
  \end{align}  
  \par
  \textbf{Step~2:}~\emph{Relating $\delta_{p\otimes p'}$ to $\delta_p$ and $\delta_{p'}$.} 
  Again, we are only interested in the color distances of points $\alpha$ and $\beta$ which both belong to $S_p$ or both to $S_{p'}$. Then,
  \begin{align*}
    ]\alpha,\beta]_{p\otimes p'}&=
                                  \begin{cases}
                                    ]\alpha,\beta]_{p}\cup S_{p'}&\text{if }\alpha,\beta\in S_p\text{ and }]\alpha,\beta[_{p\otimes p'}\cap S_{p'}\neq\emptyset,\\
                                    ]\alpha,\beta]_{p}&\text{if }\alpha,\beta\in S_p\text{ and }]\alpha,\beta[_{p\otimes p'}\cap S_{p'}=\emptyset,\\
                                    \tau(]\tau^{-1}(\alpha),\tau^{-1}(\beta)]_{p'})\cup S_p&\text{if }\alpha,\beta\in S_{p'}\text{ and }]\alpha,\beta[_{p\otimes p'}\cap S_{p}\neq \emptyset,\\
                                    \tau(]\tau^{-1}(\alpha),\tau^{-1}(\beta)]_{p'})&\text{if }\alpha,\beta\in S_{p'}\text{ and }]\alpha,\beta[_{p\otimes p'}\cap S_{p}= \emptyset.
                                  \end{cases}
  \end{align*}
  Hence, $\toco(p)=\sigma_p(S_p)$ and $\toco(p')=\sigma_{p'}(\tau^{-1}(S_{p'}))$, Equation~\eqref{eq:Z-tensor-product-1} and the definition of $\delta_{p\otimes p'}$ yield in these cases
  \begin{align*}
    \delta_{p\otimes p'}(\alpha,\beta)&=\sigma_{p\otimes p'}(]\alpha,\beta]_{p\otimes p'})+{\textstyle\frac{1}{2}}(\sigma_{p\otimes p'}(\{\alpha\})-\sigma_{p\otimes p'}(\{\beta\}))\\
    &=                                                              \begin{cases}
                                    \delta_p(\alpha,\beta)+\toco(p')&\text{if }\alpha,\beta\in S_p\text{ and }]\alpha,\beta[_{p\otimes p'}\cap S_{p'}\neq\emptyset,\\
                                    \delta_p(\alpha,\beta)&\text{if }\alpha,\beta\in S_p\text{ and }]\alpha,\beta[_{p\otimes p'}\cap S_{p'}=\emptyset,\\
                                    \delta_{p'}(\tau^{-1}(\alpha),\tau^{-1}(\beta))+\toco(p)&\text{if }\alpha,\beta\in S_{p'}\text{ and }]\alpha,\beta[_{p\otimes p'}\cap S_{p}\neq \emptyset,\\
                                    \delta_{p'}(\tau^{-1}(\alpha),\tau^{-1}(\beta))&\text{if }\alpha,\beta\in S_{p'}\text{ and }]\alpha,\beta[_{p\otimes p'}\cap S_{p}= \emptyset.
                                  \end{cases}
  \end{align*}
  In particular, it follows
  \begin{align}
    \label{eq:Z-tensor-product-2}
    \delta_{p\otimes p'}(\alpha,\beta)\equiv
    \begin{dcases}\begin{rcases}
      \delta_p(\alpha,\beta)& \text{if } \alpha,\beta\in S_p\\
      \delta_{p'}(\tau^{-1}(\alpha),\tau^{-1}(\beta)) &\text{if } \alpha,\beta\in S_{p'}
    \end{rcases}\end{dcases}
\mod \gcd(\toco(p),\toco(p')).
  \end{align}
  \par
  \textbf{Step~3:}~\emph{Implications for $Z$.} The blocks of $p\otimes p'$ are by definition precisely the blocks of $p$ plus the sets $\tau(B')$ for blocks $B'$ of $p'$. This proves Part~\ref{lemma:Z-tensor-product-1}. Especially, $B\subseteq S_p$ or $B\subseteq S_{p'}$ for all blocks $B$ of $p\otimes p'$. Hence, Equation~\eqref{eq:Z-tensor-product-1} shows Part~\ref{lemma:Z-tensor-product-2}. The same equation shows $\toco(p\otimes p')=\sigma_{p\otimes p'}(S_p\cup S_{p'})=\sigma_p(S_p)+\sigma_{p'}(\tau^{-1}(S_{p'}))=\toco(p)+\toco(p')$ and thus Claim~\ref{lemma:Z-tensor-product-3}.
  \par
  Since every block of $B$ of $p'$ satisfies $B\subseteq S_p$ or $B\subseteq S_{p'}$ any two subsequent legs of a block $B$ of $p'$ are also subsequent legs in the block of $p$ or $p'$ where they originate. Moreover, Equation~\eqref{eq:Z-tensor-product-1} shows that their combined color sum in $p\otimes p'$ is identical to the one in $p$ respectively $p'$. Thus, Equation~\eqref{eq:Z-tensor-product-2} proves Part~\ref{lemma:Z-tensor-product-4} for $Y\in \{L,K\}$.
  \par
Lastly, if  $B_1$ and $B_2$ are crossing blocks of $p\otimes p'$, then necessarily $B_1,B_2\subseteq S_p$ or $B_1,B_2\subseteq S_{p'}$ by definition of the tensor product: Indeed, if we suppose, e.g., $B_1\subseteq S_p$ and $B_2\subseteq S_{p'}$, then a crossing of $B_1$ and $B_2$ would mean a crossing between $S_p$ and $S_{p'}$, which is impossible since $S_p$ and $S_{p'}$ are disjoint consecutive sets.  Hence, Equation~\eqref{eq:Z-tensor-product-2} also proves Part~\ref{lemma:Z-tensor-product-4} for $Y=X$, which completes the proof.
\end{proof}

\subsubsection{Behavior of $Z$ under Erasing Turns}
We have not understood yet how $Z$ behaves under erasing turns. However, the effect of the erasing operation on $Z$ is too complicated to treat well abstractly beyond the following simple (but important) observations. 
      \begin{lemma}
        \label{lemma:erasing-tricks}
        Let $p\in \Cp$ be a partition with no upper points and let $T$ be a turn comprising exactly the two rightmost lower points of $p$.
        \begin{enumerate}[label=(\alph*)]
        \item\label{lemma:erasing-tricks-1} Every point of  $E(p,T)$ is also a point of $p$.
        \item\label{lemma:erasing-tricks-2} For every set $S$ of points in $E(p,T)$ holds $\sigma_{E(p,T)}(S)=\sigma_p(S)$.
          \item\label{lemma:erasing-tricks-3} For all points $\gamma_1$ and $\gamma_2$ of $E(p,T)$,
        \begin{align*}
          \delta_{E(p,T)}(\gamma_1,\gamma_2)=\delta_p(\gamma_1,\gamma_2).
        \end{align*}
        \end{enumerate}
      \end{lemma}
      \begin{proof}
        Since the erasing procedure only affects the points in $T$, Parts~\ref{lemma:erasing-tricks-1} and~\ref{lemma:erasing-tricks-2} are clear. We prove Part~\ref{lemma:erasing-tricks-3}.
        If $\gamma_1$ lies to the left of $\gamma_2$, then $]\gamma_1,\gamma_2]_{E(p,T)}=]\gamma_1,\gamma_2]_p$ and thus $\delta_{E(p,T)}(\gamma_1,\gamma_2)=\delta_{p}(\gamma_1,\gamma_2)$. Otherwise, $]\gamma_1,\gamma_2]_{E(p,T)}=]\gamma_1,\gamma_2]_p\cup T$, which is a disjoint union, implying \(\delta_{p}(\gamma_1,\gamma_2)=\delta_{E(p,T)}(\gamma_1,\gamma_2)+\sigma_p(T)=\delta_{E(p,T)}(\gamma_1,\gamma_2)\) since $\sigma_p(T)=0$. Thus, in any case $\delta_{E(p,T)}(\gamma_1,\gamma_2)=\delta_{p}(\gamma_1,\gamma_2)$.
      \end{proof}
      
      \subsection{Total Color Sum} The first family of elements $Q\in \mathsf Q$ is chosen such that all components except the one for $\toco$ are trivial. That means we restrict the allowed total color sums of partitions.
      \begin{remark}
        \label{remark:categories-1}
        For all $m\in \{0\}\cup \pint$, the set $\mc R_Q$ with $Q=(\pint,\integers,m\integers,\integers,\integers,\integers)$ is exactly \emph{the set of all partitions whose total color sum is a multiple of $m$.}
        \par
        No further constraints on $p\in \Cp$ are imposed by the other components of $Q$ via $Z(\{p\})\leq Q$ since these conditions are satisfied by \emph{any} partition.
      \end{remark}
\begin{lemma}{{\normalfont\cite[Lemma~2.6]{TaWe15a}}}
  \label{lemma:total-color-sum-operations}
    \label{lemma:categories-1}  
    For every $m\in \{0\}\cup\pint$ the set $\mc R_Q$, where \[Q=(\pint,\integers,m\integers,\integers,\integers,\integers),\] is a category of two-colored partitions. 
\end{lemma}
\begin{proof}
  We use the Alternative Characterization (Lemma~\ref{lemma:characterization-categories}) to show that $\mc R_Q$ is a category. By Remark~\ref{remark:categories-1} it suffices to show that $\toco(\{\,\cdot\,\})\subseteq m\integers$ is stable under the alternative category operations. 
  \par
  \textbf{Rotation:} According to Lemma~\ref{lemma:Z-rotation} for any $r\in \{\rcurvearrowdown,\rcurvearrowup,\lcurvearrowup,\lcurvearrowdown\}$ and $p\in \mc R_Q$ holds $\toco(p^r)=\toco(p)\subseteq m\integers$. Thus,  $\toco(\{\,\cdot\,\})\subseteq m\integers$ is stable under rotations.
  \par
  \textbf{Verticolor Reflection:} Lemma~\ref{lemma:Z-verticolor-reflection} and $-m\integers=m\integers$ show $\toco(\tilde p)=-\toco(\tilde p)\subseteq m\integers$ for all $p\in \mc R_Q$, proving that  $\toco(\{\,\cdot\,\})\subseteq m\integers$ is preserved by verticolor reflection
  \par
  \textbf{Tensor Product:} Given $p, p'\in \mc R_Q$ we can employ Lemma~\hyperref[lemma:Z-tensor-product-3]{\ref*{lemma:Z-tensor-product}~\ref*{lemma:Z-tensor-product-3}} to infer $\toco(p\otimes p')=\toco(p)+\toco(p')$ and thus $\toco(p\otimes p')\in m\integers$ since $\toco(p),\toco(p')\in m\integers$. In conclusion,  $\toco(\{\,\cdot\,\})\subseteq m\integers$ is invariant under tensor products.
  \par
  \textbf{Erasing:} Lastly, it remains to consider $p\in \mc R_Q$ and a turn $T$ in $p$ and to show $\toco(E(p,T))\in m\integers$. Because we already know $\toco(\{\,\cdot\,\})\subseteq m\integers$ is respected by rotations, we can assume that $p$ has no upper points and that $T$ comprises exactly the two rightmost lower points. Let $P_p$ denote the set of all points of $p$. Since $T$ is a turn, $\sigma_p(T)=0$. Hence, by Lemma~\hyperref[lemma:erasing-tricks-2]{\ref*{lemma:erasing-tricks}~\ref*{lemma:erasing-tricks-2}},
     \begin{align*}
       \toco(E(p,T))=\sigma_{E(p,T)}(P_p\backslash T)=\sigma_p(P_p\backslash T)=\sigma_p(P_p)-\sigma_p(T)=\sigma_p(P_p)=\toco(p).
     \end{align*}
  That concludes the proof.
  \end{proof}

  \subsection{Block Size and Color Sum} A second subset of $\mathsf Q$ gathers elements where only the components for $F$ and $V$ are non-trivial. So, we only impose (interdependent) conditions on the sizes and color sums of blocks.

  \begin{remark}
    \label{remark:categories-2}
    \begin{enumerate}[label=(\alph*)]
    \item   For $Q=(\{1,2\},\pm \{0,1,2\},\integers,\integers,\integers,\integers)$ (Part~\ref{lemma:categories-2-1} of the ensuing Lemma~\ref{lemma:categories-2}) the set $\mc R_Q$ is simply $\Cp_{\leq 2}$, \emph{the set of all partitions all of whose blocks have at most two legs}.
      \par
      Only the $F$-condition in $Z(\{\,\cdot\,\})\leq Q$ is important; The apparent $V$-constraint is merely a reflection of the one for $F$: Any singleton block must have color sum  $1$ ($\circ$) or $-1$  ($\bullet$), and a pair block can only ever be neutral ($\circ\bullet$ or $\bullet\circ$) or have color sums $2$ ($\circ\circ$) or  $-2$ ($\bullet\bullet$). And all the remaining conditions imposed by $Z(\{\,\cdot\,\})\leq Q$ are trivially satisfied by \emph{any} partition.
    \item
      If $Q=(\{2\},\pm \{0,2\},\integers,\integers,\integers,\integers)$ (Part~\ref{lemma:categories-2-2}), then $\mc R_Q=\Cp_{2}$ is simply \emph{the set of all pair partitions}. Again, the $V$-condition merely reflects the $F$-constraint.
      \par
           Likewise, the non-trivial value $2\integers$ of the $\toco$-component of $Q$ does \emph{not} impose any restriction on $\mc R_Q$ not already placed by the $F$-constraints: For any partition $p\in \Cp$,
           \begin{align}
             \label{eq:total-color-sum-block-color-sums}
     \toco(p)=\sigma_p\left(\bigcup_{B\text{ block of }p}B\right)=\sum_{B\text{ block of }p}\sigma_p(B).
   \end{align}
   Hence, $\toco(p)=2\integers$ is a necessary consequence of $V(\{p\})\subseteq \pm \{0,2\}$.
   \par
    The conditions induced by the $L$-, $K$- and $X$-components of $Q$ remain redundant.
   \item
     The set $\mc R_Q$ with $Q=(\{2\},\{0\},\{0\},\emptyset,\integers,\integers)$ (Part~\ref{lemma:categories-2-3}) is identical to $\Cppnb$, \emph{the set of all pair partitions with neutral blocks}. (See \cite{MaWe18a} and \cite{MaWe18b} for a classification of all its subcategories.)
     \par
     In contrast to the two previous examples, the constraint on block color sums is \emph{not} merely a consequence of the range of allowed block sizes. The $F$-constraint restricts $\mc R_Q$ to a subset of $\Cpp$ and the $V$-constraint prohibits non-neutral blocks.
     \par
     However, the non-trivial value $\{0\}$ of the $\toco$-component of $Q$ imposes \emph{ no} additional restrictions: $\toco(p)=0$ is, by Equation~\eqref{eq:total-color-sum-block-color-sums}, a necessary consequence of $V(\{p\})=\{0\}$  for any $p\in \Cp$.
   \par
   And also the non-trivial $L$-component  $\emptyset$ of $Q$ provides \emph{no} additional constraints beyond the ones induced by the $F$- and $V$-components: A pair block which is neutral, i.e.\ has normalized coloring $\circ\bullet$ or $\bullet\circ$, cannot have subsequent legs of the same normalized color. Thus, for every $p\in \Cpp$ with $V(\{p\})=\{0\}$ automatically holds $L(\{p\})=\emptyset$.
 \item Lastly, the set $\mc R_Q$ where $Q=(\{1,2\},\pm \{0,1\},\integers,\emptyset,\integers,\integers)$ (Part~\ref{lemma:categories-2-4}) is given exactly by \emph{the set of all partitions all of whose blocks have at most two legs and all of whose pair blocks are neutral}.
   \par
   As said before, restricting $F$ to $\{1,2\}$ reduces the allowed set of block color sums to $\pm \{0,1,2\}$ and the values $-2$ and $2$ can then only stem from pair blocks. Excluding these two values then forces all pair blocks to be neutral. So, as in Part~\ref{lemma:categories-2-3}, the $F$- and $V$-components induce  constraints in their own right.
   \par
   However, the non-trivial value  $\emptyset$ of the $L$-component of $Q$ is merely a reflection of the $F$- and $V$-constraints because neutral pair blocks cannot have subsequent legs of the same normalized color.
    \end{enumerate}
  \end{remark}
  
  \begin{lemma}
     \label{lemma:categories-2}
     The set $\mc R_Q$ is a category of partitions
   \begin{enumerate}[label=(\alph*)]
   \item\label{lemma:categories-2-1} \ldots if $Q=(\{1,2\},\pm \{0,1,2\},\integers,\integers,\integers,\integers)$,
   \item\label{lemma:categories-2-2} \ldots if $Q=(\{2\},\pm \{0,2\},2\integers,\integers,\integers,\integers)$,
   \item\label{lemma:categories-2-3} \ldots if $Q=(\{2\},\{0\},\{0\},\emptyset,\integers,\integers)$, or               
   \item\label{lemma:categories-2-4} \ldots if $Q=(\{1,2\},\pm \{0,1\},\integers,\emptyset,\integers,\integers)$.
   \end{enumerate}     
 \end{lemma}

 \begin{proof}
   We treat the four claims largely simultaneously.
   Let $Q$ be one of the four elements named in Claims~\ref{lemma:categories-2-1}--\ref{lemma:categories-2-4} and let $f_Q$ be the first and $v_Q$ the second component of $Q$. Remark~\ref{remark:categories-2} showed that, if $(f_Q,v_Q)=(\{1,2\},\pm \{0,1,2\})$ (Case~\ref{lemma:categories-2-1}) or $(f_Q,v_Q)=(\{2\},\pm \{0,2\})$  (Case~\ref{lemma:categories-2-2}), then
   \begin{align*}
     \mc R_Q=\{p\in \Cp\mid F(\{p\})\subseteq f_Q\},
   \end{align*}
   and that, if $(f_Q,v_Q)=(\{1,2\},\pm \{0,1\})$ (Case~\ref{lemma:categories-2-3}) or $(f_Q,v_Q)=(\{2\},\{0\})$ (Case~\ref{lemma:categories-2-4}), then
   \begin{align*}
      \mc R_Q=\{p\in \Cp\mid (F,V)(\{p\})\leq (f_Q,v_Q)\}.
   \end{align*}
   We use the Alternative Characterization (Lemma~\ref{lemma:characterization-categories}) to show that $\mc R_Q$ is a category. As $\PartIdenW\in \mc R_Q$ is clearly true, that means it suffices to prove that in Cases~\ref{lemma:categories-2-1} and~\ref{lemma:categories-2-2} the condition $F(\{\,\cdot\,\})\subseteq f_Q$ and in Cases~\ref{lemma:categories-2-3} and~\ref{lemma:categories-2-4} the condition $(F,V)(\{\,\cdot\,\})\leq (f_Q,v_Q)$ are stable under rotations, tensor products, verticolor reflection and erasing turns.
   \par
   \textbf{Rotation:} For all $p\in \mc R_Q$ and $r\in \{\rcurvearrowdown,\rcurvearrowup,\lcurvearrowup,\lcurvearrowdown\}$ by Lemma~\ref{lemma:Z-rotation} holds $(F,V)(\{p^r\})=(F,V)(\{p\})\leq (f_Q,v_Q)$. Thus, $(F,V)(\{\,\cdot\,\})\leq (f_Q,v_Q)$ is stable under rotations.
   \par
   \textbf{Verticolor reflection:} Given arbitrary $p\in \mc R_Q$ we are guaranteed by Lemma~\ref{lemma:Z-verticolor-reflection} that $(F,V)(\{p\})=(F,-V)(\{p\})\leq (f_Q,-v_Q)$. Since $-v_Q=v_Q$ this proves  that the constraint $(F,V)(\{\,\cdot\,\})\leq (f_Q,v_Q)$ is preserved by verticolor reflection.
   \par
   \textbf{Tensor product:} According to Lemma~\hyperref[lemma:Z-tensor-product-1]{\ref*{lemma:Z-tensor-product}~\ref*{lemma:Z-tensor-product-1}} and~\ref{lemma:Z-tensor-product-2}, for all partitions $p,p'\in\mc R_Q$ holds $(F,V)(\{p\otimes p'\})=(F,V)(\{p,p'\})\leq (f_Q,v_Q)$. Hence, tensor products respect $(F,V)(\{\,\cdot\,\})\leq (f_Q,v_Q)$. 
    \par
    \textbf{Erasing:}
    Lastly, let $p\in \mc R_Q$ be arbitrary and let $T$ be a turn in $p$. We have to show $F(\{E(p,T)\})\subseteq f_Q$  and, but just in Cases~\ref{lemma:categories-2-3} and~\ref{lemma:categories-2-4}, also $V(\{E(p,T)\})\subseteq v_Q$. Since we already know that $(F,V)(\{\,\cdot\,\})\leq (f_Q,v_Q)$ is invariant under rotations, we can assume that $p$ has no upper points and that $T$ comprises exactly the two rightmost lower points of $p$. 
    \par
    Let $B$ be an arbitrary block of $E(p,T)$. What we need to prove is $|B|\in f_Q$ as well as $\sigma_{E(p,T)}(B)\in v_Q$, the latter however just in Cases~\ref{lemma:categories-2-3} and~\ref{lemma:categories-2-4}. If $B$ is in fact a block of $p$, then $|B|\in f_Q$ and, by Lemma~\hyperref[lemma:erasing-tricks-2]{\ref*{lemma:erasing-tricks}~\ref*{lemma:erasing-tricks-2}}, also $\sigma_{E(p,T)}(B)=\sigma_p(B)\in v_Q$ since we have assumed $p\in \mc R_Q$. Thus, suppose that $B$ is not a block of $p$. Then, there are (not necessarily distinct) blocks $B_1$ and $B_2$ of $p$ with $B_1\cap T,B_2\cap T\neq \emptyset$, with $T\subseteq B_1\cup B_2$ and with $ B=(B_1\cup B_2)\backslash T\neq \emptyset$. 
    \par
    \textbf{Step 1:} \emph{Block size.} We prove $|B|\in f_Q$.  From $|T|=2$ and from $T\subseteq B_1\cup B_2$ follows
   \begin{align}
          \label{eq:categories-2-1}
    |B|=|B_1\cup B_2|- |T|=
     \begin{cases}
       |B_1|+|B_2|-2 &\text{if }B_1\neq B_2,\\
       |B_1|-2&\text{if }B_1=B_2.
     \end{cases}
   \end{align}
   By definition, $f_Q=\{1,2\}$ or $f_Q=\{2\}$. We treat the two cases individually.
   \par
   \textbf{Case~1.1:}~\emph{Parts~\ref{lemma:categories-2-1} and~\ref{lemma:categories-2-4}.} If $f_Q= \{1,2\}$, then in particular $|B_1|,|B_2|\leq 2$. Hence, no matter whether $B_1=B_2$ or $B_1\neq B_2$, Equation~\eqref{eq:categories-2-1} implies $|B|\leq 2$, meaning $|B|\in f_Q$, as was claimed.\par
   \textbf{Case~1.2:}~\emph{Parts~\ref{lemma:categories-2-2} and~\ref{lemma:categories-2-3}.} If we assume $f_Q= \{2\}$ instead, then $|B_1|=|B_2|=2$. Now, $B_1=B_2$ is impossible since otherwise Equation~\eqref{eq:categories-2-1} would imply $|B|=0$,  con\-tra\-dic\-ting $B\neq \emptyset$. Rather, $B_1\neq B_2$ must be true. Then, $|B|=2\in f_Q$ by Equation~\eqref{eq:categories-2-1}.
  \par
\textbf{Step~2:}~\emph{Block color sum.} We show $\sigma_p(B)\in v_Q$ in Cases~\ref{lemma:categories-2-3} and~\ref{lemma:categories-2-4}. Since $\sigma_p(T)=0$ and  $T\subseteq B_1\cup B_2$, we infer by Lemma~\hyperref[lemma:erasing-tricks-2]{\ref*{lemma:erasing-tricks}~\ref*{lemma:erasing-tricks-2}}:
   \begin{align}
          \label{eq:categories-2-2}     
     \sigma_{E(p,T)}(B)&=\sigma_p((B_1\cup B_2)\backslash T)\notag\\
     &=\sigma_p((B_1\cup B_2)\backslash T)+\sigma_p(T)\notag\\
     &=\sigma_p(B_1\cup B_2)\notag\\
     &=
       \begin{cases}
         \sigma_p(B_1)+\sigma_p(B_2) &\text{if }B_1\neq B_2,\\
         \sigma_p(B_1)&\text{if }B_1=B_2.
       \end{cases}
   \end{align}
   Cases~\ref{lemma:categories-2-3} and~\ref{lemma:categories-2-4} correspond to the possibilities $(f_Q,v_Q)=(\{2\},\{0\})$ and  $(f_Q,v_Q)=(\{1,2\},\pm \{0,1,2\})$. We treat them individually.
   \par
   \textbf{Case~2.1:}~\emph{Part~\ref{lemma:categories-2-3}.} If we assume $(f_Q,v_Q)=(\{2\},\{0\})$, i.e., in particular $\sigma_p(B_1)=\sigma_p(B_2)=0$, then, irrespective of whether $B_1=B_2$ or $B_1\neq B_2$, Equation~\eqref{eq:categories-2-2} proves $\sigma_{E(p,T)}(B)=0$.
   \par
      \textbf{Case~2.2:}~\emph{Part~\ref{lemma:categories-2-4}.}
      Now, let $(f_Q,v_Q)=(\{0,1,2\},\pm \{0,1\})$. Then, $B_1$ and $B_2$ are singletons or pair blocks and $\sigma_p(B_1),\sigma_p(B_2)\in \pm \{0,1\}$.  If both $B_1$ and $B_2$ were singletons, then $B_1\cap T,B_2\cap T\neq \emptyset$ and $T\subseteq B_1\cup B_2$ would imply $B_1\cup B_2=T$ and thus the contradiction $B=(B_1\cup B_2)\backslash T=\emptyset$. Hence, we know that at least one of the blocks $B_1$ and $B_2$ must be a pair. The assumptions $p\in \mc R_Q$ and $v_Q=\pm \{0,1\}$ force the pair blocks of $p$ to be neutral (since non-neutral pair blocks have color sums $2$ or $-2$). It follows $\sigma_p(B_1)=0$ or $\sigma_p(B_2)=0$. Now, no matter whether $B_1=B_2$ or $B_1\neq B_2$, Equation~\eqref{eq:categories-2-2} proves $\sigma_{E(p,T)}(B)\in \pm \{0,1\}=v_Q$, which concludes the proof. 
 \end{proof}

\subsection{Color Distances between Legs of the Same Block}
For our third family of elements of $\mathsf Q$ the $L$- and $K$-components are non-trivial, implying constraints on the color distances between subsequent legs of the same block. In part, this translates to a condition on the color distances between \emph{arbitrary} -- not just subsequent -- legs.

\begin{lemma}
  \label{lemma:simplification-of-R}
  Let $m\in \{0\}\cup \pint$ and \[Q=(\pint,\integers,m\integers,m\integers,m\integers,\integers)\] and let $p\in \Cp$ be arbitrary. Then, $Z(\{p\})\leq Q$ if and only if for all blocks $B$ in $p$ and any two legs $\alpha,\beta\in B$ holds $\delta_p(\alpha,\beta)\in m\integers$. 
\end{lemma}
\begin{proof}
  Suppose  $\delta_p(\alpha,\beta)\in m\integers$ for all blocks $B$ of $p$ and all $\alpha,\beta\in B$. Since we can especially choose $\alpha\neq \beta$ and $]\alpha,\beta[_p\cap B=\emptyset$, it follows $(L,K)(\{p\})\leq (m\integers,m\integers)$. On the other hand, picking $\alpha=\beta$ shows $\toco(p)=\delta_p(\alpha,\alpha)\in m\integers$, which lets us conclude $\toco(\{p\})\subseteq m\integers$. As $(F,V,X)(\{p\})\leq (\pint,\integers,\integers)$ is trivially true, that proves one implication.
  \par
  To show the converse, let $Z(\{p\})\leq Q$ and let $\alpha,\beta$ be legs of the same block $B$ of $p$. If $\alpha=\beta$, then $\delta_p(\alpha,\beta)=\toco(p)\in m\integers$ as $\toco(\{p\})\subseteq m\integers$. Thus, we can assume $\alpha\neq \beta$.  We can further suppose that $]\alpha,\beta[_p\cap B\neq \emptyset$ since otherwise $\delta_p(\alpha,\beta)\in (L\cup K)(\{p\})\subseteq m\integers$. Thus, let $]\alpha,\beta[_p\cap B\neq \emptyset$. Then, $p$ has at least three points. If so, then we find iteratively by moving from $\alpha$ in direction of $\beta$ in accordance with the orientation a number $n\in \pint$ of legs $\gamma_1,\ldots,\gamma_n$ of $B$ such that, writing $\gamma_0\eqpd \alpha$ and $ \gamma_{n+1}\eqpd \beta$, the points $\gamma_0,\ldots,\gamma_{n+1}$ are pairwise distinct, such that $(\gamma_0,\ldots,\gamma_{n+1})$ is ordered and such that $]\alpha,\beta[_p\cap B=\{\gamma_1,\ldots,\gamma_n\}$. Now, Lemma~\hyperref[lemma:color-distance-3]{\ref*{lemma:color-distance}~\ref*{lemma:color-distance-3}} implies $\delta_p(\alpha,\beta)\equiv\sum_{k=0}^n\delta_p(\gamma_k,\gamma_{k+1})\mod m$. Because for every $k\in \{0\}\cup\dwi{n}$  holds  $]\gamma_k,\gamma_{k+1}[\cap B=\emptyset$ and thus $\delta_p(\gamma_k,\gamma_{k+1})\in (L\cup K)(\{p\})\subseteq m\integers$,  that proves the claim.
\end{proof}
\begin{remark}
  \label{remark:categories-3}
  \begin{enumerate}[label=(\alph*)]
  \item For every $m\in \{0\}\cup \pint$, if $Q=(\pint,\integers,m\integers,m\integers,m\integers,\integers)$ (Part~\ref{lemma:categories-3-1} of Lemma~\ref{lemma:categories-3} below), then  $\mc R_Q$ is given by \emph{the set of all partitions such that the color distance between any two legs of the same block is a multiple of $m$}, as seen in Lemma~\ref{lemma:simplification-of-R}.
    Mostly, it is the $L$- and $K$-components of $Q$ inducing the characteristic constraints via $Z(\{\,\cdot\,\})\leq Q$. The $F$-, $V$- and $X$-conditions are redundant. However,  the $\toco$-condition is \emph{not}. And it is not implied by the $L$- and $K$-constraints, either.
  \item If $Q=(\{1,2\},\pm\{0,1,2\},2m\integers,m\!+\!2m\integers,2m\integers,\integers)$ (Part~\ref{lemma:categories-3-2}) for some $m\in \pint$, then the set $\mc R_Q$ is a subset of the set from Part~\ref{lemma:categories-3-1}: Still, a partition $p\in \Cp$ is  required to have color distances in $m\integers$ between any two legs of the same block if it is to be an element of $\mc R_Q$. However, now, additionally, all blocks of $p$ must have size one or two, the total color sum $\toco(p)$ must be a multiple of $2m$ (and not just $m$) and, most importantly, the color distances between subsequent legs of the same block must satisfy two different conditions, depending on their normalized colors. Since blocks can have size two at most, two legs of the same block are subsequent if and only if they are distinct.  Hence, $\mc R_Q$ is \emph{the set of all partitions such that}
    \begin{itemize}[label=$\filledsquare$]
    \item \emph{every block has at most two legs,}
    \item \emph{the total color sum is an \emph{even} multiple of $m$,}
    \item \emph{the color distance between any two \emph{distinct} legs of the same block is}
      \begin{itemize}[label=--]
      \item \emph{an \emph{odd} multiple of $m$ if they have \emph{identical} normalized colors and}
              \item \emph{an \emph{even} multiple of $m$ if they have \emph{different} normalized colors.}
      \end{itemize}
    \end{itemize}
    Also, note that, in contrast to Part~\ref{lemma:categories-3-1}, the parameter $m=0$ is \emph{not} allowed.
    \par
    For points $\alpha$ and $\beta$ in $p\in \Cp$, saying 
    \begin{align*}
      \delta_p(\alpha,\beta)\in \frac{\sigma_p(\{\alpha,\beta\})}{2}m+2m\integers.
    \end{align*}
    is a helpful technical way of expressing that the distance of the points is an odd multiple of $m$ if they have the same normalized color and an even multiple otherwise.
    \par
    The constraints induced via $Z(\{\,\cdot\})\leq Q$ by the $F$-, $V$- and $\toco$-components are non-trivial and \emph{not} implied by the $L$- and $K$-restrictions. 
    Of course, the $X$-condition is still redundant.
  \end{enumerate}
\end{remark}

\begin{lemma}
  \label{lemma:categories-3}
  For every $m\in \{0\}\cup \pint$ the set $\mc R_Q$ is a category of partitions
  \begin{enumerate}[label=(\alph*)]
  \item\label{lemma:categories-3-1} \ldots if $Q=(\pint,\integers,m\integers,m\integers,m\integers,\integers)$, or
  \item\label{lemma:categories-3-2} \ldots if $Q=(\{1,2\},\pm\{0,1,2\},2m\integers,m\!+\!2m\integers,2m\integers,\integers)$.
  \end{enumerate}
\end{lemma}

\begin{proof}
  The two claims can be verified largely simultaneously. Let $Q$ be one of the tuples given in Claims~\ref{lemma:categories-3-1} and~\ref{lemma:categories-3-2} and abbreviate $Q=(f_Q,\{0\}\cup(\pm f_Q),k_Q,l_Q,k_Q,\integers)$.
  We prove that $\mc R_Q$ is a category by means of the Alternative Characterization (Lemma~\ref{lemma:characterization-categories}). Clearly, $\PartIdenW\in \mc R_Q$. By Lemma~\ref{lemma:categories-1} the condition $\toco(\{\,\cdot\,\})\subseteq k_Q$ is stable under rotation, tensor products, verticolor reflection and erasing turns. The constraint $(F,V)(\{\,\cdot\,\})\leq (f_Q, \{0\}\cup(\pm f_Q))$ is trivially preserved in Case~\ref{lemma:categories-3-1}, and it is preserved in Case~\ref{lemma:categories-3-2}  as well according to Lemma~\hyperref[lemma:categories-2-1]{\ref*{lemma:categories-2}~\ref*{lemma:categories-2-1}}. Hence, it is sufficient to show that $(L,K)(\{\,\cdot\,\})\leq (l_Q,k_Q)$ is invariant under the alternative category operations.
  \par
  \textbf{Rotation:} By Lemma~\ref{lemma:Z-rotation} holds $(L,K)(\{p^r\})=(L,K)(\{p\})\leq (l_Q,k_Q)$ for all $p\in \mc R_Q$ and  $r\in \{\rcurvearrowdown,\rcurvearrowup,\lcurvearrowup,\lcurvearrowdown\}$. Consequently, $(L,K)(\{\,\cdot\,\})\leq (l_Q,k_Q)$ is stable under rotations.
  \par
  \textbf{Verticolor reflection:} Given $p\in \mc R_Q$,  Lemma~\ref{lemma:Z-verticolor-reflection} lets us infer that $(L,K)(\{\tilde p\})=(-L,-K)(\{\tilde p\})\leq (-l_Q,-k_Q)$. Because $-l_Q=l_Q$ and $-k_Q=k_Q$, we have thus proven that verticolor reflection preserves $(L,K)(\{\,\cdot\,\})\leq (l_Q,k_Q)$.
  \par
  \textbf{Tensor product:} Let $p,p'\in \mc R_Q$ be arbitrary. Then, in particular $\toco(p),\toco(p')\in k_Q$. Since $k_Q$ is a subgroup of $\integers$ we conclude $\gcd(\toco(p),\toco(p'))\integers\subseteq  k_Q$. Therefore, Lemma~\hyperref[lemma:Z-tensor-product-4]{\ref*{lemma:Z-tensor-product}~\ref*{lemma:Z-tensor-product-4}} implies $L(\{p\otimes p'\})\subseteq L(\{p,p'\})+\gcd(\toco(p),\toco(p'))\integers\subseteq l_Q+k_Q$ and $K(\{p\otimes p'\})\subseteq K(\{p,p'\})+\gcd(\toco(p),\toco(p'))\integers\subseteq k_Q+k_Q$. Because $l_Q+k_Q\subseteq l_Q$ and $k_Q+k_Q\subseteq k_Q$ by definition of $l_Q$ and $k_Q$, this proves $(L,K)(\{p\otimes p'\})\leq (l_Q,k_Q)$. In conclusion, $(L,K)(\{\,\cdot\,\})\leq (l_Q,k_Q)$ is respected by tensor products.
  \par
  \textbf{Erasing:} Lastly, we let $p\in \mc R_Q$ have a turn $T$ and show $(L,K)(\{E(p,T)\})\leq (l_Q,k_Q)$. Because we have already shown invariance under rotation, we can suppose that $p$ has no upper points and that $T$ comprises exactly the rightmost two lower points of $p$.
  Let $\alpha$ and $\beta$ be legs of the same block $B$ of $E(p,T)$ such that $\alpha\neq \beta$ and such that $]\alpha,\beta[_{E(p,T)}\cap B=\emptyset$. What we then have to prove is that $\delta_{E(p,T)}(\alpha,\beta)\in l_Q$ if $\sigma_{E(p,T)}(\{\alpha,\beta\})\neq 0$ and that $\delta_{E(p,T)}(\alpha,\beta)\in k_Q$ if $\sigma_{E(p,T)}(\{\alpha,\beta\})=0$. By Lemma~\hyperref[lemma:erasing-tricks-2]{\ref*{lemma:erasing-tricks}~\ref*{lemma:erasing-tricks-2}} and~\ref{lemma:erasing-tricks-3} it suffices to show
  \begin{align}
    \label{eq:categories-3-proof-1}
    \delta_p(\alpha,\beta)\overset{!}{\in}
    \begin{cases}
      l_Q&\text{if }\sigma_p(\{\alpha,\beta\})\neq 0,\\
      k_Q&\text{otherwise}.
    \end{cases}
  \end{align}
  To prove \eqref{eq:categories-3-proof-1} we now distinguish the two Cases~\ref{lemma:categories-3-1} and~\ref{lemma:categories-3-2}.
  \par
  \textbf{Case~1:}~\emph{Part~\ref{lemma:categories-3-1}.}
  Since $l_Q=k_Q=m\integers$, Assertion~\eqref{eq:categories-3-proof-1} simplifies to the claim
    \begin{align}
      \label{eq:categories-3-proof-2}
      \delta_p(\alpha,\beta)\overset{!}{\in} m\integers.
    \end{align}
    Because $Q$ is as required by Lemma~\ref{lemma:simplification-of-R}, we can immediately  deduce \eqref{eq:categories-3-proof-2} from $p\in \mc R_Q$ if $\alpha$ and $\beta$ belong to the same block in $p$. Thus, it remains to  assume that the blocks $B_1$ of $\alpha$ and $B_2$ of $\beta$ in $p$ are distinct and prove \eqref{eq:categories-3-proof-2} under this premise.
    \par
    \textbf{Step~1.1:}~\emph{Decomposing $\delta_p(\alpha,\beta)$.} Because $B_1\neq B_2$, the definition of $E(p,T)$ requires the existence of $\alpha'\in B_1\cap T$ and $\beta'\in B_2\cap T$.
Lemma~\hyperref[lemma:color-distance-3]{\ref*{lemma:color-distance}~\ref*{lemma:color-distance-3}} yields
    \begin{align*}
      \delta_p(\alpha,\beta) &\equiv \delta_p(\alpha,\alpha')+\delta_p(\alpha',\beta')+\delta_p(\beta',\beta) \mod m,
    \end{align*}
    where we have used the consequence $\toco(p)\in m\integers$ of $p\in \mc R_Q$. By this congruence, Assertion~\eqref{eq:categories-3-proof-2} can be seen by proving that all summands on the right hand side are multiples of $m$. Hence, that is what we now show.
    \par
    \textbf{Step~1.2:}~\emph{Color distances of $\alpha$ and $\alpha'$ and of $\beta$ and $\beta$'.} Because $\alpha,\alpha'\in B_1$ are legs of the same block of $p$ and because $p\in \mc R_Q$, Lemma~\ref{lemma:simplification-of-R} guarantees $\delta_p(\alpha,\alpha')\in m\integers$. Likewise, $\beta,\beta'\in B_2$ implies $\delta_p(\beta,\beta')\in m\integers$. 
    \par
    \textbf{Step~1.3:}~\emph{Color distance of $\alpha'$ and $\beta'$.} Because $B_1\neq B_2$, necessarily $\alpha'\neq \beta'$ and $T=\{\alpha',\beta'\}$. As $T$ is a turn, the neighboring points  $\alpha'$ and $\beta'$ have opposite normalized colors. We infer $\delta_p(\alpha',\beta')=0$ or $\delta_p(\alpha',\beta')=\toco(p)$, depending on whether $\alpha'$ is left of $\beta'$ or not. As $\toco(p)\in m\integers$ due to $p\in \mc R_Q$, it thus follows $\delta_p(\alpha',\beta')\in m\integers$ in any case. And that is what we needed to see.
    \par
    \textbf{Case~2:}~\emph{Part~\ref{lemma:categories-3-2}.}
    In this case, $l_Q=m\!+\!2m\integers$ and $k_Q=2m\integers$ mean that Assertion~\eqref{eq:categories-3-proof-1} can be expressed equivalently as
    \begin{align}
      \label{eq:categories-3-proof-3}
      \delta_p(\alpha,\beta)\overset{!}{\in} \frac{\sigma_p(\{\alpha,\beta\})}{2}m+2m\integers.
    \end{align}
    If $B$ is a block of $p$, then $p\in \mc R_Q$ implies \eqref{eq:categories-3-proof-3} immediately. Hence, we only need to prove  \eqref{eq:categories-3-proof-3} for the case that $B$ is not a block of $p$.  We again want to use Lemma~\hyperref[lemma:color-distance-3]{\ref*{lemma:color-distance}~\ref*{lemma:color-distance-3}}. Hence, we must find a suitable choice of points $\alpha'$ and $\beta'$ to decompose $\delta_p(\alpha,\beta)$ with.
    \par
    \textbf{Step~2.1:}~\emph{Finding points $\alpha'$ and $\beta'$.}
    As $B$ is not a block of $p$, there exist blocks $B_1$ and $B_2$ of $p$ with $B_1\cap T, B_2\cap T\neq \emptyset$, with $T\subseteq B_1\cup B_2$ and with $B=(B_1\cup B_2)\backslash T$. In particular,  $\alpha,\beta\in B_1\cup B_2$.
    \par
    The points $\alpha$ and $\beta$ belong to different blocks in $p$: If there existed  $i\in \{1,2\}$ with $\alpha,\beta\in B_i$, then it would follow $B_i=\{\alpha,\beta\}$ since $\alpha\neq\beta$ and $|B_i|\leq 2$ by $p\in \Cp_{\leq 2}$. As $\alpha,\beta\notin T$, it would have to hold $B_i=B_i\backslash T=\{\alpha,\beta\}$ and thus $2=|B_i|=|B_i\backslash T|+|B_i\cap T|=2+|B_i\cap T|$. Because we have assumed $B_i\cap T\neq \emptyset$, this would be a contradiction.
    \par
    By renaming $B_1$ and $B_2$ we can assume $\alpha\in B_1$ and $\beta\in B_2$. Since distinct blocks are disjoint, $B_1\cap B_2=\emptyset$. Because $\alpha,\beta\notin T$ and  $B_1\cap T,B_2\cap T\neq \emptyset$ and because $|B_1|,|B_2|\leq 2$, we find $\alpha'\in B_1\cap T$ and $\beta'\in B_2\cap T$ such that $B_1=\{\alpha,\alpha'\}$ and $B_2=\{\beta,\beta'\}$.
    \par
\textbf{Step~2.2:}~\emph{Relating the normalized colors of $\alpha$, $\alpha'$, $\beta$ and $\beta'$.}  Since  $B_1\neq B_2$, since $B=(B_1\cup B_2)\backslash T$ and since $\sigma_p(T)=0$,
    \begin{align}
      \label{eq:categories-3-proof-4}
      \sigma_p(\{\alpha,\beta\})=\sigma_p(B)&=\sigma_p(B_1)+\sigma_p(B_2)=\sigma_p(\{\alpha,\alpha'\})+\sigma_p(\{\beta',\beta\}),
    \end{align}
    as in the proof of Lemma~\ref{lemma:categories-2}.
    \par
\textbf{Step~2.3:}~\emph{Color distances of $\alpha$ and $\alpha'$ and of $\beta$ and $\beta'$.}    The assumption $p\in \mc R_Q$ furthermore guarantees
    \begin{align}
      \label{eq:categories-3-proof-5}
      \delta_p(\alpha,\alpha')\in \frac{\sigma_p(\{\alpha,\alpha'\})}{2}m\!+\!2m\integers\quad\text{and}\quad\delta_p(\beta',\beta)\in \frac{\sigma_p(\{\beta',\beta\})}{2}m\!+\! 2m\integers
    \end{align}
    because $\alpha,\alpha'\in B_\alpha$ and $\alpha\neq \alpha'$ on the one hand and $\beta',\beta\in B_\beta$ and $\beta'\neq \beta$ on the other hand.
    \par
\textbf{Step~2.4:}~\emph{Color distance of $\alpha'$ and $\beta'$.}    Because $T=\{\alpha',\beta'\}$ is a turn, $\delta_p(\alpha',\beta')=0$ or $\delta_p(\alpha',\beta')=\toco(p)$. From $p\in \mc R_Q$ follows $\toco(p)\in k_Q=2m\integers$ and thus
    \begin{align}
      \label{eq:categories-3-proof-6}
      \delta_p(\alpha',\beta')\in 2m\integers.
    \end{align}
    We now have all ingredients to prove  \eqref{eq:categories-3-proof-3}.
    \par
\textbf{Step~2.5:}~\emph{Synthesis.}     Lemma~\hyperref[lemma:color-distance-3]{\ref*{lemma:color-distance}~\ref*{lemma:color-distance-3}}, yields
    \begin{IEEEeqnarray*}{rCl}
      \delta_p(\alpha,\beta) &\equiv& \delta_p(\alpha,\alpha')+\delta_p(\alpha',\beta')+\delta_p(\beta',\beta) \mod 2m,\\
&\overset{\eqref{eq:categories-3-proof-6}}{\equiv}&\delta_p(\alpha,\alpha')+\delta_p(\beta',\beta) \mod 2m.\\
&\overset{\eqref{eq:categories-3-proof-5}}{\equiv}&\frac{\sigma_p(\{\alpha,\alpha'\})}{2}m+\frac{\sigma_p(\{\beta',\beta\})}{2}m \mod 2m\\
&\overset{\eqref{eq:categories-3-proof-4}}{\equiv}&\frac{\sigma_p(\{\alpha,\beta\})}{2}m \mod 2m.                                                           
    \end{IEEEeqnarray*}
    With the proof of \eqref{eq:categories-3-proof-3} thus complete, so is the proof overall.\qedhere
  \end{proof}

\subsection{Color Distances between Legs of Crossing Blocks}
The last family of elements of $\mathsf Q$ we treat exhibits non-trivial $X$-com\-ponents. 
\begin{remark}
  \label{remark:categories-4}
  If $Q=(\pint,\integers,m\integers,m\integers,m\integers,\integers\backslash E)$  for some $m\in \{0\}\cup \pint$ and some $E\subseteq \integers$ with $E=-E=E+m\integers$ (as in  Lemma~\ref{lemma:categories-4} below), we can employ Lemma~\ref{lemma:simplification-of-R} to understand the set $\mc R_Q$. It is given by \emph{the set of all partitions such that the color distance between any two legs of the same block is a multiple of $m$ and such that two points belong to non-crossing blocks whenever the color distance between them is an element of $E$.}
  \par
  While the $F$- and $V$-components of $Q$, of course, effectively induce no conditions at all via $Z(\{\,\cdot\,\})\leq Q$, the $\toco$-, $L$- and $K$-conditions are non-trivial and they are \emph{not} implied by the $X$-constraint. 
\end{remark}

\begin{lemma}
  \label{lemma:categories-4}
  For every $m\in \{0\}\cup \pint$ and  $E\subseteq \integers$ with $E=-E=E+m\integers$ the set $\mc R_Q$ is a category if \[Q=(\pint,\integers,m\integers,m\integers,m\integers,\integers\backslash E).\]
\end{lemma}

\begin{proof}
  Let $Q$ be of the kind described in the claim. Once more, we use the Alternative Characterization (Lemma~\ref{lemma:characterization-categories}) to show that $\mc R_Q$ is a category. By Lem\-ma~\hyperref[lemma:categories-3-1]{\ref*{lemma:categories-3}~\ref*{lemma:categories-3-1}} it suffices to consider the $X$-component of $Z$, i.e.\ to prove that the constraint $X(\{\,\cdot\,\})\subseteq \integers\backslash E$ is invariant under rotation, tensor products, verticolor reflection and erasing turns.
  \par
  \textbf{Rotation:} With the help of Lemma~\ref{lemma:Z-rotation} we can infer $X(\{p^r\})=X(\{p\})\subseteq \integers\backslash E$ for all $p\in \mc R_Q$ and all $r\in \{\rcurvearrowdown,\rcurvearrowup,\lcurvearrowup,\lcurvearrowdown\}$. Thus, $X(\{\,\cdot\,\})\subseteq \integers\backslash E$ is preserved by rotations.
  \par
  \textbf{Verticolor reflection:} For all $p\in \mc R_Q$ holds $X(\{\tilde p\})=-X(\{p\})\subseteq -\integers\backslash E$ by Lemma~\ref{lemma:Z-verticolor-reflection}. Since our assumption $E=-E$ implies $-\integers\backslash E=\integers\backslash E$, this proves the condition $X(\{\,\cdot\,\})\subseteq \integers\backslash E$ stable under verticolor reflection.
  \par
  \textbf{Tensor product:} Let $p,p'\in \mc R_Q$ be arbitrary. Then, $\toco(p),\toco(p')\in m\integers$ implies $\gcd(\toco(p),\toco(p))\integers\subseteq m\integers$. Consequently, Lemma~\hyperref[lemma:Z-tensor-product-4]{\ref*{lemma:Z-tensor-product}~\ref*{lemma:Z-tensor-product-4}} yields $X(\{p\otimes p'\})\subseteq X(\{p,p'\})+\gcd(\toco(p),\toco(p'))\integers\subseteq \integers\backslash E+m\integers$. As we have assumed $E=E+m\integers$, also $\integers\backslash E=\integers\backslash E+m\integers$ and thus $X(\{p\otimes p'\}) \subseteq \integers\backslash E$.  In conclusion, tensor products respect  $X(\{\,\cdot\,\})\subseteq \integers\backslash E$.
  \par
  \textbf{Erasing:} Let $p\in \mc R_Q$ have a turn $T$. We show $X(\{E(p,T)\})\subseteq \integers\backslash E$. Since $X(\{\,\cdot\,\})\subseteq \integers\backslash E$ is preserved by rotations, no generality is lost assuming that $p$ has no upper points and that the two rightmost lower points of $p$ make up $T$.
  Let $\alpha_1$ and $\alpha_2$ be points in $E(p,T)$ whose blocks $B_1$ and $B_2$ cross in $E(p,T)$. By Lemma~\hyperref[lemma:erasing-tricks-3]{\ref*{lemma:erasing-tricks}~\ref*{lemma:erasing-tricks-3}} all we have to show is
  \begin{align}
    \label{eq:categories-4-proof-1}
    \delta_p(\alpha_1,\alpha_2)\notin E.
  \end{align}
  If $B_1$ and $B_2$ are both blocks of $p$ as well, then \eqref{eq:categories-4-proof-1} is true by the assumption $p\in \mc R_Q$. Thus, we only need to show \eqref{eq:categories-4-proof-1} in the opposite case. If so, then, by nature of the erasing operation, exactly one of the blocks $B_1$ and $B_2$ is also a block of $p$.
  \par
  \textbf{Step~1:}~\emph{Reduction to the case that $B_1$ is not a block of $p$.} The assumption $p\in \mc R_Q$ ensures $\toco(p)\in m\integers$. Hence,  Lemma~\hyperref[lemma:color-distance-2]{\ref*{lemma:color-distance}~\ref*{lemma:color-distance-2}} implies  $\delta_{p}(\alpha_1,\alpha_2)\equiv-\delta_{p}(\alpha_2,\alpha_1)\mod m$.  Because we assume $\integers\backslash E=-(\integers\backslash E)+m\integers$, that means $\delta_{p}(\alpha_1,\alpha_2)\in \integers\backslash E$ if and only if $\delta_{p}(\alpha_2,\alpha_1)\in\integers\backslash E$. Hence, it suffices to prove \eqref{eq:categories-4-proof-1} for the case that $B_1$ is not a block of $p$.
  \par
  \textbf{Step~2:}~\emph{Expressing $\delta_p(\alpha_1,\alpha_2)$ as the color distance between some blocks $B_{1,i}$ and $B_2$ of $p$.}    Since $B_1$ is not a block of $p$, there are (not necessarily distinct) blocks $B_{1,1}$ and $B_{1,2}$ of $p$ with $B_{1,1}\cap T,B_{1,2}\cap T\neq \emptyset$ with $T\subseteq B_{1,1}\cup B_{1,2}$ and with $B_1=(B_{1,1}\cup B_{1,2})\backslash T$. By renaming if necessary, we can always achieve $\alpha_1\in B_{1,1}$.  We show that, for all $\beta_1\in B_{1,1}\cup B_{1,2}$ and all $\beta_2\in B_2$,
  \begin{align}
    \label{eq:categories-4-proof-2}
    \delta_p(\alpha_1,\alpha_2)\equiv \delta_p(\beta_1,\beta_2)\mod m.
  \end{align}
  Two cases must be distinguished.
  \par
  \textbf{Case~2.1:}~\emph{$\beta_1$ and $\alpha_1$ belong to the same block.} Let $\beta_1\in B_{1,1}$.
  By Lemma~\hyperref[lemma:color-distance-3]{\ref*{lemma:color-distance}~\ref*{lemma:color-distance-3}},
  \begin{align}
    \label{eq:categories-4-proof-3}
    \delta_p(\alpha_1,\alpha_2)\equiv \delta_p(\alpha_1,\beta_1)+\delta_p(\beta_1,\beta_2)+\delta_p(\beta_2,\alpha_2)\mod m.
  \end{align}
  According to Lemma~\ref{lemma:simplification-of-R} the assumption
   $p\in \mc R_Q$ implies $\delta_p(\alpha_1,\beta_1),\delta_p(\beta_2,\alpha_2)\in m\integers$ since $\alpha_1,\beta_1\in B_{1,1}$ and $\beta_2,\alpha_2\in B_2$. Hence, the only term on the right side of \eqref{eq:categories-4-proof-3} possibly surviving  is $\delta_p(\beta_1,\beta_2)$. That proves \eqref{eq:categories-4-proof-2} in this case.
  \par
  \textbf{Case~2.2:}~\emph{$\beta_1$ and $\alpha_1$ belong to different blocks.} Now, suppose $\beta_1\in B_{1,2}$ instead.  Since we have assumed $B_{1,1}\cap T,B_{1,2}\cap T\neq \emptyset$ we can infer the existence of $\gamma_{1,1}\in B_{1,1}$ and $\gamma_{1,2}\in B_{1,2}$ with $\gamma_{1,1}\neq \gamma_{1,2}$ and $T=\{\gamma_{1,1},\gamma_{1,2}\}$. Again, we use Lemma~\hyperref[lemma:color-distance-3]{\ref*{lemma:color-distance}~\ref*{lemma:color-distance-3}} to derive
  \begin{IEEEeqnarray}{rCl}
    \label{eq:categories-4-proof-4}
    \delta_p(\alpha_1,\alpha_2)&\equiv& \delta_p(\alpha_1,\gamma_{1,1})+\delta_p(\gamma_{1,1},\gamma_{1,2})+\delta_p(\gamma_{1,2},\beta_1)\nonumber\\
    &&+\delta_p(\beta_1,\beta_2)+\delta_p(\beta_2,\alpha_2) \mod m.
  \end{IEEEeqnarray}
  All summands on the right hand side of \eqref{eq:categories-4-proof-4} except for possibly $\delta_p(\beta_1,\beta_2)$ are multiples of $m$:
  \par
  Because $T$ is a turn, $\gamma_{1,1}$ and $\gamma_{1,2}$ are neighbors of different normalized colors, meaning $\delta_p(\gamma_{1,1},\gamma_{1,2})=0$ or $\delta_p(\gamma_{1,1},\gamma_{1,2})=\toco(p)\in m\integers$ and thus $\delta_p(\gamma_{1,1},\gamma_{1,2})\in m\integers$ in any case.
  \par
  And, thanks to $p\in \mc R_Q$, from $\alpha_1,\gamma_{1,1}\in B_{1,1}$, from $\gamma_{1,2},\beta_1\in B_{1,2}$ and from $\beta_2,\alpha_2\in B_2$ follow $\delta_p(\alpha_1,\gamma_{1,1}),\delta_p(\gamma_{1,2},\beta_1),\delta_p(\beta_2,\alpha_2)\in m\integers$ according to Lemma~\ref{lemma:simplification-of-R}.
  \par
Thus  \eqref{eq:categories-4-proof-4} verifies \eqref{eq:categories-4-proof-2} in this case.
  \par
  \textbf{Step~3:}~\emph{Showing that $B_{1,i}$ and $B_2$ cross in $p$.}
  If we can establish that $B_{1,1}$ and $B_2$ or $B_{1,2}$ and $B_2$ cross each other in $p$, then Equation~\eqref{eq:categories-4-proof-2} proves \eqref{eq:categories-4-proof-1} as $\integers\backslash E=(\integers\backslash E)+m\integers$. So, this is all we have to show.
  \par

  Because $B_1$ and $B_2$ cross in $E(p,T)$ we find points $\epsilon_{1,1},\epsilon_{1,2}\in B_1$ and $\eta_2,\theta_2\in B_2$ such that $(\epsilon_{1,1},\eta_2,\epsilon_{1,2},\theta_2)$ is ordered in $E(p,T)$ (and thus also in $p$). If $\epsilon_{1,1},\epsilon_{1,2}\in B_{1,1}\cup B_{1,2}$ both belong to $B_{1,1}$ or both to $B_{1,2}$, then we have already found the desired crossing with $B_2$. Hence, we can assume that neither is the case, i.e., that $B_{1,1}\neq B_{1,2}$ and that $\epsilon_{1,1}$ and $\epsilon_{1,2}$ belong to different blocks in $p$. If $\epsilon_{1,1}\notin B_{1,1}$, then we rename $\epsilon_{1,1}\leftrightarrow\epsilon_{1,2}$ and $\eta_2\leftrightarrow\theta_2$. Thus, we can achieve that  $\epsilon_{1,1}\in B_{1,1}$ and $\epsilon_{1,2}\in B_{1,2}$ while maintaining  $\eta_2,\theta_2 \in B_2$ and while keeping $(\epsilon_{1,1},\eta_2,\epsilon_{1,2},\theta_2)$ ordered.
  \par
Let  $\gamma_{1,1}$ and $\gamma_{1,2}$ be as before, i.e., $\gamma_{1,1}\in B_{1,1}\cap T$ and $\gamma_{1,2}\in B_{1,2}\cap T$. Since they are neighbors, they have the same position in the cyclic order of $p$ with respect to the tuple $(\epsilon_{1,1},\eta_2,\epsilon_{1,2},\theta_2)$. Let $i\in \{1,2\}$ be such that $\gamma_{1,i}$ lies left of $\gamma_{1,\neg i}$ (in the native ordering). Then, the below table shows that, no matter how the points are arranged, we can find a crossing between $B_{1,1}$ and $B_2$ or between $B_{1,2}$ and $B_2$:
  \begin{gather*}
    \begin{matrix}
      \text{ordered tuple} & \text{crossing}& \text{between}\\
 (\epsilon_{1,1},\eta_2,\epsilon_{1,2},\theta_2,\underline{\gamma_{1,i},\gamma_{1,\neg i}}) &  (\epsilon_{1,1},\eta_2,\gamma_{1,1},\theta_2) & B_{1,1}\text{ and }B_2 \\
(\epsilon_{1,1},\eta_2,\epsilon_{1,2},\underline{\gamma_{1,i},\gamma_{1,\neg i}},\theta_2) &  (\epsilon_{1,1},\eta_2,\gamma_{1,1},\theta_2) & B_{1,1}\text{ and }B_2\\
(\epsilon_{1,1},\eta_2,\underline{\gamma_{1,i},\gamma_{1,\neg i}},\epsilon_{1,2},\theta_2) & (\epsilon_{1,1},\eta_2,\gamma_{1,1},\theta_2)& B_{1,1}\text{ and }B_2\\
(\epsilon_{1,1},\underline{\gamma_{1,i},\gamma_{1,\neg i}},\eta_2,\epsilon_{1,2},\theta_2) & (\gamma_{1,2},\eta_2,\epsilon_{1,2},\theta_2)& B_{1,2}\text{ and }B_2
    \end{matrix}
  \end{gather*}
  So, in combination with what we showed in Step~2, that completes the proof.
\end{proof}
\subsection{Synthesis}
As announced we can combine the previous results to prove that all sets defined in Section~\ref{section:definition-of-the-categories} are categories. Recall Notation~\ref{integer-notations}.
\begin{theorem}
  \label{theorem:main}
  For every $Q\in \mathsf Q$ the set $\mc R_Q$ is a non-hyperoctahedral category: 
   \begin{align*}
     \mc R: \, \mathsf Q\to \nhoc.
  \end{align*}
  \end{theorem}
\begin{proof}
For every family $\mathsf Q'\subseteq \mathsf Q$ holds by Lemma~\ref{lemma:R-monotonic-meet-preserving} 
  \begin{align*}
    \bigcap \{\mc R_Q\mid Q\in \mathsf Q'\}=\mc R_{\bigcap _\times \mathsf Q'}.
  \end{align*}
  We show that for every $Q\in \mathsf Q$ there exists a set  $\mathsf Q'\subseteq \mathsf Q$ such that $Q={\bigcap}_\times \mathsf Q'$ and such that $\mc R_{Q'}$ is a category for every $Q'\in \mathsf Q'$. As $\cotcp$ is a complete lattice, that then proves that $\mc R_Q$ is a category of two-colored partitions for every $Q\in \mathsf Q$. It is straightforward to check that indeed $\PartSinglesWBTensor\in \mc R_Q$ or $\PartFourWBWB\notin \mc R_Q$ for every $Q\in \mathsf Q$.
  \par
For this proof only, we use names for specific elements of the set $\mathsf Q$: The below table applies for all $m\in \{0\}\cup \pint$ and $E\subseteq \integers$ with $E=-E=E+m\integers$. 

\begin{align*}
  \begin{array}{c| c c c c c c |c}
    &F&V&\toco&L&K& X &\\[0.1em]\hline
  \mr F_2& \{2\} &\pm \{0,2\}& 2\integers &\integers &\integers& \integers & {\ref*{lemma:categories-2}~\ref*{lemma:categories-2-2}}\Tstrut\\
  \mr F_{\leq 2}& \{1,2\}&\pm \{0,1,2\}&\integers&\integers&\integers&\integers & {\ref*{lemma:categories-2}~\ref*{lemma:categories-2-1}}\\
  \mr V_0& \{2\}&\{0\}&\{0\}&\emptyset&\integers&\integers & {\ref*{lemma:categories-2}~\ref*{lemma:categories-2-3}}\\
  \mr V_{01}& \{1,2\}&\pm\{0,1\}&\integers&\emptyset&\integers&\integers & {\ref*{lemma:categories-2}~\ref*{lemma:categories-2-4}}\\  
  \mr S_m& \pint&\integers&m\integers&\integers&\integers&\integers & \ref*{lemma:categories-1}\\
  \mr K_m& \pint&\integers&m\integers&m\integers&m\integers&\integers & {\ref*{lemma:categories-3}~\ref*{lemma:categories-3-1}}\\
  \mr K_m^\upY& \{1,2\}&\pm \{0,1,2\}&2m\integers &m\!+\!2m\integers &2m\integers&\integers & {\ref*{lemma:categories-3}~\ref*{lemma:categories-3-2}}\\
  \mr X_{m,E}& \pint&\integers&m\integers&m\integers&m\integers&\integers\backslash E& \ref*{lemma:categories-4}
  \end{array}
\end{align*}
The corresponding sets $\mc R_Q$ for all these elements $Q\in \mathsf Q$ are categories of partitions, as shown by the respective lemma cited in the last column.
\par

The ensuing table lists how every element of $\mathsf Q$ can be written as a meet in $\mathsf L$ of the specific elements defined above. Here, $u\in \{0\}\cup \pint$, $m\in \pint$, $D\subseteq \{0\}\cup\dwi{\lfloor\frac{m}{2}\rfloor}$, $E\subseteq \{0\}\cup\pint$ and $N$ is a subsemigroup of $(\pint,+)$.
                    \begin{align*}
                      \begin{array}{c c c c c c | c}
                        F&V&S&L&K& X &\bigcap_\times  \\ \hline \\[-0.85em]
                        \{2\} & \pm\{0, 2\} & 2um\integers & m\integers & m\integers & \integers & \mr F_2, \mr S_{2um}, \mr K_m\\
                      \{2\} & \pm\{0, 2\} & 2um\integers & m\hspace{-2.5pt}+\hspace{-2.5pt}2m\integers & 2m\integers & \integers & \mr F_2, \mr S_{2um}, \mr K_m^\upY\\
                        \{2\} & \pm \{0, 2\} & 2um\integers & m\hspace{-2.5pt}+\hspace{-2.5pt}2m\integers & 2m\integers & \integers\backslash m\integers & \mr F_2, \mr S_{2um}, \mr K_m^\upY, X_{m,m\integers}\\
                      \{2\} & \{0\} & \{0\} & \emptyset & m\integers & \integers & \mr V_0, \mr K_m\\
                      \{2\} & \pm\{0, 2\} & \{0\} & \{0\} & \{0\} &  \integers\backslash N_0 & \mr F_2, \mr X_{0,N_0}\\
                      \{2\} & \{0\} & \{0\} & \emptyset & \{0\} & \integers\backslash N_0& \mr V_0, \mr X_{0,N_0}\\
                      \{2\} & \{0\} & \{0\} & \emptyset & \{0\} &\integers\backslash N_0'& \mr V_0, \mr X_{0,N_0'}\\
                         \{1,2\}&\pm\{0, 1, 2\} & um\integers & m\integers & m\integers & \integers\backslash D_m& \mr F_{\leq 2}, \mr S_{um}, \mr X_{m,D_m}\\
                                                 \{1,2\}&\pm\{0, 1, 2\} & 2um\integers & m\hspace{-2.5pt}+\hspace{-2.5pt}2m\integers & 2m\integers & \integers\backslash D_m&  \mr S_{2um}, \mr K_m^\upY, \mr X_{m,D_m} \\
                         \{1,2\}&\pm \{0, 1\} & um\integers & \emptyset & m\integers & \integers\backslash D_m& \mr V_{01}, \mr S_{um}, \mr X_{m,D_m}\\
                        \{1,2\}&\pm\{0, 1, 2\} & \{0\} & \{0\} & \{0\} & \integers\backslash E_0& \mr F_{\leq 2}, \mr X_{0,E_0}\\
                      \{1,2\}&\pm\{0,  1\} & \{0\} & \emptyset & \{0\} & \integers\backslash E_0& \mr V_{01}, \mr X_{0,E_0}\\
                        \pint & \integers & um\integers & m\integers & m\integers & \integers\backslash D_m&  \mr S_{um}, \mr X_{m,D_m}\\
                        \pint & \integers & \{0\} & \{0\} & \{0\} & \integers\backslash E_0&\mr X_{0,E_0}
                    \end{array}
                    \end{align*}
                  That concludes the proof.
\end{proof}

\section{Further Categories}
Although we are only interested in non-hyperoctahedral categories in this article, it deserves pointing out that the proof Lemma~\ref{lemma:categories-2} also shows the existence of certain hyperoctahedral categories.
\begin{lemma}
  \label{lemma:block-categories-H}
The set $\{ p\in \Cp\mid V(\{p\})\subseteq g\integers\}$ is a category for every $g\in\{0\}\cup \pint$.
\end{lemma}
\begin{proof}
  We use the Alternative Characterization (Lemma~\ref{lemma:characterization-categories}) of categories. Lemmata~\ref{lemma:Z-rotation}, \ref{lemma:Z-verticolor-reflection} and~\hyperref[lemma:Z-tensor-product-2]{\ref*{lemma:Z-tensor-product}~\ref*{lemma:Z-tensor-product-2}} imply that $V(\{\,\cdot\,\})\subseteq g\integers$ is stable under rotations, verticolor reflection and tensor products. While verifying Lemma~\ref{lemma:categories-2} we showed that for every  $p\in \Cp$, every turn $T$ in $p$ and every block $B$ of $E(p,T)$ the following is true: Either $B$ is a block of $p$ and then $\sigma_{E(p,T)}(B)=\sigma_p(B)$ or $B$ is not a block of $p$ and then there are blocks $B_1$ and $B_2$ of $p$ such that $\sigma_{E(p,T)}(B)=\sigma_p(B_1)+\sigma_p(B_2)$ or $\sigma_{E(p,T)}(B)=\sigma_p(B_1)$. Thus,  $V(\{E(p,T)\})\subseteq V(\{p\})\cup (V(\{p\})+V(\{p\}))$, which proves the claim.
\end{proof}

\printbibliography

\end{document}